\documentclass[a4page]{article}
\usepackage[a4paper,
            bindingoffset=0.2in,
            left=1in,
            right=1in,
            top=1in,
            bottom=1in,
            footskip=.25in]{geometry}

\usepackage{amsmath}
\usepackage{graphicx}
\usepackage{xspace}
\usepackage{enumerate}
\usepackage[utf8]{inputenc}
\usepackage[square]{natbib}
\usepackage{color}
\usepackage[dvipsnames]{xcolor}
\definecolor{myblue}{rgb}{0, 0, 0.66}
\usepackage{url}
\usepackage{subfig}
\usepackage[colorlinks=true,linkcolor=blue,citecolor=myblue,urlcolor=cyan]{hyperref}
\usepackage{breakurl}
\usepackage{breakcites}

\usepackage{amsfonts}
\usepackage{amsthm}
\usepackage{bbm}
\usepackage{MnSymbol}
\usepackage{pifont}
\usepackage{cases}
\usepackage{mdframed}
\usepackage{stmaryrd}
\usepackage{mathpartir}
\usepackage{listings}
\usepackage{tikz}
\usepackage{tikz-cd}
\usetikzlibrary{decorations.pathmorphing}
\usepackage{multirow}
\usepackage{hhline}

\newtheorem{lemma}{Lemma}
\newtheorem{definition}[lemma]{Definition}
\newtheorem{proposition}[lemma]{Proposition}

\newtheorem{therm}[lemma]{Theorem}
\newtheorem{corollary}[lemma]{Corollary}

\newcommand{\systemname}[1]{\texttt{\color{darkgray}#1}\xspace}

\newcommand{\Agda}{\systemname{Agda}}
\newcommand{\HoTTAgda}{\systemname{HoTT}-\Agda}

\newcommand{\Coq}{\systemname{Coq}}
\newcommand{\CubicalAgda}{\systemname{Cubical} \systemname{Agda}}
\newcommand{\Lean}{\systemname{Lean}}
\newcommand{\mathlib}{\systemname{mathlib}}
\newcommand{\bZ}{\mathbb{Z}}
\newcommand{\bZTwo}{\bZ/2\bZ}
\newcommand{\cbase}{\mathsf{base}}

\newcommand{\CP}{\ensuremath{\mathbb{C}P^2}}
\newcommand{\RP}{\ensuremath{\mathbb{R}P^2}}
\newcommand{\RPinf}{\ensuremath{\mathbb{R}P^{\infty}}}
\newcommand{\KTwoOne}{\ensuremath{K(\bZ/2\bZ,1)}}
\newcommand{\cmark}{\ding{51}}%
\newcommand{\xmark}{\ding{55}}%

\newcommand{\refl}{\mathsf{refl}}
\newcommand{\killDiag}{\ensuremath{\mathsf{kill-}\Delta}}

\newcommand{\inrop}{\mathsf{inr}}
\newcommand{\inlop}{\mathsf{inl}}
\newcommand{\pushop}{\mathsf{push}}
\newcommand{\inr}[1]{\inrop\,{#1}}
\newcommand{\inl}[1]{\inlop\,{#1}}
\newcommand{\push}[1]{\pushop\,{#1}}
\newcommand{\ap}[2]{\mathsf{ap}_{#1}(#2)}
\newcommand{\apfunct}[2]{\mathsf{ap}_F^2\mathsf{\textsf{-}funct}(#1,#2)}
\newcommand{\app}[3]{\mathsf{ap}^{2}_{#1}(#2,#3)}
\newcommand{\Loop}{\mathsf{loop}}
\newcommand{\LoopG}[1]{\ensuremath{\mathsf{loop}_k\,{#1}}}
\newcommand{\trunc}[1]{\lvert\,#1\,\rvert}
\newcommand{\ntype}[1]{\normalfont{\textsf{${#1}$-Type}}}
\newcommand{\truncT}[1]{\lVert#1\rVert}
\newcommand{\fib}[2]{\mathsf{fib}_{#1}(#2)}
\newcommand{\Square}[4]{\ensuremath{#1 =_{#4}^{#3} #2}}
\newcommand{\KO}{\ensuremath{\mathsf{0}_k}}
\newcommand{\cfcod}{\normalfont{\ensuremath{\mathsf{cfcod}}}}
\newcommand{\timeswedge}[2]{\ensuremath{{#1}\tilde{\vee}{#2}}}
\newcommand{\cofib}[1]{\ensuremath{\mathsf{cofib}\,#1}}
\newcommand{\im}[1]{\mathsf{im}(#1)}
\newcommand{\transport}[3]{\ensuremath{\mathsf{transport}^{#1}(#2,#3)}}
\newcommand{\north}{\mathsf{north}}
\newcommand{\south}{\mathsf{south}}
\newcommand{\merid}[1]{\mathsf{merid}\,#1}
\newcommand{\plustimes}[2]{\ensuremath{{#1} \oplus {#2}}}

\definecolor{myred}{rgb}{.94, 0.05, 0.05}
\definecolor{mygreen}{rgb}{0 , 0.55, 0}
\newcommand{\cmarkgreen}{{\color{mygreen}{\cmark}}}
\newcommand{\xmarkred}{{\color{myred}{\xmark}}}

\definecolor{dkblue}{rgb}{0,0.1,0.5}
\definecolor{lightblue}{rgb}{0,0.5,0.5}
\definecolor{dkgreen}{rgb}{0,0.6,0}
\definecolor{dkbrown}{rgb}{0.4,0,0}
\definecolor{dkviolet}{rgb}{0.6,0,0.8}

\begin{document}

\title{Computational Synthetic Cohomology Theory in Homotopy Type Theory}

\author{
  Axel Ljungström\\
  \footnotesize{Stockholm University, Sweden}\\
  \footnotesize{\texttt{axel.ljungstrom@math.su.se}}
  \and
  Anders Mörtberg\\
  \footnotesize{Stockholm University, Sweden}\\
  \footnotesize{\texttt{anders.mortberg@math.su.se}}
}
\date{}


\maketitle



\paragraph*{Abstract:}
  This paper discusses the development of synthetic cohomology in
  Homotopy Type Theory (HoTT), as well as its computer formalisation. The
  objectives of this paper are (1) to generalise previous work on integral
  cohomology in HoTT by the current authors and Brunerie~(\citeyear{BLM})
  to cohomology with arbitrary coefficients and (2) to provide the mathematical details of, as well as extend, results underpinning the computer formalisation of cohomology rings by the current authors and Lamiaux~(\citeyear{CohomRings23}).
  %
  %
  With respect to objective~(1), we provide new direct definitions
  of the cohomology group operations and of the cup product, which,
  just as in~\citep{BLM}, enable significant simplifications of many earlier proofs in
  synthetic cohomology theory. In particular,
  the new definition of the cup product allows us to give the first
  complete formalisation of the axioms needed to turn the cohomology
  groups into a graded commutative ring. We also establish that this
  cohomology theory satisfies the HoTT formulation of the
  Eilenberg-Steenrod axioms for cohomology and study the classical
  Mayer-Vietoris and Gysin sequences. With respect to objective (2), we characterise the
  cohomology groups and rings of various spaces, including the
  spheres, torus, Klein bottle, real/complex projective planes, and infinite real projective space.
  All results have been formalised in
  \CubicalAgda and we obtain multiple new numbers, similar to the
  famous `Brunerie number', which can be used as benchmarks for
  computational implementations of HoTT. Some of these numbers are
  infeasible to compute in \CubicalAgda and hence provide new
  computational challenges and open problems which are much easier to
  define than the original Brunerie number.

\paragraph*{Acknowledgements:}
The authors are grateful to Guillaume Brunerie and Thomas Lamiaux for
collaborating with us and co-authoring earlier related work in
\citep{BLM} and \citep{CohomRings23}. We are also grateful to Evan
Cavallo for insightful comments and many cubical tricks. The first author would also like to thank David Wärn for many fruitful discussions on related work which inspired the proof of Proposition~\ref{prop:RPinf-cohom}. Finally, we wish to thank Dan Christensen for his helpful comments on an older version of this paper.

This paper is based upon research supported by the Swedish Research
Council (Vetenskapsrådet) under Grant No.~2019-04545. The research has
also received funding from the Knut and Alice Wallenberg Foundation
through the Foundation's program for mathematics.

\section{Introduction}
\label{sec:intro}

A fundamental idea in algebraic topology is that spaces can be
analysed in terms of homotopy invariants---functorial assignments of
algebraic objects to spaces. The close correspondence between types
and spaces in Homotopy Type Theory and Univalent Foundations (HoTT/UF)
allows these concepts to be developed \emph{synthetically} using type
theory. This was explored in the HoTT Book \citep{HoTT13}, written
during the IAS special year on HoTT/UF in 2012--2013, and has since
led to the formalisation of many results from homotopy theory in
HoTT/UF. Using these results, homotopy groups of
many spaces---represented as types---have been characterised. However, just
like in classical algebraic topology, these groups tend to be
complicated to compute. Because of this, other topological invariants
that are easier to compute, like cohomology, have also been developed
synthetically in HoTT/UF.

Intuitively, the cohomology groups $H^n(X,G)$ of a space $X$, relative
to an abelian group $G$, characterise the connected components of $X$
as well as its $(n+1)$-dimensional holes. Figure \ref{fig:circles}
depicts the circle, $\mathbb{S}^1$, and two circles which have been
glued together in a point, i.e.\ the \emph{wedge sum}
$\mathbb{S}^1 \vee \mathbb{S}^1$. The fact that these spaces have a
different number of holes is captured using, e.g.\ singular
cohomology, by the existence of isomorphisms
$H^1(\mathbb{S}^1,G) \cong G$ and
$H^1(\mathbb{S}^1 \vee \mathbb{S}^1,G) \cong G \times G$, which
geometrically means that they have one, respectively two,
$2$-dimensional holes (i.e.\ empty interiors). As cohomology groups
are homotopy invariants, this then means that the spaces cannot be
continuously deformed into each other.

\begin{figure}[h!]
  \centering
  \begin{mathpar}
  \begin{tikzpicture}[scale=0.5]
    \draw[color=blue!90](0.5,0) circle (1.5);
  \end{tikzpicture}
  \and
  \begin{tikzpicture}[scale=0.5]
    \draw[color=blue!90](5,0) circle (1.5);
    \draw[color=blue!90](8,0) circle (1.5);
    \filldraw[color=blue!90, very thick](6.5,0) circle (0.05);
  \end{tikzpicture}
  \end{mathpar}
  \caption{$\mathbb{S}^1$ and $\mathbb{S}^1 \vee \mathbb{S}^1$.}
  \label{fig:circles}
\end{figure}

The usual formulation of
singular cohomology using cochain complexes relies on taking the
underlying set of topological spaces when defining the singular
cochains \citep{Hatcher2002}. This operation is \emph{not} invariant
under homotopy equivalence, which makes it impossible to use when
formalising cohomology synthetically. For the construction of
cohomology in HoTT/UF, we instead rely on Brown
representability~\citep{BrownRepr} and define cohomology groups as homotopy classes of
maps into Eilenberg-MacLane spaces (as defined in HoTT by
\citet{LicataFinster14}). This approach to
synthetic cohomology theory was initially studied at the IAS special
year \citep{ShulmanBlog13} and is classically provably equivalent to
the singular cohomology of the spaces we consider in this paper. It has since been expanded with many
classical results, for example the Eilenberg-Steenrod axioms for
cohomology~\citep{CavalloMsc15}, cellular
cohomology~\citep{BuchholtzFavonia18}, the Atiyah-Hirzebruch and Serre
spectral sequences \citep{FlorisPhd}, and it plays a key role in
the synthetic proof that $\pi_4(\mathbb{S}^3) \cong \bZ/2\bZ$ by
\citet{Brunerie16}.

This paper develops the theory of synthetic cohomology from a
computational perspective with the aim of characterising cohomology
groups and rings of various spaces. This is achieved by extending
earlier work by \citet{BLM} where integral cohomology $H^n(X,\bZ)$ was
developed in \CubicalAgda. \citet{BLM} also computed $\bZ$-cohomology
groups of various classical spaces synthetically, including the
spheres, torus, real/complex projective planes, and Klein bottle. This
was aided by a new synthetic construction of the group
structure on $H^n(X,\bZ)$ with better computational properties than
the one already defined in HoTT by \citet{LicataFinster14}.
%
%
\citet{BLM} also redefined the cup product
$\smile \,: H^n(X,\bZ) \times H^m(X,\bZ) \to H^{n+m}(X,\bZ)$ which led
to substantially simplified proofs and better computational properties
compared to the one originally defined synthetically by
\citet{Brunerie16}. \citet{CohomRings23} organised this into a ring
$H^*(X,R)$, for a general ring $R$, and various such cohomology rings
were computed for different spaces. The authors have also used some of
these results, and extensions thereof, in a complete \CubicalAgda
formalisation of Brunerie's synthetic proof that
$\pi_4(\mathbb{S}^3) \simeq \bZ/2\bZ$ \citep{LICS23}. However, due to
page constraints many proofs and constructions were omitted from
\citep{BLM}, \citep{CohomRings23}, and \citep{LICS23}. The aim of the
current paper is to spell these details out, develop new results, and
generalise various results already present in the previous papers and
the general HoTT literature. Hence, our aim is to provide a more comprehensive
synthetic treatment of cohomology in HoTT/UF, with a view towards
effective computations. The main contents and contributions of this paper include:

\begin{itemize}
\item We spell out the details of how to generalise the
  computationally well-behaved integral cohomology theory of
  \citet{BLM} to general cohomology groups $H^n(X,G)$. We first do
  this for Eilenberg-MacLane spaces in Section
  \ref{sec:emgroupstructure}, which then induce the corresponding
  operations on cohomology in Section \ref{sec:cohomology}.
\item In Section \ref{sec:omegakgn} we prove that $K(G,n) \simeq
  \Omega K(G,n+1)$ by a direct encode-decode proof, avoiding the
  use of the Freudenthal suspension theorem.
\item In Sections \ref{sec:cupprodgroupcoeff} and
  \ref{sec:cupprodringcoeff} we spell out the details of how to
  generalise the theory of integral cohomology rings of
  \citet{Brunerie16} to general cohomology rings $H^*(X,R)$ as in
  \citep{CohomRings23}. In particular, this involves a general form of
  the cup product
  $\smile \,: H^n(X,G_1) \times H^m(X,G_2) \to H^{n+m}(X,G_1 \otimes
  G_2)$, inspired by a Book HoTT formalisation by \citet{Baumann18}.
\item To ensure that our definition of $H^n(X,G)$ is sensible, we
  verify that it satisfies the HoTT formulation of the
  Eilenberg-Steenrod axioms for cohomology in
  Section \ref{sec:eilenbergsteenrodaxioms}, which implies the
  existence of the Mayer-Vietoris sequence in Section
  \ref{sec:mayervietoris}, as originally shown in HoTT by \citet{CavalloMsc15}.
\item We generalise the synthetic definition of the
  Thom isomorphism and Gysin sequence of \citet{Brunerie16} from
  $H^n(X,\bZ)$ to cohomology with arbitrary commutative ring coefficients in Section
  \ref{sec:gysin}.
\item The cohomology theory is then used to compute various $H^n(X,G)$ and
  $H^*(X,R)$ for concrete values of $X$, $G$, and $R$ in Section
  \ref{sec:direct}. We do this for the spaces studied by \citet{BLM}
  (spheres, torus, real/complex projective planes, and Klein
  bottle), generalising those cohomology group computations from $\bZ$
  to arbitrary $G$. We also provide details for how to carry out the
  cohomology ring computations of \citet{CohomRings23}. Finally, we
  also have some new synthetic computations of these invariants for
  $\RPinf$.
\end{itemize}

The present paper is carefully written so that all proofs are
constructive. Furthermore, all results have been formalised in
\CubicalAgda \citep{cubicalagda2} which has enabled us to do concrete
computations with the operations defined in the paper. A summary file
linking the paper to the formalised results can be found here:
\url{https://github.com/agda/cubical/blob/master/Cubical/Papers/ComputationalSyntheticCohomology.agda}

In order to facilitate efficient computer computations we have been
careful to give as direct constructions as possible. This allows us to
define various new numbers similar to the famous `Brunerie number'
\citep{Brunerie16}, but which are much simpler and still intractable
to compute in \CubicalAgda. In Section \ref{sec:computations}, we
collect various open computational problems and benchmark challenges
for \CubicalAgda and related systems in which univalence and HITs have
computational content. However, while everything has been formalised
in \CubicalAgda, we have been careful to write the paper in the
informal style of the HoTT Book. This means that the results in the
paper, except for those in Section \ref{sec:computations}, can all be
interpreted in suitably structured $\infty$-topoi \citep{Shulman19},
despite this not yet being known for results formalised in the
particular cubical type theory of \CubicalAgda.

While the paper is written in Book HoTT, we take some liberties and
sometimes deviate from the notational conventions of the book. We
detail this, as well some convenient cubical ideas that are useful
also when working in Book HoTT, in Section \ref{sec:background}.

\section{Background on HoTT/UF and notations}
\label{sec:background}

\defcitealias{HoTT13}{HoTT Book}

We assume familiarity with the contents and notations of the HoTT Book
\citep{HoTT13}, from here on referred to as \citepalias{HoTT13}. In
this section we briefly recall a few key definitions and introduce
some of our notational conventions.

\begin{definition}[Binary $\mathsf{ap}$]\label{def:ap2}
  Given a binary function $f : A \to B \to C$, we denote by
  $\mathsf{ap}^2_f$
  the function
  \[(a_0 = a_1) \times (b_0 = b_1) \to f(a_0,b_0) = f(a_1,b_1) \]
  The following path exists\footnote{Note that we could have defined
    $\mathsf{ap}^2_f$ in terms of $\mathsf{ap}$ so that
    $\mathsf{ap}^2_f\mathsf{-funct}$ holds definitionally. We choose
    not to do this in order to stay neutral with respect to the
    flavour of HoTT our proof holds in---the usual definition of
    $\mathsf{ap}^2_f$ in cubical type theory does not make this kind
    of functoriality hold definitionally.} for each
    $p : a_0 = a_1$ and $q : b_0 = b_1$:
  \[\mathsf{ap}^2_f\mathsf{-funct}(p , q) : \mathsf{ap}^2_f(p,q) = \mathsf{ap}_{f(-,b_0)}(p) \cdot \mathsf{ap}_{f(a_1,-)}(q)\]
\end{definition}

We define the \emph{(homotopy) fibre} of a function $f : A \to B$ over
a point $b :B$ as in \citepalias[Definition 4.2.4]{HoTT13}, i.e.\ as
$\fib{f}{b} := \Sigma_{a : A}(f\,a = b)$. We use the
\emph{contractible maps} definition of equivalences in
\citepalias[Definition 4.4.1]{HoTT13} and say that a function $f : A \to B$
is an \emph{equivalence} if, for each $b:B$, we have that $\fib{f}{b}$
is contractible \citepalias[Definition 3.11.1]{HoTT13}. We will simply say
\emph{proposition} for mere propositions
\citepalias[Definition~3.1.1]{HoTT13} and use \emph{set} and
\emph{$n$-type} (with $n \geq -2$) as in \citepalias[Section 3.1 and
7]{HoTT13}.

A \emph{pointed type} \citepalias[Definition 2.1.7]{HoTT13} is a pair
$(A,\star_A)$ where $A$ is a type and $\star_A : A$ is a chosen
basepoint of $A$. We often omit $\star_A$ and simply write $A$ for the
pointed type $(A,\star_A)$. We always take $\star_A$ to denote the
basepoint of a pointed type $A$. Given two pointed types $A$ and $B$,
we denote by $A \to_\star B$ the type of \emph{pointed functions} from
$A$ to $B$ \citepalias[Definition 8.4.1]{HoTT13}, i.e.\ the type of pairs
$(f,p)$ where $f : A \to B$ is a function and
$p : f(\star_A) = \star_B$. We often simply write $f : A \to_\star B$
and take $p$ implicit. If $f$ further is an equivalence, we write
$f : A \simeq_\star B$.

A class of pointed types particularly important for us is that of
\emph{H-spaces}. We borrow the definition (including notation) from \citet[Definition
  2.5.1]{Brunerie16}.
\begin{definition}[H-spaces]\label{def:hspace}
  An H-space is a pointed type $A$ equipped with a multiplication
  $\mu : A \times A \to A$ and homotopies
  \begin{align*}
    \mu_\mathsf{l} &: (a : A) \to \mu(\star_A,a) = a\\
    \mu_\mathsf{r} &: (a : A) \to \mu(a,\star_A) = a\\
    \mu_{\mathsf{lr}} &: \mu_{\mathsf{l}}(\star_A) = \mu_{\mathsf{r}}(\star_A)
  \end{align*}
\end{definition}
The notion of an H-space is closely related to \emph{homogeneous
  types}. 
\begin{definition}[Homogeneous types]\label{def:homogeneous}
  A pointed type $A$ is homogeneous if, for every $a : A$ there is a
  pointed equivalence $(A,\star_A) \simeq_\star (A,a)$.
\end{definition}
%
%

The loop space of a pointed type $A$ is defined as $\Omega(A) :=
(\star_A = \star_A)$. Path composition defines an invertible H-space
structure on $\Omega(A)$. Consequently, $\Omega(A)$ is also
homogeneous.

We will also rely on a few standard higher inductive types (HITs).
For spheres, and many other constructions, we need
suspensions. We use the standard definition from \citetalias[Section 6.5]{HoTT13}, which we recall here:
\begin{definition}[Suspensions]
  The suspension of a type $A$, denoted $\Sigma A$, is defined as a
  HIT with the following constructors:
  \begin{itemize}
  \item $\north,\south : \Sigma A$
    \item $\mathsf{merid} : A \to \north = \south$
  \end{itemize}
\end{definition}
Using this, we define the n-spheres by:
\[
  \mathbb{S}^n =
  \begin{cases}
    \mathsf{Bool} & \text{when } n = 0 \\
    \Sigma \mathbb{S}^{n-1} & \text{otherwise}
  \end{cases}
\]
In fact, suspensions are just a special case of \emph{(homotopy) pushouts}:
\begin{definition}[Homotopy pushouts]
  Given a span of functions $A \xleftarrow{f} C \xrightarrow{g} B$, we
  define its (homotopy) pushout, denoted $A \sqcup^C B$, as a HIT with
  the following constructors:
  \begin{itemize}
  \item $\inlop : A \to A \sqcup^C B$
  \item $\inrop : B \to A \sqcup^C B$
  \item $\pushop : (c : C) \to \inl{(f(c))} = \inr{(g(c))}$
  \end{itemize}
\end{definition}

We will also make use of $n$-truncations. These are defined using the
`hub and spoke' definition of \citepalias[Section 7.3]{HoTT13} and, as
usual, we denote the $n$-truncation of a type $A$ by $\truncT{A}_n$
and write $\trunc{a} : \truncT{A}_n$ for its canonical elements.
Following \citepalias[Definition 7.5.1]{HoTT13}, we say that a type $A$ is
\emph{$n$-connected} if $\truncT{A}_n$ is contractible and a function
$f : A \to B$ is said to be \emph{$n$-connected} if the fibre of $f$
over any $b : B$ is $n$-connected.

\subsection{Dependent paths and cubical thinking}


While this paper is written in Book HoTT, it is still often helpful to use
ideas from cubical type theory and `think cubically'. One reason for this is that iterated path types are
conveniently represented by higher cubes. This cubical approach to
Book HoTT was explored in depth by \citet{LicataBrunerie15} and we
will here outline some cubical ideas relevant to this paper.

We will often speak of \emph{dependent} path types---in particular
squares of paths. Given two paths $p : x = y$ and $q : z = w$, we cannot
ask whether $p = q$, since this is not well-typed. We could, however,
ask whether $p$ and $q$ are equal \emph{modulo} composition with some
other paths $r : x = z$ and $s : y = w$. We say that we ask for a
\emph{filler} of the square
\[\begin{tikzcd}[ampersand replacement=\&]
	z \&\& w \\
	x \&\& y
	\arrow["q", from=1-1, to=1-3]
	\arrow["p"', from=2-1, to=2-3]
	\arrow["s"', from=2-3, to=1-3]
	\arrow["r", from=2-1, to=1-1]
\end{tikzcd}\]
by which we mean a proof of the identity
$$\mathsf{transport}^{X \mapsto X}(\app{=}{r}{s},p) = q$$
(or,
equivalently, $r^{-1}\cdot p \cdot s = q$, or
$\mathsf{Square}\,r\,p\,q\,s$ using the notation of
\citet[Section~IV.C]{LicataBrunerie14}.
In cubical type theory, this would simply amount to providing a term
of type $\mathsf{PathP}\,(\lambda i.\,r\,i = s\,i)\,p\,q$. We note also
that the type of filler of a degenerate square
\[\begin{tikzcd}[ampersand replacement=\&]
	x \&\& y \\
	x \&\& y
	\arrow["q", from=1-1, to=1-3]
	\arrow["p"', from=2-1, to=2-3]
	\arrow["{\mathsf{refl}}"', from=2-3, to=1-3]
	\arrow["{\mathsf{refl}}", from=2-1, to=1-1]
\end{tikzcd}\]
precisely coincides with the type $p = q$.

As with regular paths, we may apply functions also to squares. For
instance, given a function $f : A \to B$ and a filler $sq$ of the
following square in $A$:
\[\begin{tikzcd}[ampersand replacement=\&]
	z \&\& w \\
	x \&\& y
	\arrow["q", from=1-1, to=1-3]
	\arrow["p"', from=2-1, to=2-3]
	\arrow["r"', from=2-3, to=1-3]
	\arrow["s", from=2-1, to=1-1]
\end{tikzcd}\]
we get the following filler in $B$:
\[\begin{tikzcd}[ampersand replacement=\&]
	f(z) \&\& f(w) \\
	f(x) \&\& f(y)
	\arrow["\ap{f}{q}", from=1-1, to=1-3]
	\arrow["\ap{f}{p}"', from=2-1, to=2-3]
	\arrow["\ap{f}{r}"', from=2-3, to=1-3]
	\arrow["\ap{f}{s}", from=2-1, to=1-1]
\end{tikzcd}\]
We will, with some minor abuse of notation\footnote{Technically, the
  outer $\mathsf{ap}$ should be dependent function application,
  i.e.\ $\mathsf{apd}$, if we were to follow HoTT Book notation.
  We will not distinguish the two and use $\mathsf{ap}$ for both.},
denote this filler by $\mathsf{ap}_{\mathsf{ap}_f}(\mathsf{sq})$

Let us give another example of manipulations of squares which will
play a crucial roll in Section~\ref{sec:direct}. There is a map which
takes a square filler and returns a filler of the same square, flipped
along its diagonal:
\[\begin{tikzcd}[ampersand replacement=\&]
	z \&\& w \\
	x \&\& y
	\arrow["q", from=1-1, to=1-3]
	\arrow["p"', from=2-1, to=2-3]
	\arrow["s"', from=2-3, to=1-3]
	\arrow["r", from=2-1, to=1-1]
\end{tikzcd}
\overset{\mathsf{flip}}{\rightsquigarrow}
\begin{tikzcd}[ampersand replacement=\&]
	y \&\& w \\
	x \&\& z
	\arrow["s", from=1-1, to=1-3]
	\arrow["r"', from=2-1, to=2-3]
	\arrow["q"', from=2-3, to=1-3]
	\arrow["p", from=2-1, to=1-1]
\end{tikzcd}
\]
The map is easily defined by path induction. In cubical type theory,
$\mathsf{flip}$ simply takes a square $I^2 \to A$ and flips the order
of the arguments. This operation is homotopically non-trivial in a
crucial sense. Consider the type of fillers of the square
\[\begin{tikzcd}[ampersand replacement=\&]
	x \&\& x \\
	x \&\& x
	\arrow["{\mathsf{refl}}", from=1-1, to=1-3]
	\arrow["{\mathsf{refl}}"', from=2-1, to=2-3]
	\arrow["{\mathsf{refl}}"', from=2-3, to=1-3]
	\arrow["{\mathsf{refl}}", from=2-1, to=1-1]
\end{tikzcd}\]
This is precisely the path type $\mathsf{refl}_x = \mathsf{refl}_x$,
i.e.\ $\Omega^2(A,x)$. Since all sides of the square are the same, $\mathsf{flip}$ defines an endofunction
$\Omega^2(A,x) \to \Omega^2(A,x)$. One might suspect that it reduces
to the identity in this degenerate setting, but this is actually not
the case. Rather, we get the following:
\begin{lemma}\label{lem:flip-id}
  Given $p : \Omega^2(A,x)$, we have $\mathsf{flip}(p) = p^{-1}$.
\end{lemma}
\begin{proof}
  Let $Q$ be the following filler.
  \[\begin{tikzcd}[ampersand replacement=\&]
	z \& w \\
	x \& y
	\arrow[""{name=0, anchor=center, inner sep=0}, "p"', from=2-1, to=2-2]
	\arrow["s"', from=2-2, to=1-2]
	\arrow["r", from=2-1, to=1-1]
	\arrow[""{name=1, anchor=center, inner sep=0}, "q", from=1-1, to=1-2]
	\arrow["Q"{description}, shorten <=4pt, shorten >=4pt, Rightarrow, from=0, to=1]
\end{tikzcd}\]
Consider the following cube:
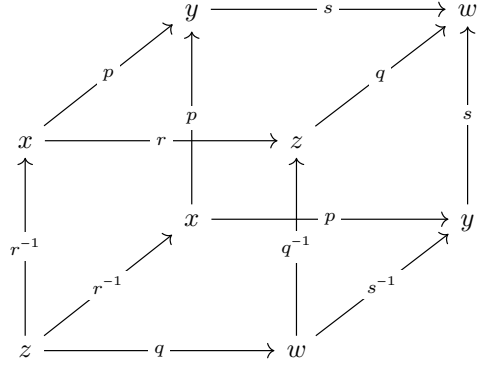
\[\begin{tikzcd}[ampersand replacement=\&]
	\&\& y \&\&\& w \\
	\\
	x \&\&\& z \\
	\&\& x \&\&\& y \\
	\\
	z \&\&\& w
	\arrow["{r^{-1}}"{description}, from=6-1, to=4-3]
	\arrow["q"{description}, from=6-1, to=6-4]
	\arrow["p"{description}, from=4-3, to=4-6]
	\arrow["{s^{-1}}"{description}, from=6-4, to=4-6]
	\arrow["{r^{-1}}"{description}, from=6-1, to=3-1]
	\arrow["p"{description}, from=3-1, to=1-3]
	\arrow["p"{description}, from=4-3, to=1-3]
	\arrow["r"{description}, from=3-1, to=3-4]
	\arrow["{q^{-1}}"{description}, from=6-4, to=3-4]
	\arrow["q"{description}, from=3-4, to=1-6]
	\arrow["s"{description}, from=4-6, to=1-6]
	\arrow["s"{description}, from=1-3, to=1-6]
\end{tikzcd}\]
Let its bottom be filled by $Q^{-1}$ and its top be filled by
$\mathsf{flip}(Q)$. All sides have their obvious fillers defined by
path induction on $p$, $q$, $r$ and $s$. We claim that the cube has a
filler. We prove this by first applying path induction on $p$,$r$ and
$s$, transforming $Q$ to a path $\mathsf{refl}_x = q$. Finally, after
applying path induction on $Q$ we are left to fill a cube with
$\mathsf{refl}_{\mathsf{refl}_x}$ on each side, which is trivial.

Specialising the above argument to the case when $p,q,r$ and $s$ are
all $\mathsf{refl}_x$ and $Q$ is arbitrary, we have our lemma.
\end{proof}

\section{Eilenberg-MacLane spaces}
\label{sec:EM}

In order to define representable cohomology in HoTT, we will need to
define Eilenberg-MacLane spaces. These are spaces $K(G,n)$ associated
to an abelian\footnote{Technically, this is only needed when $n>1$,
  but in this paper we will, for simplicity, always assume $G$ to be
  abelian.} group $G$ and a natural number $n$ such that
$\pi_n(K(G,n)) \simeq G$ and all other homotopy groups vanish. These
spaces carry the structure of an H-space which we will later see
induces the group structure on cohomology. We will also see that they
come equipped with a graded multiplication which, also in
Section~\ref{sec:cohomology}, will be lifted to define cohomology
rings.

The definition of Eilenberg-MacLane spaces and their H-space structure
in HoTT is due to \citet{LicataFinster14}. The special case of
integral Eilenberg-MacLane spaces, i.e.\ $K(\bZ,n)$, was considered by
\citet{Brunerie16}, who gave an alternative and very compact
definition in terms of truncated spheres. This definition of
$K(\bZ,n)$ was also considered by \citet{BLM}. In this section, we will
generalise the optimised proofs regarding integral Eilenberg-MacLane
spaces and their H-space structure found in~\citep{BLM} to a definition
of general Eilenberg-MacLane spaces following~\citet{LicataFinster14}.

In order to ease the notation when describing the induction principles of $K(G,n)$, we will define it one
step at a time. The crucial step is $K(G,1)$, which will be defined in
terms of the following construction.


\begin{definition}
  Given a type $A$, we define $L(A)$ by the HIT
  \begin{itemize}
    \item $\star : L(A)$
    \item $\Loop_k : A \to \star = \star$
    \end{itemize}
\end{definition}

Given a group $G$, the type $L(G)$ is \emph{almost} the first Eilenberg-MacLane
space of $G$. By adding a constructor connecting the group structure
on $G$ with the H-space structure
on $\Omega(L(G))$, we approximate it further.
\begin{definition}[Raw Eilenberg-MacLane spaces]
  We define the first raw Eilenberg-MacLane space, denoted
  $\widetilde{K}(G,1)$, by the following HIT
  \begin{itemize}
  \item $\iota : L(A) \to \widetilde{K}(G,1)$
  \item For every $g_1,g_2 : G$, a filler $\mathsf{sq}(g_1,g_2)$ of the square
    \[\begin{tikzcd}[ampersand replacement=\&]
	\iota(\star) \&\&\& \iota(\star) \\
	\iota(\star) \&\&\& \iota(\star)
	\arrow["{\ap{\iota}{\LoopG{(g_1 + g_2)}}}", from=1-1, to=1-4]
	\arrow["{\ap{\iota}{\LoopG{g_1}}}"', from=2-1, to=2-4]
	\arrow["{\ap{\iota}{\LoopG{g_2}}}"', from=2-4, to=1-4]
	\arrow["{\mathsf{refl}}", from=2-1, to=1-1]
\end{tikzcd}\]
  \end{itemize}
  For $n > 1$, we define $\widetilde{K}(G,n) := {\Sigma
    (\widetilde{K}(G,n-1))}$.
\end{definition}

In practice, we will omit the $\iota$ and simply write e.g.\
$\star : \widetilde{K}(G,1)$ and
$\LoopG{g} : \Omega(\widetilde{K}(G,1))$. We let $0_k$ denote the
basepoint of $\widetilde{K}(G,n)$, i.e.\ $0_k := \star$ when $n=1$ and
$0_k := \north$ when $n>1$.

\begin{definition}[Eilenberg-MacLane spaces]
  Given an integer $n \geq 1$ and an abelian group $G$, we define the
  $n$th Eilenberg-MacLane space of $G$  by
  $K(G,n) := \truncT{\widetilde{K}(G,n)}_n$. The zeroth
  Eilenberg-MacLane space is simply $K(G,0) := G$.
\end{definition}

The type $K(G,n)$ is pointed by $0_G$ when $n = 0$ and $\trunc{0_k}$
when $n\geq 1$. With some abuse of notation, we will simply write $0_k$ to
denote the basepoint of $K(G,n)$.

Eilenberg-MacLane spaces come with important elimination
principles. Let $n \geq 1$, the fundamental elimination principle of
$K(G,n)$ is given in Figure~\ref{elim:KG} and says that, given a
fibration of $n$-types $P : K(G,n) \to \ntype{n}$, in order to define a section
$(x : K(G,n)) \to P(x)$, it suffices to describe its action on
canonical elements $\trunc{x} : K(G,n)$, where $x:\widetilde{K}(G,n)$.
If $P$ is set-valued and $n = 1$, it suffices to define the section
for elements in $L(G)$---see Figure~\ref{elim:KG-set}. In other words,
we do not need to prove that our construction respects the
$\mathsf{sq}$-constructor.
Finally, it is easy to see that if $P$ is a family of $(n-2)$-types,
then any section $(x : K(G,n)) \to P(x)$ is uniquely determined by its
action on the basepoint $0_k : K(G,n)$---see Figure~\ref{elim:KG-triv}.

\begin{figure}[h!]%
    \centering
    \subfloat[\centering $n$-type elimination \label{elim:KG}]{\begin{tikzcd}[ampersand replacement=\&]
	\& \Sigma_{x : K(G,n)}{P(x)} \\
	\widetilde{K}(G,n) \& {K}(G,n)
	\arrow["{\mathsf{fst}}"', from=1-2, to=2-2]
	\arrow["{|-|}"', from=2-1, to=2-2]
	\arrow[from=2-1, to=1-2]
	\arrow["{\exists!}"', bend right=40, dashed, from=2-2, to=1-2]
  \end{tikzcd}}%
    \quad
    \subfloat[\centering $0$-type elimination for $K(G,1)$ \label{elim:KG-set}]{\begin{tikzcd}[ampersand replacement=\&]
	\& \Sigma_{x : K(G,1)}{P(x)} \\
	L(G) \& {K}(G,1)
	\arrow["{\mathsf{fst}}"', from=1-2, to=2-2]
	\arrow["{|-|\circ \iota}"', from=2-1, to=2-2]
	\arrow[from=2-1, to=1-2]
	\arrow["{\exists!}"', bend right=40, dashed, from=2-2, to=1-2]
\end{tikzcd}}%
    \quad
    \subfloat[\centering $(n-2)$-type elimination \label{elim:KG-triv}]{
      \begin{tikzcd}[ampersand replacement=\&]
	\& \Sigma_{x : K(G,n)}{P(x)} \\
	\mathbbm{1} \& {K}(G,n)
	\arrow["{\mathsf{fst}}"', from=1-2, to=2-2]
	\arrow["{\_ \, \mapsto 0_k}"', from=2-1, to=2-2]
	\arrow[from=2-1, to=1-2]
	\arrow["{\exists!}"', bend right=40, dashed, from=2-2, to=1-2]
      \end{tikzcd}}%
    \caption{Elimination principles for $K(G,n)$}%
    \label{fig:example}%
\end{figure}

We will, in Section~\ref{sec:omegakgn}, see that
$\Omega^n K(G,n) \simeq G$ but, before this, let us establish that $K(G,n)$ is
$(n-1)$-connected. These two facts imply that $\pi_n(K(G,n)) \simeq G$
and that all other homotopy groups vanish, which therefore shows that
$K(G,n)$ indeed is an Eilenberg-MacLane space.

\begin{proposition}
  \label{prop:K-connected}
  The type $K(G,n)$ is $(n-1)$-connected.
\end{proposition}
\begin{proof}
  We want to show that $\truncT{K(G,n)}_{n-1}$ is contractible.  When
  $n = 0$, the statement is trivial, since $\truncT{A}_{-1}$ is
  contractible for any pointed type $A$. Let $n \geq 1$, we choose the
  centre of contraction to be $\trunc{\KO}$. We now want to construct
  a section $(x : \truncT{K(G,n)}_{n-1}) \to \trunc{\KO} = x$. Since
  $\truncT{K(G,n)}_{n-1}$ is an $(n-1)$-type, all path types over it
  are $(n-2)$-types. Thus, it suffices to construct the section over
  $\trunc{\KO}$, which is trivial by $\refl : \trunc{\KO} =
  \trunc{\KO}$.
\end{proof}

The type ${K}(G,n)$ is functorial in $G$, i.e.\ for any homomorphism
$\phi : G_1 \to G_2$, we have an induced map $\phi_n : K(G_1,n) \to
K(G_2,n)$. We define this as a map
$\widetilde{K}(G_1,n) \to \widetilde{K}(G_2,n)$. When $n = 0$, we set
$\phi_0 := \phi$. For $n = 1$, $\phi_1$ is defined by
\begin{align*}
  \phi_1(\star) &:= \star \\
  \ap{\phi_1}{\LoopG{g}} &:= \LoopG{(\phi_0(g))}
\end{align*}

This respects the $\mathsf{sq}$ constructor of $\widetilde{K}(G_1,n)$
since $\phi$ is assumed to be a homomorphism. For $n \geq 2$, we
define it recursively using functoriality of $\Sigma$ and the fact
that $\widetilde{K}(G,n) := {\Sigma (\widetilde{K}(G,n-1))}$:

\begin{center}
\begin{tikzcd}[ampersand replacement=\&]
	{\widetilde{K}(G_1,n)} \&\& {\widetilde{K}(G_2,n)} \\
	{\Sigma(\widetilde{K}(G_1,n-1))} \&\& {\Sigma(\widetilde{K}(G_2,n-1))}
	\arrow[Rightarrow, no head, from=1-1, to=2-1]
	\arrow["{\Sigma(\phi_{n-1})}", from=2-1, to=2-3]
	\arrow["{\phi_n}", dashed, from=1-1, to=1-3]
	\arrow[Rightarrow, no head, from=1-3, to=2-3]
\end{tikzcd}
\end{center}

We also remark that there is a map
$\sigma_n : K(G,n) \to \Omega(K(G,n+1))$. We define it for raw
Eilenberg-MacLane spaces by
\begin{align*}
  \sigma_n(x) :=\begin{cases} \LoopG{g} & \text{ if $n = 0$}\\
  \merid{x} \cdot (\merid{0_k})^{-1} &\text{ otherwise}
  \end{cases}
\end{align*}
As $\Omega(K(G,n+1))$ is an $n$-type, this definition also induces, via $n$-type elimination for Eilenberg-MacLane spaces (see Figure~\ref{elim:KG}),
a map $K(G,n) \to \Omega(K(G,n+1))$:
\begin{center}
\begin{tikzcd}[ampersand replacement=\&]
	{\widetilde{K}(G,n)} \&\& {\Omega(\widetilde{K}(G,n+1))} \& {\Omega(K(G,n+1))} \\
	{K(G,n)}
	\arrow["{\sigma_n}", from=1-1, to=1-3]
	\arrow["{\mathsf{ap}_{\trunc{-}}}", from=1-3, to=1-4]
	\arrow[from=1-1, to=2-1]
	\arrow[dashed, from=2-1, to=1-4]
\end{tikzcd}
\end{center}
With some abuse of
notation, we also denote this map by $\sigma_n$. The following fact follows immediately by construction of $\sigma_n$.
\begin{lemma}\label{lem:funct-susp-commute}
  Given a group homomorphism $\phi : G_1 \to G_2$, we have $$\ap{\phi_{n+1}}{\sigma_n(x)} =_{\Omega(K(G_2,n))} \sigma_n(\phi_n(x))$$ for $x : K(G_1,n)$.
\end{lemma}

\subsection{Group structure}
\label{sec:emgroupstructure}

One of the key contributions of \citet{BLM} was the computationally
optimised definition of the H-space structure on $K(\bZ,n) :=
\truncT{\mathbb{S}^n}_n$ (when $n \geq 1$). Our goal here is to
rephrase this optimised construction to $K(G,n)$ defined as above. The
following proposition is crucial. It is a special case of~\citepalias[Lemma
  8.6.2.]{HoTT13} for which we provide a direct proof in order to
make future constructions relying on it more computationally efficient.

\begin{proposition}
  \label{prop:wedgeconn}
  Let $n,m\geq 1$ and assume we are
  given a fibration $\widetilde{K}(G_1,n) \times
  \widetilde{K}(G_2,m) \to \ntype{(n+m-2)}$, sections $f : (x :
  \widetilde{K}(G_1,n)) \to P(x,\KO)$ and $g : (y : \widetilde{K}(G_2,m))
  \to P(\KO,y)$, and a path $p : f(0_k) = g(0_k)$.
In this case, there is a unique section $$\timeswedge{f}{g} : ((x , y)
: {\widetilde{K}(G_1,n) \times \widetilde{K}(G_2,m)}) \to P(x,y)$$ equipped with homotopies
\begin{align*}
  &l : (x : \widetilde{K}(G_1,n)) \to
  (\timeswedge{f}{g})(x,0_k) = f(x) \\
  &r : (y : \widetilde{K}(G_2,m))
  \to (\timeswedge{f}{g})(0_k,y) = g(y)
\end{align*}
satisfying $l(0_k)^{-1} \cdot r(0_k) = p$.
\end{proposition}
\begin{proof}
  An easy way of proving this is to note that the canonical map
  \[
  i_{\vee} : {\widetilde{K}(G_1,n) \vee \widetilde{K}(G_2,m)} \to
  {\widetilde{K}(G_1,n) \times \widetilde{K}(G_2,m)}
  \]
     is $((n-1) +
     (m-1))$-connected, thereby inducing a section via the map
     \[
     f \vee g
  : (x : {\widetilde{K}(G_1,n) \vee \widetilde{K}(G_1,m)}) \to P
     (i_\vee(x))
     \]
     From the point of view of computer formalisation,
  however, this proof leads to poor computational behaviour of
  subsequent definitions. For this reason, we provide a direct proof.

  We proceed by induction on $n$ and $m$. We first consider the
  base-case: $n=m=1$. We are to construct a section $\timeswedge{f}{g}
  : ((x,y): \widetilde{K}(G_1,1) \times \widetilde{K}(G_2,1)) \to
  P(x,y)$. Since $P$ is a set, it suffices to define a section
  $\timeswedge{f}{g} : ((x,y): L(G_1) \times L(G_2)) \to P(x,y)$.
  We proceed by $L(G_1)$-induction on $x$. For the point constructor, we define
  \begin{align*}
    (\timeswedge{f}{g})(\star,y) &:= g(y)
  \end{align*}
  We now need to define $\ap{(\timeswedge{f}{g})(-,y)}{\LoopG{a}}$ for $a : G_1$. Since $P$ is set valued, it suffices to do so when $y$ is $\star$. We define
  \begin{align*}
    \ap{(\timeswedge{f}{g})(-,\star)}{\LoopG{a}} &:= p^{-1}\cdot \ap{f}{\LoopG{a}} \cdot p
  \end{align*}
We note that this
construction satisfies the additional properties by default: we
construct $l : (x : \widetilde{K}(G_1,1)) \to (\timeswedge{f}{g})(x,\KO)
= f(x)$ by induction on $x$. Since this is a proposition, it suffices
to construct $l(0_k) : g(\KO) = f(\KO)$, which we do by $l(0_k) :=
p^{-1}$.
We construct $r : (y : \widetilde{K}(G_2,1)) \to
(\timeswedge{f}{g})(\KO,x)$ simply by $r(y) := \refl$, since the
equality holds by definition. We finally need to show that ${l(0_k)}^{-1}
\cdot r(0_k)= p$ which is immediate by construction of $l$ and $r$.

We proceed by induction on $n$, letting $n \geq 2$ and omitting the case when $n = 1$ and $m\geq 2$ which is entirely symmetric. Let us define $\timeswedge{f}{g} : ((x , y) :
{\widetilde{K}(G_1,n) \times \widetilde{K}(G_2,m)}) \to P(x,y)$. We stress that since
$n \geq 2$, we have that $\widetilde{K}(G_1,n) = \Sigma
\widetilde{K}(G_1,n-1)$ by definition. Hence, we may define
$\timeswedge{f}{g}$ by suspension induction on $x$. Let us start with the
action on point constructors.
\begin{align*}
  (\timeswedge{f}{g})(\mathsf{north},y) &:= g(y)\\
  (\timeswedge{f}{g})(\mathsf{south},y) &:= \transport{P(-,y)}{\merid{\KO}}{g(y)}
\end{align*}
We now need to define $\ap{\timeswedge{f}{g}}{-,y}(\merid{a}) :
\Square{g(y)}{\transport{P(-,y)}{\merid{\KO}}{g(y)}}{P(-,y)}{\merid{a}}$
for \linebreak $a :\widetilde{K}(G_1,n-1)$ and $y : \widetilde{K}(G_2,m)$. The type
of $\ap{\timeswedge{f}{g}}{-,y}(\merid{a})$ is an $(n+m-3)$-type, and
thus, by induction hypothesis, it suffices to construct the following:
\begin{align*}
  f' &: (a : \widetilde{K}(G_1,n-1))\to\Square{g(\KO)}{\transport{P(-,\KO)}{\merid{\KO}}{g(\KO)}}{P(-,\KO)}{\merid{a}} \\
  g' &: (y : \widetilde{K}(G_2,m))\to \Square{g(y)}{\transport{P(-,y)}{\merid{\KO}}{g(y)}}{P(-,y)}{\merid{\KO}} \\
  p' &: f'(\KO) = g'(\KO)
\end{align*}
By composition with $p$, the target type of $f'$ is
equivalent to
$\Square{f(\north)}{f(\south)}{P(-,\KO)}{\merid{a}}$, of which we have
a term: $\ap{f}{\merid{a}}$. The target type of $g'$ is equivalent to
$g(y) = g(y)$ and we simply give the term $\refl$. The construction of $p'$
is an easy but technical lemma. Thus we have constructed
$\timeswedge{f}{g}$. Now note that $r : (y : \widetilde{K}(G_2,m)) \to
(\timeswedge{f}{g})(\KO,y) = g(y)$ holds by $\refl$. The path $l : (x
: \widetilde{K}(G_1,n-1)) \to (\timeswedge{f}{g})(x,\KO) = g(y)$ is
easily constructed using the corresponding $l$- and $r$-paths from the
inductive hypothesis. This can be done so that $l(\KO) =
r(\KO)$ holds almost by definition.
\end{proof}

Our goal now is to define addition $+_k : K(G,n) \times K(G,n) \to
K(G,n)$. When $n = 0$, it is simply addition in $G$. For $n \geq 1$,
we note that, by the elimination principle
for $K(G,n)$ in Figure~\ref{elim:KG}, it suffices
to define a corresponding operation $+_k : \widetilde{K}(G,n) \times
\widetilde{K}(G,n) \to K(G,n)$.
When $n = 1$, we define it explicitly by
\begin{align*}
  \star +_k y &:= y\\
  \ap{\_ +_k \star}{\LoopG{g}}  &:= \LoopG{g} \\
  \ap{x \mapsto \, \to \ap{\_ +_k x}{\LoopG{g_1}}} {\LoopG{g_2}}  &:= Q
\end{align*}
where $Q$ is a filler of the following square
\[\begin{tikzcd}[ampersand replacement=\&]
	\star \& \star \\
	\star \& \star
	\arrow["{\LoopG{g_1}}", from=1-1, to=1-2]
	\arrow["{\LoopG{g_1}}"', from=2-1, to=2-2]
	\arrow["{\LoopG{g_2}}"', from=2-2, to=1-2]
	\arrow["{\LoopG{g_2}}", from=2-1, to=1-1]
\end{tikzcd}\]
which is easily constructed using that $\LoopG{}$ preserves composition
by definition of $K(G,1)$. All other cases are immediate since
$K(G,1)$ is $1$-truncated. For higher values of $n$, the construction
is made straightforward by Proposition~\ref{prop:wedgeconn}: Indeed, in this case, we have
that $K(G,n)$ is of h-level $n\leq
n + n -2$. Therefore, we may define $+_k$ simply by
\begin{align}
  \label{eq1}
  x +_k 0_k & := x \\
  \label{eq2}
  0_k +_k y & := y
\end{align}
Proposition \ref{prop:wedgeconn} also requires us to prove the above
definitions to coincide when $x= 0_k$ and $y = 0_k$, which they do by
$\refl$. Thus, we have defined $+_k : \widetilde{K}(G,n) \times
\widetilde{K}(G,n) \to K(G,n)$, which lifts to an addition $K(G,n)
\times K(G,n) \to K(G,n)$ which we, abusing notation, will also denote
by $+_k$.

The laws governing $+_k$ are remarkably easy to prove.
\begin{proposition}
  \label{prop:+-unit}
  The addition $+_k$ has $0_k$ as a left- and
  right-unit. Furthermore, these are coherent in the sense that we
  have a filler of the following square \textnormal{
    \[\begin{tikzcd}[ampersand replacement=\&]
	{0_k + 0_k} \& {0_k} \\
	{0_k + 0_k} \& {0_k}
	\arrow["{\textsf{r-unit}}(0_k)", from=1-1, to=1-2]
	\arrow[Rightarrow, no head, from=1-1, to=2-1]
	\arrow["{\textsf{l-unit}}(0_k)"', from=2-1, to=2-2]
	\arrow[Rightarrow, no head, from=1-2, to=2-2]
\end{tikzcd}\]}
\end{proposition}
\begin{proof}
  The statement follows from the group structure on $G$ when $n =
  0$. When $n = 1$ the target type is a set, since $K(G,1)$ is
  $1$-truncated (and hence path types over it are sets). By the
  set-elimination rule for $K(G,n)$ (see~Figure~\ref{elim:KG-set}), it
  suffices to show that $\trunc{x} +_k 0_k = 0_k = 0_k +_k \trunc{x}$
  for $x : L(G)$. These equalities hold definitionally by $L(G)$
  induction. Hence, we also get the filler of the square by
  $\refl$. When $n \geq 2$, the statement follows immediately by the
  defining equations~\eqref{eq1}~and~\eqref{eq2}, using the family of
  paths $l$ and $r$ and the coherence between them, as given in
  Proposition \ref{prop:wedgeconn}.
\end{proof}

In fact, modulo $\mathbb{N}$-induction, by the direct proof of
Proposition \ref{prop:wedgeconn}, we get that the left-and right-unit
laws are definitionally equal at $0_k$, with both definitionally equal
to $\refl$.

\begin{proposition}
  \label{prop:+-comm}
  The addition $+_k$ is commutative.
\end{proposition}
\begin{proof}
  When $n = 0$, the statement is simply that $G$ is abelian, which we
  have assumed.  Let $n \geq 1$.  By truncation elimination, it
  suffices to construct paths $\trunc{x} +_k \trunc{y} = \trunc{y} +_k
  \trunc{x}$ for $x,y:\widetilde{K}(G,n)$. This path type is
  $(n-1)$-truncated which is sufficiently low for Proposition
  \ref{prop:wedgeconn} to apply. Hence it suffices to construct paths
  $p_x : \trunc{x} + 0_k = 0_k + \trunc{x}$ and $q_y : 0_k +_k
  \trunc{y} = \trunc{y} + 0_k$ and show that $p_{0_k} = q_{0_k}$. We
  construct both paths as compositions of the left- and right-unit
  laws of $+_k$ in the obvious way. These definitions coincide over
  $0_k$ by the second part of Proposition \ref{prop:+-unit}.
\end{proof}

The following is equally direct to prove (using Proposition~\ref{prop:wedgeconn}), so we omit its
proof.
\begin{proposition}
  \label{prop:+-assoc}
  The addition $+_k$ is associative.
\end{proposition}
Propositions~\ref{prop:+-unit},~\ref{prop:+-comm}~and~\ref{prop:+-assoc}
amount precisely to showing that $K(G,n)$ is a commutative,
associative H-space. There is, however, more to the story. In HoTT, it
is easy to define an inversion on $K(G,n)$. When $n = 0$, it is simply
negation in $G$. When $n \geq 1$, we have two different options. One
approach is to note that for any $x : K(G,n)$, the map
$y \mapsto x + y$ is an equivalence. This is
proved by observing that the property of being an equivalence is a
proposition. The elimination rule of $K(G,n)$ into $(n-2)$-types tells
us that it is enough to show that the map is an equivalence when $x$
is $0_k$; in this case the map is simply the identity, and thus also
an equivalence. The unique element in the fibre of $y \mapsto x + y$
over $0_k$ is taken to be the inverse of $x$. This proof also gives
the desired cancellations laws by construction. This is the approach
used in e.g.\ \citep{Brunerie16,LicataFinster14}.

While elegant, the above approach has one major disadvantage: the
action of $-_k$ on any element other than $0_k$ is described using
transports. One consequence of this is that the terms produced by
$-_k$ often are hard to work with directly. Another consequence is
that terms described using $-_k$ tend to explode in size, which leads
to poor computatonal behaviour in implementations in proof
assistants. For these reasons, we choose the obvious direct
construction of $-_k$. We define it as a map
$-_k : \widetilde{K}(G,n) \to \widetilde{K}(G,n)$. When $n = 1$, we
define
\begin{align*}
  -_k \,\star &:= \star\\
  \ap{-_k}{\LoopG{g}} &:= (\LoopG{g})^{-1}\\
  \ap{\mathsf{ap}_{-_k}}{\mathsf{sq}\,g_1\,g_2} &:= Q
\end{align*}
where $Q$ is a filler of the square
\[\begin{tikzcd}[ampersand replacement=\&]
	\star \&\& \star \\
	\star \&\& \star
	\arrow[Rightarrow, no head, from=2-1, to=1-1]
	\arrow["{(\LoopG{g_1})^{-1}}"', from=2-1, to=2-3]
	\arrow["{\LoopG{(g_1+g_2)}^{-1}}", no head, from=1-1, to=1-3]
	\arrow["{(\LoopG{g_2})^{-1}}"', from=2-3, to=1-3]
\end{tikzcd}\]
which is easy to construct using the functoriality of $\LoopG{}$.
For $n \geq 2$, we define it by
\begin{align*}
  -_k \, \north &:= \north\\
  -_k \, \south &:= \north\\
  \ap{-_k}{\merid{g}} &:= \sigma_n(g)^{-1}
\end{align*}

For the cancellation laws, we will need the following lemma. We remark
that it is only is well-typed in the case when $n = 1$ is fixed or
when $n$ is on the form $(2 + m)$.
The fact that it is well-typed exploits the fact that
$0_k = 0_k +_k 0_k$ holds definitionally modulo the case distinction
on $n$.
\begin{proposition}
  \label{prop:ap-+}
  Let $n \geq 1$. Given $p, q : \Omega (K(G,n))$, we have
  $\app{+_k}{p}{q} = p \cdot q$.
\end{proposition}
\begin{proof}
  We have
  \[
\app{+_k}{p}{q} = \ap{x\mapsto x +_k 0_k}{p} \cdot \ap{y\mapsto 0_k +_k y}{q} = p \cdot q
\]
The final step consists in applying the right- and left-unit laws to
both components. This is justified since they agree at $0_k$.
\end{proof}

We write $x -_k y$ for $x +_k (-_k\,y)$ and have that the directly
defined $-_k$ indeed is inverse to $+_k$ by the following proposition:

\begin{proposition}
  For $x : K(G,n)$, we have $x -_k x = 0_k = (-_k\,x) +_k x$.
\end{proposition}
\begin{proof}
  By commutativity of $+_k$, it it suffices to show that
  $x -_k x = 0_k$. When $n = 0$, this comes from the group structure
  on $G$. When $n = 1$, we note that the goal type is a set, which
  means that it suffices to show the statement for $x : L(G)$. We have
  $\star -_k \star = 0_k$ by $\refl$. For $\LoopG{}$, we need to show
  that
    \[
    \app{x,y\mapsto x -_k y}{\LoopG{g}}{\LoopG{g}} = \refl
    \]
    By Proposition~\ref{prop:ap-+}, the LHS above is equal to
    $\LoopG{g} \cdot (\LoopG{g})^{-1}$ which is equal to $\refl$. The
    case when $n \geq 2$, is entirely analogous, again using~Proposition~\ref{prop:ap-+} for the path constructors.
\end{proof}

And there we have it: the $3$-tuple $(K(G,n),+_k,0_k)$ forms, for any
$n$, a commutative and associative H-space with inverse given by $-_k$. We will see in
Section~\ref{sec:cohomology} that this trivially gives an induced
group structure on cohomology groups, but let us first take a closer look at
$\sigma_n$ and also examine the multiplicative structure on $K(G,n)$.

\subsection{$K(G,n)$ vs. $\Omega K(G,n+1)$}
\label{sec:omegakgn}

We will soon see that $\sigma_n : K(G,n) \to \Omega K(G,n+1)$ in
fact is an equivalence. It will be crucial when computing
cohomology groups of spaces. However, we do not need it yet:
Proposition~\ref{prop:ap-+} already provides us with a strong link
between the structures on $K(G,n)$ and $\Omega K(G,n+1)$. For
instance, we may already now deduce the commutativity of
$\Omega K(G,n+1)$ from the commutativity of $+_k$:
\begin{proposition}
  \label{prop:comm-path}
  For $x : K(G,n)$ and $p,q : x = x$, we have $p \cdot q = q \cdot p$.
\end{proposition}
\begin{proof}
  When $n = 0$, we are automatically done, since $G$ is a set. For
  $n \geq 1$, it suffices, by $(n-2)$-type elimination, to show the
  statement when $x$ is $0_k$. Let $p,q: \Omega K(G,n)$. We have
  \begin{align*}
    p \cdot q = \app{+_k}{p}{q} = \app{+_k}{q}{p} = q \cdot p
  \end{align*}
  where the second equality comes from the commutativity of $+_k$.
\end{proof}

The usefulness of the above proof (as opposed to one using an
equivalence $K(G,n) \simeq \Omega K(G,n+1)$ as in~\citep{Brunerie16})
is that it is rather direct and is easily shown to be homotopically
trivial when $p$ and $q$ are $\refl$.

On a similar note, we may also provide a direct proof of the following proposition.
\begin{proposition}
  \label{prop:ap--}
  Let $n \geq 1$ and $p: \Omega K(G,n)$, we have $\ap{-_k}{p} = p^{-1}$.
\end{proposition}
\begin{proof}
  We give the proof for $n = 1$ as the case when $n > 1$ is
  entirely analogous.
  Let us define a fibration $\mathsf{Code}_x: (p : \star = x) \to \mathsf{hProp}$ for
  $x:K(G,1)$. As $\mathsf{hProp}$ is a set, it suffices to define
  $\mathsf{Code}_x(p)$ for $x : L(G)$. We do this by induction on
  $x$ and define
  \begin{align*}
    \mathsf{Code}_\star(p) := (\ap{-_k}{p} = p^{-1})
  \end{align*}
  which is a proposition since $K(G,1)$ is $0$-connected. We need to
  check that this definition respects the action of $\mathsf{Code}$ on
  $\LoopG{g}$ for $g:G$. This amounts to showing that, for any $p :
  \Omega K(G,1)$, we have an equivalence
  \begin{align*}
    (\ap{-_k}{p} = p^{-1}) \simeq (\ap{-_k}{p \cdot \LoopG{g}} &= (p \cdot \LoopG{g})^{-1})
  \end{align*}
  We rewrite the terms on the RHS as follows:
  \begin{align*}
    \ap{-_k}{p \cdot \LoopG{g}} &= \ap{-_k}{p} \cdot  \ap{-_k}{\LoopG{g}}
    \\ &= \ap{-_k}{p} \cdot {(\LoopG{g})}^{-1}
  \end{align*}
  and
  \begin{align*}
(p \cdot \LoopG{g})^{-1} &= (\LoopG{g})^{-1} \cdot p^{-1} \\
&= p^{-1} \cdot (\LoopG{g})^{-1}
  \end{align*}
Thus, we only need to construct an equivalence.
  \begin{align*}
    (\ap{-_k}{p} = p^{-1}) \simeq (\ap{-_k}{p} \cdot (\LoopG{g})^{-1} &= p^{-1} \cdot (\LoopG{g})^{-1})
  \end{align*}
  This equivalence is given by $\mathsf{ap}_{q \mapsto q \cdot
    (\LoopG{g})^{-1}}$ and thereby the construction of $\mathsf{Code}$
  is complete.

  We now note that we have a section
  $(p : \star = x) \to \mathsf{Code}_x(p)$ for any $x:K(G,1)$ since, by path
  induction, it suffices to note that $\refl$ is an element of
  $\mathsf{Code}_\star(\refl)$. Hence, in particular, we have an
  element of $\mathsf{Code}_\star(p)$ for all $p : \Omega K(G,1)$,
  which is what we wanted to show.
\end{proof}

The following proposition is
  also crucial and its proof can be found in
  \citep[Lemma~4.4]{LicataFinster14}.
\begin{proposition}
  \label{prop:+-hom}
  For $x,y : K(G,n)$, we have $\sigma_n(x +_k y) = \sigma_n(x) \cdot
  \sigma_n(y)$.
\end{proposition}

Because $\sigma_n$ is pointed, Proposition~\ref{prop:+-hom}
immediately gives us the following corollary (which is proved just
like in group theory).
\begin{corollary}
  \label{prop:--hom}
  For $x: K(G,n)$, we have $\sigma_n(-_k \, x) = \sigma_n(x)^{-1}$.
\end{corollary}
%
We are now ready for the main result of the section, namely that
$\Omega K(G,n+1) \simeq K(G,n)$. The original proof in HoTT is due
to~\citet[Theorem 4.3]{LicataFinster14} who used encode-decode proofs
in the cases $n=0$ and $n=1$, and the Freudenthal suspension theorem
(formalised in HoTT by~\citet{FavoniaFinster+16}) in the case $n \geq
2$. Here we use an identical proof in the case $n = 1$. However, it is
possible to avoid the Freudenthal suspension theorem completely by
simply noting that the proof of Licata~and~Finster when $n = 1$
actually can be used for all $n \geq 1$. We primarily wish to avoid
the Freudenthal suspension theorem as circumventing it makes the
computer formalisation more computationally viable. 
%
\begin{therm}
  \label{thm:deloop}
  The map $\sigma_n : K(G,n) \to \Omega K(G,n+1)$ is an equivalence.
\end{therm}
\begin{proof}
  For $n = 0$, we refer to the proof in~\citep{LicataFinster14}. For
  $n \geq 1$, we choose to provide an encode-decode proof similar to
  their $n = 1$ case. We define
  $\mathsf{Code} : K(G,n+1) \to \ntype{n}$ by
\begin{align*}
  \mathsf{Code}(\trunc{\north}) &:= K(G,n) \\
  \mathsf{Code}(\trunc{\south}) &:= K(G,n) \\
  \ap{\mathsf{Code}\,\circ\, \trunc{-}}{\merid\,a} &:= \mathsf{ua}(x \mapsto \trunc{a} +_k x)
\end{align*}
Here $\mathsf{ua}$ is the part of univalence which turns an equivalence into a path. Clearly the map to which it is applied is an equivalence.
For any $t : K(G,n+1)$, we have a map $\mathsf{encode}_t : 0_k = t \to
\mathsf{Code}(t)$ defined by
$\mathsf{encode}_t(p):=\transport{\mathsf{Code}}{p}{0_k}$ (this is
equivalent to path induction on $p$). We define $\mathsf{decode}_t :
\mathsf{Code}(t) \to 0_k = t$ by truncation elimination and
$K(G,n)$-induction on $t$. For the point constructors, we define it by
\begin{align*}
  \mathsf{decode}_{\trunc{\north}}(x) &= \sigma_n(x)\\
  \mathsf{decode}_{\trunc{\south}}(x) &= \sigma_n(x)\cdot \ap{\trunc{-}}{\merid{0_k}}
\end{align*}
For the path constructor, we need to, after some simple rewriting of the goal, provide a path
\[
\sigma_n(x-_k\trunc{a}) = \sigma_n(x)\cdot \sigma_n(\trunc{a})^{-1}
\]
for $x : K(G,n)$ and $a : \widetilde{K}(G,n)$ which is easily done using Proposition~\ref{prop:+-hom} and Corollary~\ref{prop:--hom}. We can now show that, for any $t : K(G,n+1)$ and $p
: 0_k = t$, we have $\mathsf{encode}_t(\mathsf{decode}_t(p)) = p$ by
path induction on $p$. In particular, we have
$\mathsf{decode}_{0_k}(\mathsf{encode}_{0_k}(p)) = p$ for all $p :
\Omega K(G,n+1)$. We can also show that
$\mathsf{encode}_{0_k}(\mathsf{decode}_{0_k}(x)) = x$ for all
$x:K(G,n)$ by simply unfolding the definitions of
$\mathsf{decode}_{0_k}$ and $\mathsf{encode}_{0_k}$. Indeed, after applying truncation elimination on $x$, we have:
\begin{align*}
  \mathsf{encode}_{0_k}(\mathsf{decode}_{0_k}(\trunc{a})) &:= \mathsf{encode}_{0_k}(\sigma_n(\trunc{a})) \\
  &:= \transport{\mathsf{Code}}{\sigma_n(\trunc{a})}{0_k}\\
  &= \transport{\mathsf{Code}\circ \trunc{-}}{\merid(0_k)^{-1}}{\transport{\mathsf{Code}\circ \trunc{-}}{\merid{a}}{0_k}} \\
  &= \transport{\mathsf{Code}\circ \trunc{-}}{\merid(0_k)^{-1}}{\trunc{a}+0_k}\\
  &=(-_k 0_k) +_k (\trunc{a}+0_k)\\
  &= \trunc{a}
\end{align*}
and we are done.
\end{proof}

Note that this implies that we have $\Omega^n (K(G,n)) \simeq G$. By
Proposition~\ref{prop:+-hom}, this equivalence is an isomorphism of
groups. Together with Proposition~\ref{prop:K-connected}, this shows
that $K(G,n)$ indeed is an Eilenberg-MacLane space.

\subsection{The cup product with group coefficients}
\label{sec:cupprodgroupcoeff}

There is an additional structure definable already on
Eilenberg-MacLane spaces: the cup product. The first construction in
HoTT is due to~\citet{Brunerie16}, who defined it for integral
cohomology, i.e.  $\smile_k \, : K(\bZ,n) \times K(\bZ,m) \to
K(\bZ,n+m)$. This definition was later extended to
$\smile_k \, : K(G_1,n) \times K(G_2,m) \to K(G_1 \otimes G_2,n+m)$ by
\citet{Baumann18}. The issue with the definitions of these cup
products is that they rely heavily on the properties of smash
products. As smash products have turned out to be rather difficult to
reason about in HoTT \citep{brunerie18}, not all cup product axioms have been formally
verified using these definitions. For this reason, a new definition of
the cup product on $\bZ$-cohomology was given by~\citet{BLM} which
allowed for a complete proof of the graded commutative ring laws. Our
goal in this section is to generalise this definition and get a map
$\smile_k \, : K(G_1,n) \times K(G_2,m) \to K(G_1 \otimes G_2,n+m)$
satisfying the expected laws.

Before defining the cup product, we need to define tensor products of
abelian groups. In HoTT, we can do this very directly using the
obvious HIT:
\begin{definition}
  Given abelian groups $G_1$ and $G_2$ we define their tensor product
  $G_1 \otimes G_2$ to be the HIT generated by 
  \begin{itemize}
    \item points $g_1 \otimes g_2 : G_1 \otimes G_2$, for each $g_1 :G_1$ and $g_2:G_2$
    \item points $\plustimes{x}{y}$ for each pair of points $x,y : G_1 \otimes G_2$
    \item for any $x,y,z : G_1 \otimes G_2$, paths
      \begin{align*}
        \plustimes{x}{y} &= \plustimes{y}{x} \\
        \plustimes{x}{(\plustimes{y}{z})} &= \plustimes{(\plustimes{x}{y})}{z} \\
        \plustimes{(0_{G_1} \otimes 0_{G_1})}{x} &= x
      \end{align*}
    \item for any $x : G_1$ and $y,z : G_2$,
      paths $x \otimes (y +_{G_2} z) = \plustimes{(x \otimes y)}{(x \otimes z)}$
    \item for any $x, y : G_1$ and $z : G_2$,
      paths $(x +_{G_1} y) \otimes z = \plustimes{(x \otimes z)}{(y \otimes z)}$
    \item a constructor $\normalfont{\textsf{is-set} : \mathsf{isSet}\,(G_1 \otimes G_2)}$
  \end{itemize}
\end{definition}
One easily verifies that this construction has the expected universal
property. Consequently, a map out of $G_1 \otimes G_2$ into a set is
well-defined if and only if it is structure preserving. When mapping into a
proposition, it suffices to construct the map over canonical elements
$g_1 \otimes g_2$ and show that the proposition respects $\oplus$.

Let us return to the problem at hand: defining the cup product. We
note first that it is very easy to define a preliminary cup product
\begin{align}
  \label{red-cup}
  {\smile}_k' : \widetilde{K}(G_1,n) \times K(G_2,m) \to K(G_1 \otimes G_2,n+m)
\end{align}
When $n = 0$, we define $g \smile_k' (-)$ by the functorial action of $K(-,m)$ on the group homomorphism given by $g \otimes (-) : G_2 \to G_1 \otimes G_2$. When $n = 1$, we inductively define:
\begin{align*}
  \star \smile_k' y &:= 0_k\\
  \ap{x \mapsto x \smile_k'\,y}{\LoopG{g}} &:= \sigma_m(g \smile_k' y)
\end{align*}
The fact that this respects the $\mathsf{sq}$ constructor is easy to
check. For $n \geq 2$, the idea is the same:
\begin{align*}
  \north \smile_k' y &= 0_k\\
  \south \smile_k' y &= 0_k\\
  \ap{x \mapsto x \smile_k'\,y}{\merid\,a} &= \sigma_{(n-1)+m}(a \smile_k' y)
\end{align*}
We would now like to be able to lift $\smile_k'$ to a full cup product
$\smile_k \, : K(G_1,n) \times K(G_2,m) \to K(G_1 \otimes G_2,n+m)$:
\[\begin{tikzcd}[ampersand replacement=\&]
	{\widetilde{K}(G_1,n) \times K(G_2,m)} \& {K(G_1\otimes G_2,n+m)} \\
	{K(G_1,n) \times (K(G_2,m)}
	\arrow["{\smile_k'}", from=1-1, to=1-2]
	\arrow[from=1-1, to=2-1]
	\arrow[dotted, from=2-1, to=1-2]
\end{tikzcd}\]
or, after currying:
\[\begin{tikzcd}[ampersand replacement=\&]
	{\widetilde{K}(G_1,n)} \& {(K(G_2,m) \to K(G_1\otimes G_2,n+m))} \\
	{K(G_1,n)}
	\arrow["{\smile_k'}", from=1-1, to=1-2]
	\arrow[from=1-1, to=2-1]
	\arrow[dotted, from=2-1, to=1-2]
\end{tikzcd}\]
However, we run into an obstruction: for this lift to
exist, we require the function type ${K(G_2,m) \to K(G_1\otimes G_2,n+m)}$ to be an
$n$-type, which it is not unless $m = 0$. Fortunately, a very trivial
observation makes the lift possible: we want our cup product to
satisfy $x \smile_k 0_k = 0_k$. In other words, we may see the
cup product as a map into a pointed function type:
\begin{align*}
  \smile_k \, : K(G_1,n) \to (K(G_2,m) \to_\star K(G_1 \otimes G_2,n+m))
\end{align*}
In this case, we do get a lift
\[\begin{tikzcd}[ampersand replacement=\&]
	{\widetilde{K}(G_1,n)} \& {(K(G_2,m) \to_\star K(G_1\otimes G_2,n+m))} \\
	{K(G_1,n)}
	\arrow["{\smile_k'}", from=1-1, to=1-2]
	\arrow[from=1-1, to=2-1]
	\arrow["{\smile_k}"', dotted, from=2-1, to=1-2]
\end{tikzcd}\]
and thereby the cup product is defined. The existence of the lift is
motivated by the following result which follows from a more general
result~\citep[Corollary 1]{BDR18}, but has a very simple direct proof
when the spaces involved are Eilenberg-MacLane spaces.

\begin{proposition}
  \label{prop:pointed-hlevel}
  The space $K(G_1,n) \to_\star K(G_2,n+m)$ is $m$-truncated.
\end{proposition}
\begin{proof}
  It is enough to show that $\Omega^{m+1}(K(G_1,n) \to_\star
  K(G_2,n+m))$ is contractible. We have
  \begin{align*}
    \Omega^{m+1} (K(G_1,n) \to_\star K(G_2,n+m)) &\simeq (K(G_1,n) \to_\star \Omega^{m+1} K(G_2,n+m)) \\
    &\simeq (K(G_1,n) \to_\star K(G_2,n-1))
  \end{align*}
  The fact that $(K(G_1,n) \to_\star K(G_2,n-1))$ is contractible
  follows easily from the the $(n-2)$- type elimination rule for
  $K(G_1,n)$ in Figure~\ref{elim:KG-triv}.
\end{proof}
This concludes the construction of $\smile_k$. Let us now prove the various ring laws governing it. We first note that since $\smile_k'$ is pointed in both arguments, this also applies to $\smile_k$. Thus, we have verified the first law:
\begin{proposition}
  For $x : K(G,n)$, we have $x \smile_k 0_k = 0_k$ and $0_k \smile_k x = 0_k$.
\end{proposition}
We now set out to prove the remaining laws.
This turns out to be rather difficult to do directly: we want to prove
these laws using $K(G,n)$-elimination via Proposition
\ref{prop:pointed-hlevel}, but this forces us to prove equalities of
pointed maps, rather than just maps. For instance, for
left-distributivity we have to prove that for any two points
$x,y: K(G_1,n)$ we have that the two maps
\begin{align*}
  z &\mapsto (x +_k y) \smile_k z \\
  z &\mapsto (x \smile_k z) +_k (y \smile_k z)
\end{align*}
are equal as \emph{pointed} maps living in $K(G_2,m) \to_\star K(G_1
\otimes G_2 , n + m)$. This is far more involved than proving equality
of the underlying maps: a proof of equality of pointed maps
$(f,p),(g,q) : A \to_\star B$ consists not only of a homotopy $h : (a
: A) \to f\,a = g\,a$, but also
a coherence $\Square{p}{q}{h(\star_A)}{\refl}$. Fortunately, there is
a trick to automatically infer the coherence.  We attribute this
result to Evan Cavallo.\footnote{The original form of the lemma was
  conjectured only for Eilenberg-MacLane spaces and appears as
  \citep[Lemma~14]{BLM}. Its proof, and generalisation to arbitrary
  homogeneous types, is attributed to Cavallo, whose original formalisation can be
  found in~\citep{HomogeneousFormalization}. The result has later been generalised by~\citet[Lemma 2.7]{buchholtz2023central}}


\begin{lemma}[Evan's trick]
  \label{lem:Evan}
  Given two pointed functions $(f,p), (g,q) : A \to_\star B$ with $B$ homogeneous.
  If $f = g$, then there is a path of pointed functions $(f,p) = (g,q)$.
\end{lemma}
\begin{proof}
  By assumption, we have a homotopy $h : (x : A) \to f(x) = g(x)$. We
  construct $r : \star_B = \star_B $ by
  \begin{align*}
    r := p^{-1} \cdot h(\star_A) \cdot q
  \end{align*}
  and define $P : (B,\star_B) =_{\mathcal{U}_*} (B,\star_B)$ by
  \begin{align*}
    P := \ap{(B,-)}{r}
  \end{align*}
  We get $P = \refl$ as an easy consequence of the homogeneity of
  $B$. Hence, instead of proving that $(f,p) = (g,q)$, it is enough to
  prove that $\transport{{A \to_\star (-)}}{P}{(f,p)} = (g, q)$. The
  transport only acts on $p$ and $q$, so the identity of the first
  components holds by $h$. For the second components, we are reduced
  to proving that $p \cdot r = h(\star_A) \cdot q$. This is true
  immediately by the construction of $r$.
\end{proof}

\begin{proposition}[Left-distributivity of $\smile_k$]
  Let $x,y : K(G_1,n)$ and $z : K(G_2,m)$. We have
  \[
  (x +_k y) \smile_k z = x \smile_k z +_k y \smile_k z
  \]
\end{proposition}
\begin{proof}
  When $n = 0$, the result follows by a straightforward induction on
  $m$. When $n \geq 1$, we generalise and, for $x,y : K(G_1,n)$,
  consider the two functions
  \begin{align*}
  z &\mapsto (x +_k y) \smile_k z \\
  z &\mapsto (x \smile_k z) +_k (y \smile_k z)
  \end{align*}
  We claim that these are equal as \emph{pointed} functions living in
  $K(G_2,m) \to_\star K(G_1\otimes G_2 , n + m)$. By Proposition
  \ref{prop:pointed-hlevel}, all path types over $K(G_2,m) \to_\star
  K(G_1\otimes G_2 , n + m)$ are $(n-1)$-types.
  This allows us to apply truncation elimination and Proposition
  \ref{prop:wedgeconn}. In combination with Lemma \ref{lem:Evan}, it
  suffices to construct, for each $z : K(G_2,m)$:
  \begin{align*}
    l(y) &: (0_k +_k \trunc{y}) \smile_k z = (0_k \smile_k z) +_k (\trunc{y} \smile_k z)  \\
    r(x) &: (\trunc{x} +_k 0_k) \smile_k z = (\trunc{x} \smile_k z) +_k (0_k \smile_k z) \\
    q &: l(0_k) = r(0_k)
  \end{align*}
  which is direct using the unit laws
  for $+_k$ and $\smile_k$. The obvious construction makes $q$ hold
  by $\refl$.
\end{proof}
Right-distributivity follows by an almost identical proof, so we omit
it.
\begin{proposition}[Right-distributivity of $\smile_k$]
  Let $x: K(G_1,n)$ and $y,z : K(G_2,m)$. We have
  \[
  x \smile_k (y +_k z) = x \smile_k y +_k x \smile_k z
  \]
\end{proposition}
Before moving on to associativity, we will need the following
lemma. As pointed out by~\citet{BLM}, this lemma also appears
in~\citep[Proposition 6.1.1]{Brunerie16} where it is used for the
construction of the Gysin sequence. Its proof relies on e.g.\ the
pentagon identity for smash products which, at the time, was open in HoTT. In our setting, which like that of \citet{BLM} uses pointed
maps instead of smash products, it holds by construction.
\begin{lemma}
  \label{lem:ap-cup}
  For $x : K(G_1,n)$ and $y : K(G_2,m)$, we have
  $$\ap{x\mapsto x \smile_k y}{\sigma_n(x)} = \sigma_{n+m}(x \smile_k y)$$
\end{lemma}
\begin{proof}
  When $n = 0$, the statement holds by definition. Let us consider the
  case $n \geq 1$. The maps we are comparing are of type $K(G_1,n) \to K(G_2,m)
  \to_\star \Omega (K(G_1 \otimes G_2 , (n + 1) + m))$. The
  type $$K(G_2,m) \to_\star \underset{\simeq K(G_1\otimes
    G_2,n+m)}{\underbrace{\Omega (K(G_1 \otimes G_2 , (n + 1) + m))}}$$ is an
  $n$-type by Proposition~\ref{prop:pointed-hlevel}, and hence we may
  use $K(G,n)$-elimination on $x$. Hence, we want to construct an equality of pointed functions
  \[(y\mapsto \ap{x\mapsto x \smile_k y}{\sigma_n(a)}) = (y \mapsto \sigma_{n+m}(a \smile_k' y))\]
  for $a : \widetilde{K}(G,n)$. By Lemma~\ref{lem:Evan}, it suffices to
  show this equality for underlying functions. This follows trivially:
\begin{align*}
  \ap{x\mapsto x \smile_k y}{\sigma_n(a)} &:= \ap{x\mapsto \trunc{x} \smile_k y}{\merid{a} \cdot (\merid{0_k})^{-1}} \\
  &= \ap{x\mapsto \trunc{x} \smile_k y}{\merid{a}} \cdot (\ap{x\mapsto \trunc{x} \smile_k y}{\merid{0_k}})^{-1}  \\
  &:= \sigma_{n+m}(a \smile_k' y) \cdot \sigma_{n+m}(0_k \smile_k' y)^{-1}\\
  &= \sigma_{n+m}(a \smile_k' y)
\end{align*}
\end{proof}

With Lemma \ref{lem:ap-cup}, associativity is straightforward to prove. 
\begin{proposition}
  The cup product is associative, i.e.\ the following diagram
  commutes:
\[\begin{tikzcd}[ampersand replacement=\&]
	{K(G_1,n)\times K(G_2,m) \times K(G_3,k)} \&\&\& {K(G_1\otimes (G_2 \otimes G_3),n+(m+k))} \\
	\&\&\& {K((G_1\otimes G_2) \otimes G_3,(n+m)+k)}
	\arrow["{x,y,z\mapsto x \smile_k (y \smile_k z)}", from=1-1, to=1-4]
	\arrow["{x,y,z\mapsto (x \smile_k y) \smile_k z}"', sloped, from=1-1, to=2-4]
	\arrow["\sim", sloped, from=2-4, to=1-4]
      \end{tikzcd}\]
\end{proposition}
\begin{proof}
  We induct on $n$. Let us first consider the inductive step and save the base case for last. To this end, let $n \geq 1$. We may consider the two functions we are comparing
  as doubly pointed functions of type
  \[K(G_1,n) \to (K(G_2,m) \to_\star (K(G_3,k) \to_\star K(G_1 \otimes (G_2 \otimes G_3) , n + m + k)))\]
  Applying an appropriate variant of Proposition
  \ref{prop:pointed-hlevel}, the codomain of the above function type is
  an $n$-type. When $x$ is a point constructor of $K(G_1,n)$, the
  diagram in the statement commutes by $\refl$. For the higher
  constructor ($\merid$ or $\LoopG{}$, depending on the value of $n$),
  we are done if we can show, for $x : K(G_1,n-1)$, that
  \begin{align*}
    \ap{x \mapsto x \smile_k (y \smile_k z)}{\sigma_{n-1}(x)} = \ap{x \mapsto \alpha_{n+m+k} ((x \smile_k y) \smile_k z)}{\sigma_{n-1}(x)}
  \end{align*}
   where $\alpha : (G_1 \otimes G_2) \otimes G_3 \simeq G_1 \otimes
   (G_2 \otimes G_3)$ and, recall, $\alpha_{p} : K((G_1 \otimes G_2)
   \otimes G_3,p) \to K(G_1 \otimes (G_2 \otimes G_3),p)$ is the
   functorial action of $K(-,p)$ on $\alpha$. We have
  \begin{align*}
    \ap{x \mapsto x \smile_k (y \smile_k z)}{\sigma_{n-1}(x)} &=
    \sigma_{(n-1)+(m+k)}(x \smile_k (y \smile_k z)) & \text{ (Lemma~\ref{lem:ap-cup})} \\
    &= \sigma_{(n-1)+(m+k)}(\alpha_{(n-1)+(m+k)} ((x \smile_k y) \smile_k z)) & \text{ (Ind. hyp.)} \\
    &= \ap{\alpha_{n+m+k}}{\sigma_{((n-1) + m)+k}((x \smile_k y) \smile_k z)} &\text{ (Lemma~\ref{lem:funct-susp-commute})} \\
    &= \ap{\alpha_{n+m+k}}{\ap{x \mapsto x \smile_k z}{\sigma_{((n-1) + m)}(x \smile_k y)}} &\text{ (Lemma~\ref{lem:ap-cup})} \\
    &= \ap{\alpha_{n+m+k}}{\ap{x \mapsto x \smile_k z}{{\ap{x \mapsto x \smile y}{\sigma_{(n-1)}(x)}}}} &\text{ (Lemma~\ref{lem:ap-cup})} \\
    &= \ap{x \mapsto \alpha_{n+m+k} ((x \smile_k y) \smile_k z)}{\sigma_{n-1}(x)}
  \end{align*}
  which concludes the inductive step.

  Let us return to the base case, i.e. the case $n = 0$. In this case,
  we fix $g_1 : G_1$ and instead consider the functions $y,z \mapsto
  g_1 \smile_k(y\smile_k z)$ and $y,z \mapsto (g_1 \smile_k y)\smile_k
  z$ to be of type $$K(G_2,m) \to K(G_3,k) \to_\star K(G_1\otimes
  (G_2 \otimes G_3) , m + k).$$ We induct on $m$. The inductive step
  follows in the same way as the inductive step above. For the base
  case, we fix $y$ to be $g_2 : G_2$ and are left to prove $g_1
  \smile_k (g_2 \smile_k z) = (g_1 \smile_k g_2) \smile_k z$ for $z :
  K(G_3,k)$. Again, this is follows by induction on $k$. The base case
  is given by associativity of $G_1\otimes (G_2 \otimes G_3)$ and the
  inductive step follows in exactly the same way as before.
\end{proof}
Let us verify graded commutativity. Although our argument is a direct generalisation of that of \citet[Proposition 18]{BLM}, we remark that a similar albeit somewhat more high-level argument appears in~\citep[Section~4.3]{Warn2023}.
\begin{proposition}
  \label{prop:cup-comm}
  The cup product is graded commutative, i.e.\ the following diagram
  commutes:
\[\begin{tikzcd}[ampersand replacement=\&]
	{K(G_1,n)\times K(G_2,m)} \&\&\& {K(G_1\otimes G_2,n+m)} \\
	\&\&\& {K(G_2\otimes G_1,m+n)}
	\arrow["{x,y\hspace{1mm}\mapsto (-1)^{nm}(y \smile_k x)}"', sloped, from=1-1, to=2-4]
	\arrow["\smile_k", from=1-1, to=1-4]
	\arrow["\sim", sloped, from=2-4, to=1-4]
\end{tikzcd}\]
\end{proposition}
\begin{proof}
  We may, without loss of generality, assume that $n \geq m$. We
  induct on the quantity $n+m$. The base-case when $n = m = 0$ is
  trivial. Let us consider the two cases $n = m = 1$ and $n,m \geq 2$
  and omit the two cases when $n \geq 2$ and $m < 2$ as they
  follow in an entirely analogous manner. In what follows, let $c_p :
  K(G_2\otimes G_1,p) \to K(G_1 \otimes G_2,p)$ denote the commutator.

We start with the case $n = m = 1$. In this case, we would like to be
able to apply the set-elimination rule for $K(-,1)$. As before, fix $x
: K(G_1,1)$. We strengthen the goal by asking for \linebreak $x \smile_k (-)$ and
$y \mapsto (-1)^{nm}(y \smile_k x)$ to be equal as pointed functions
living in $K(G_2,1) \to_\star K(G_1\otimes G_2 , 2)$. The type of such
functions is a $1$-type, and hence any path type over it is a
set. Thus, we may apply set-elimination and assume that $x :
L(G_1)$. We may now repeat the argument for $y : K(G_2,n)$: we would
like $(-) \smile_k y$ and $x \mapsto (-1)^{nm}(y \smile_k x)$ to be equal
as pointed functions living in $L(G_1) \to_\star K(G_1\otimes G_2,
2)$. By an very similar argument to that of Proposition
\ref{prop:pointed-hlevel}, this pointed function type is again a
$1$-type, which justifies set-elimination being applied to
$y$. Combining this with Lemma \ref{lem:Evan}, it is enough to show
that
  \begin{align*}
    x \smile_k y = - c_{2}(y \smile_k x)
  \end{align*}
  for $x : L(G_1)$ and $y : L(G_2)$. We do this by $L(G_1)$
  induction. When either $x$ or $y$ is a point constructor, the
  identity holds by $\refl$. The remaining case amounts to (by
  evaluation of $c_2$) the following identity
  \begin{align*}
    \ap{x \mapsto \ap{x \smile_k (-)}{\LoopG{h}}}{\LoopG{g}} = \mathsf{flip}(\ap{x \mapsto \ap{x \smile_k (-)}{\LoopG{h}}}{\LoopG{g}}^{-1})
  \end{align*}
  which holds by Lemma~\ref{lem:flip-id}.

  We give a rough sketch of the case $n,m \geq 2$. We may apply
  $K(G_1,n)$-elimination again, by the same argument as in the
  previous case. The crucial step this time is verifying
  \begin{align}
    \label{eq:cup-comm}
    \ap{x \mapsto \ap{x \smile_k (-)}{\merid\,h}}{\merid\,g} = \ap{x
      \mapsto \ap{(-1)^{nm} (x \,\smile_k
        (-))}{\merid\,h}}{\merid\,g}
  \end{align}
We note that the above identity and the identities that follow only
are well-typed up to some coherences which we brush under the rug for
the sake of readability. We notice that \eqref{eq:cup-comm} follows if
we can show that for $x : K(G_1,n-1)$, $y : K(G_2,m-1)$, we have
\begin{align}
  \label{eq:cup-comm-alt}
    \ap{x \mapsto \ap{x \smile_k (-)}{\sigma_{m-1}(y)}}{\sigma_{n-1}(x)} = \ap{x
      \mapsto \ap{(-1)^{nm} (x \,\smile_k
        (-))}{\sigma_{m-1}(y)}}{\sigma_{n-1}(x)}
\end{align}
Let us, for further readability, omit the equivalence $c_{p}$
in the following equations. Some of the
following equalities hold up to multiplication by $(-1)^p$ for some $p
: \mathbb{N}$; this will be indicated in the right-hand column.
\begin{align*}
  \ap{a \mapsto \ap{a \smile_k (-)}{\sigma_{m-1}(y)}}{\sigma_{n-1}(x)}
  &= \ap{a \mapsto \ap{(-) \smile_k a}{\sigma_{n-1}(x)}}{\sigma_{m-1}(y)} & (-1)\\
  &= \ap{\sigma_{(n-1)+m}}{\ap{x \smile_k (-)}{\sigma_{m-1}(y)}} \\
  &= \ap{\sigma_{m+(n-1)}}{\ap{(-) \smile_k x}{\sigma_{m-1}(y)}} & m(n-1) \\
  &= \ap{\sigma_{(m-1)+(n-1)+1}}{\sigma_{(m-1)+(n-1)}(y \smile_k x)} \\
  &= \ap{\sigma_{{(n-1)+(m-1)+1}}}{\sigma_{{(n-1)+(m-1)}}(x \smile_k y)} & (n-1)(m-1) \\
  &= \ap{\sigma_{n+(m-1)}}{\ap{(-) \smile_k y}{\sigma_{n-1}(x)}} \\
  &= \ap{\sigma_{m+(n-1)}}{\ap{y \smile_k (-)}{\sigma_{n-1}(x)}} & (m-1)n \\
  &= \ap{a
      \mapsto \ap{(-) \,\smile
        a}{\sigma_{m-1}(y)}}{\sigma_{n-1}(x)}
\end{align*}

Since $(-1)+n(m-1) + (n-1)(m-1) + (n-1)m =_{\bZ/2\bZ} nm$, the above identity
holds up to multiplication by $(-1)^{nm}$, which gives
\eqref{eq:cup-comm-alt}.
\end{proof}

\subsection{The cup product with ring coefficients}
\label{sec:cupprodringcoeff}

We now consider the special case of $\smile_k \, : K(G_1,n) \times
K(G_2,m) \to K(G_1\otimes G_2,n + m)$ where $G_1 = G_2 = R$ for a (not
necessarily commutative) ring $R$. Let $\mathsf{mult}_{n} : K(R\otimes
R, n) \to K(R,n)$ be the map induced by
the multiplication $R \otimes R \to R$. Now, consider the
composition
\[\begin{tikzcd}[ampersand replacement=\&]
	{K(R,n)\times K(R,m)} \&\& {K(R\otimes R,n+m)} \&\& {K(R,n+m)}
	\arrow["{\smile_k}", from=1-1, to=1-3]
	\arrow["{\mathsf{mult}_{n+m}}", from=1-3, to=1-5]
\end{tikzcd}\]
Let us, with some overloading of notation, use $\smile_k :
K(R,n)\times K(R,m) \to K(R,n)$ to denote also the composition
above. Note that the above tensor product is a tensor product of
underlying abelian groups and not of rings, although this is of no
consequence for the construction. Using the results from Section~\ref{sec:cupprodgroupcoeff}, we easily arrive at the following results. Their proofs
all follow the same strategy and we will only sketch the proof of
Proposition \ref{prop:cup-comm-ring}.
\begin{proposition}
  \label{prop:cup-props}
  The cup product $\smile_k$ with ring coefficients is associative,
  left- and right-distributive and has $0_k$ as annihilating element.
\end{proposition}
\begin{proposition}
  \label{prop:cup-comm-ring}
  The cup product $\smile_k$ with coefficients in a commutative ring
  is graded-commutative.
\end{proposition}
\begin{proof}
  Let $c_p$ denote the commutator $K(G\otimes H,p) \to K(H \otimes
  G,p)$. Consider the following diagram.
\[\begin{tikzcd}[ampersand replacement=\&]
	{K(R,n)\times K(R,m)} \&\& {K(R\otimes R,n+m)} \&\& {K(R,n+m)} \\
	{K(R,m)\times K(R,n)} \&\& {K(R\otimes R,m+n)} \&\& {K(R,m+n)}
	\arrow["{\smile_k}", from=1-1, to=1-3]
	\arrow[from=1-1, to=2-1]
	\arrow["{\smile_k}", from=2-1, to=2-3]
	\arrow["{\mathsf{mult}_{n+m}}", from=1-3, to=1-5]
	\arrow["{\mathsf{mult}_{m+n}}", from=2-3, to=2-5]
	\arrow["{(-1)^{nm}}", from=1-5, to=2-5]
	\arrow["{c\,\circ\,(-1)^{nm}}"', from=1-3, to=2-3]
\end{tikzcd}\]
The statement is that the outer square commutes. By Proposition
\ref{prop:cup-comm} the left-square commutes. It is easy to verify,
using commutativity of $R$, that $\mathsf{mult}_{n+m} \circ c =
\mathsf{mult}_{m+n}$. Hence, the right-square commutes, and we are
done.
\end{proof}

Finally, we note that we now have access to a neutral element,
i.e.\ $1_R : K(R,0)$ which is a left and right unit for $\smile_k$:

\begin{proposition}
  For $x : K(R,n)$, we have $x \smile_k 1_R = x = 1_R \smile_k x$.
\end{proposition}
\begin{proof}
  By a straightforward induction on $n$, using the recursive
  definition of the cup product.
\end{proof}

\section{Cohomology}
\label{sec:cohomology}

For the construction of cohomology groups, we rely on
\citet{BrownRepr} representability: given $n:\mathbb{N}$, an abelian group $G$
and a type $X$, we define the \emph{nth
  cohomology group of X with coefficients in $G$} by\footnote{Here, one could more generally have chosen a dependent version by letting $H^n(X,G_{(-)}) := \truncT{(x : X) \to K(G_x,n)}_0$ for any family of abelian groups $G_x$ over $X$ as is done by e.g.~\citet[Definition 5.4.2]{FlorisPhd}. We choose the non-dependent version for ease of presentation, although everything which follows can be read with the dependent version in mind.}
\begin{align*}
  H^n(X,G) := \truncT{X \to K(G,n)}_0
\end{align*}

This construction is clearly contravariantly functorial by
pre-composition. The group structure on $K(G,n)$ naturally carries
over to $H^n(X,G)$ by point-wise application of $+_k$ and $-_k$. The
neutral element is, in both cases, the constant map $\_ \, \mapsto
0_k$. In what follows, we will denote addition, inversion and the
neutral element in $H^n(X,G)$ by $+_h$, $-_h$ and $0_h$ respectively.
Explicitly, they are defined by
\begin{align*}
  \trunc{f} +_h \trunc{g} &:= \trunc{x \mapsto f(x) +_k g(x)}\\
  -_h \trunc{f} &:= \trunc{x \mapsto -_k f(x)}\\
  0_h &:= \trunc{\_ \, \mapsto 0_k}
\end{align*}

In the same way, our construction of the cup product also carries over to cohomology groups. In particular, the cup product with ring coefficients in
Section~\ref{sec:cupprodringcoeff} gives us a cup product on cohomology
\begin{align*}
  \smile \, : H^n(X,R) \times H^m(X,R) \to H^{n+m}(X,R)
\end{align*}
satisfying the usual graded ring axioms and, when
$R$ is commutative, graded commutativity. In addition, it induces a multiplication on
$H^*(X,R) = \oplus_{n : \mathbb{N}} H^n(X,R)$, turning it into a
graded-commutative (assuming $R$ is commutative) ring known as the \emph{cohomology ring} of $X$
with coefficients in $R$. This was formalised in \CubicalAgda
by~\citet{CohomRings23}.
This gives an even more fine-grained invariant than just the
cohomology groups $H^n(X,G)$, as it can help to distinguish spaces for
which all cohomology groups agree. As a classic example which was
also considered by \citet{CohomRings23}, consider
Figure~\ref{fig:torusandmickeymouse} which depicts a torus and the
`Mickey Mouse space' consisting of a sphere glued together with two
circles.

\begin{figure}[h!]
  \centering
  \begin{mathpar}
    \begin{tikzpicture}[scale=0.5]
      \useasboundingbox (-3,-1.5) rectangle (3,1.5);
      \draw[color=blue!90] (0,0) ellipse (3 and 1.5);
      \begin{scope}
        \clip (0,-1.8) ellipse (3 and 2.5);
        \draw[color=blue!90] (0,2.2) ellipse (3 and 2.5);
      \end{scope}
      \begin{scope}
        \clip (0,2.2) ellipse (3 and 2.5);
        \draw[color=blue!90] (0,-2.2) ellipse (3 and 2.5);
      \end{scope}
      \draw[color=blue!90] (0,-1.5) to[bend right=100] (0,-0.3);
      \draw[color=blue!90,densely dotted] (0,-1.5) to[bend left=100] (0,-0.3);
    \end{tikzpicture}
    \and
    \begin{tikzpicture}[scale=0.5]
      \draw[color=blue!90](0.5,0) circle (1.5);

      \draw[color=blue!90](0.5,1.55) circle (0.6);
      \draw[color=blue!90](-2.35,1.55) circle (0.6);
      \draw[color=blue!90] (-2.4,0) to[bend right=40] (0.58,0);
      \draw[color=blue!90,densely dotted] (-2.4,0) to[bend left=40] (0.58,0);
    \end{tikzpicture}
  \end{mathpar}
  \caption{$\mathbb{T}^2$ and $\mathbb{S}^2 \vee \mathbb{S}^1 \vee \mathbb{S}^1$.}
  \label{fig:torusandmickeymouse}
\end{figure}

Let $X$ be either of the two spaces in the figure; one can prove that
$H^0(X,\bZ) \cong \bZ$ as they are both connected, $H^1(X,\bZ) \cong \bZ
\times \bZ$ as they each have two 2-dimensional holes (the circle in
the middle of the torus and the one going around the interior, and the
`ears' of Mickey), and $H^2(X,\bZ) \cong \bZ$ as they have one
3-dimensional hole each (the interior of the torus and Mickey's
head). Furthermore, their higher cohomology groups are all trivial as
they do not have any further higher dimensional cells. The cohomology
groups are hence insufficient to tell the spaces apart. However, one
can show that the cup product on $\mathbb{S}^2 \vee \mathbb{S}^1 \vee
\mathbb{S}^1$ is trivial, while for $\mathbb{T}^2$ it is not. So
$H^*(\mathbb{S}^2 \vee \mathbb{S}^1 \vee \mathbb{S}^1,\bZ) \not \cong
H^*(\mathbb{T}^2,\bZ)$, which again means that the spaces cannot be
equivalent (see Section \ref{subsec:proofsbycomp} for details on how this can be proved by computer computation in \CubicalAgda).

\subsection{Reduced cohomology}
\label{sec:reduced}

When $X$ is pointed, we may also define \emph{reduced} cohomology
groups by simply requiring all functions to be pointed:
\begin{align*}
  \widetilde{H}^n(X,G) := \truncT{X \to_\star K(G,n)}_0
\end{align*}

Just like before, the additive structure on $K(G,n)$ carries over directly
to $\widetilde{H}(X,G)$. As we will soon see, this definition is
equivalent to the non-reduced version when $n\geq 1$ but has the
advantage of vanishing for connected spaces when $n = 0$.

\begin{proposition}
  The first projection (sending a pointed map to its underlying
  function) induces an isomorphism
  $\widetilde{H}^n(X,G) \cong H^n(X,G)$ for $n \geq 1$.
\end{proposition}
\begin{proof}
  Let $F$ denote the function in question. The fact that it is a
  homomorphism is trivial. We may construct its inverse $F^{-1} :
  H^n(X,G) \to \widetilde{H}^n(X,G)$ explicitly by
  \begin{align*}
    F^{-1}\,\trunc{f}:= \trunc{ x \mapsto f(x) -_k f(\star_X)}
  \end{align*}
  In order to see that $F(F^{-1} \trunc{f}) = \trunc{f}$, we note that
  this is a proposition. By connectedness of $K(G,1)$, we may
  assume we have a path $f(\star_X) = 0_k$. Hence
  \begin{align}\label{eq:red-cohom}
    F(F^{-1}\,\trunc{f}):= \trunc{ x \mapsto f(x) -_k f(\star_X)} = \trunc{f}
  \end{align}
  For the other direction, we need to check that
  \begin{align*}
    F^{-1}(F\,\trunc{(f,p)}) = \trunc{(f,p)}
  \end{align*}
  By Lemma \ref{lem:Evan}, we only need to verify that the first
  projections agree, and we are reduced to the second equality
  in~\eqref{eq:red-cohom}, which we have already verified.
\end{proof}

\begin{proposition}
  $\widetilde{H}^0(X,G) \oplus G \cong H^0(X,G)$.
\end{proposition}
\begin{proof}
  We define $\phi : (X \to_\star G) \times G \to (X \to G)$ by
  \[
  \phi(f,g) := x \mapsto f(x) +_G g
  \]
  and $\psi : (X \to G) \to (X \to_\star G) \times G$ by
  \[
  \psi(f) := ((x \mapsto f(x) -_G f(\star_X)) , f(\star_X))
  \]
  The fact that these maps cancel out follows immediately from the
  group laws on $G$. This induces, after applying the appropriate set truncations, an equivalence $H^0(X,G) \simeq
  \widetilde{H}^0(X,G) \oplus G$. The fact that it is an isomorphism
  is also trivial using the group laws on $G$.
\end{proof}


\subsection{Eilenberg-Steenrod axioms for cohomology}
\label{sec:eilenbergsteenrodaxioms}

In order to verify that our cohomology theory is sound, let us verify
the \emph{Eilenberg-Steenrod Axioms}~\citep{eilenberg1952foundations}
for the reduced version of cohomology $\widetilde{H}^n(X,G)$. These
axioms have been studied previously in HoTT by
\citet{ShulmanBlog13,CavalloMsc15,BuchholtzFavonia18,FlorisPhd}. To
state the \textbf{Exactness} axiom, we need to introduce cofibres
(also known as \emph{mapping cones}).

\begin{definition}
  Given a function $f : X \to Y$, the cofibre of $f$, denoted
  $\cofib{f}$, is pushout of the span
  $\mathbbm{1} \leftarrow X \xrightarrow{f} Y$. We will use the name
  $\cfcod$ for the inclusion $\mathsf{inr}: Y \to \cofib{f}$.
\[\begin{tikzcd}[ampersand replacement=\& , row sep = small]
	X \&\& Y \\
	\\
	{\mathbbm{1}} \&\& {\mathsf{cofib\,f}}
	\arrow["f", from=1-1, to=1-3]
	\arrow[from=1-1, to=3-1]
	\arrow[from=3-1, to=3-3]
	\arrow["{\mathsf{cfcod}}", from=1-3, to=3-3]
	\arrow["\lrcorner"{anchor=center, pos=0.125, rotate=180}, draw=none, from=3-3, to=1-1]
\end{tikzcd}\]
\end{definition}

We will also need \emph{wedge sums}.

\begin{definition}
  Given a family of pointed types $X_i$ over an indexing type $I$, the wedge
  sum of $X_i$, denoted $\bigvee_i X_i$, is the cofibre of the
  inclusion $j : I \to \Sigma_{i:I}X_i$ defined by
  $j(i) = (i,\star_{X_i})$.
\end{definition}

From a constructive point of view, we will primarily be interested in
wedge sums indexed by a type satisfying the set-level axiom of choice
(see e.g.\ \citep[Section 3.2.3]{CavalloMsc15}, \citep[Section
5.4]{FlorisPhd}, and \citep[Section 6]{BuchholtzFavonia18} for further
discussions of this).

\begin{definition}
  A type $I$ is said to satisfy the set-level axiom of choice
  ($\mathsf{AC}_0$) if for any family of types $X_i$ over $I$, the
  canonical map
  \[
  \truncT{(i : I) \to X_i}_0 \to \left((i : I) \to \truncT{X_i}_0\right)
  \]
  is an equivalence.
\end{definition}
In particular, any finite set $I$ satisfies $\mathsf{AC}_0$. We may
now state the Eilenberg-Steenrod axioms for reduced cohomology.
\begin{definition}
  A $\bZ$-indexed family of contravariant functors
  $E^n : \mathcal{U}_\star \to \mathsf{AbGroup}$ is said to be an
  ordinary, reduced cohomology theory if the following axioms hold.
  \begin{itemize}
  \item \textbf{Suspension:} there is an isomorphism
    $E^n(X) \cong E^{n+1}(\Sigma\,X)$ natural in $X$.
  \smallskip
  \item \textbf{Exactness:} given a function $f : X \to_\star Y$,
    there is an exact sequence
  \begin{align*}
    E^n(\cofib{f}) \xrightarrow{E^n \,\cfcod} E^n(Y) \xrightarrow{E^n\,f} E^n(X)
  \end{align*}
  \item \textbf{Dimension:} for $n : \mathbb{Z}$ with $n \neq 0$, we have
  $E^n(\mathbb{S}^0) \cong \mathbbm{1}$.
  \smallskip
\item \textbf{Additivity:} for $n : \mathbb{Z}$ and a family of
  pointed types $X_i$ over a type $I$ satisfying $\mathsf{AC}_0$, the
  map $E^n(\bigvee_i X_i) \to \Pi_{i:I}(E^n(X_i))$ induced by the
  functorial action of $E^n$ on the inclusion
  $X_i \to \bigvee_{i:I} X_i$ is an isomorphism of groups.
  \end{itemize}
\end{definition}

The following theorem was originally proved in Book HoTT by
\citet{CavalloMsc15}. We spell out some of the details for our setup here. In what follows, let $\widetilde{H}^n(X,G) := \mathbbm{1}$ for $n < 0$.

\begin{therm}
  Given an abelian group $G$, the reduced cohomology theory
  $\widetilde{H}^n(-,G)$ satisfies the Eilenberg-Steenrod Axioms.
\end{therm}
\begin{proof}
  We verify that $\widetilde{H}^n(-,G)$ satisfies the four axioms.
  \begin{itemize}
  \item \textbf{Suspension:} this axiom is trivially satisfied when
    $n < -1$. When $n = -1$, we need to verify that
    $\widetilde{H}^0(\Sigma\,X,G)$ is trivial, which follows
    immediately by connectedness of $\Sigma\,X$.
    For $n \geq 0$, the suspension axiom follows
    immediately from the fact that $\Sigma$ and $\Omega$ are adjoint.
    We have
  \begin{align*}
    \widetilde{H}^{n+1}(\Sigma\,X,G) &= \truncT{\Sigma X \to_\star K(G,n+1)}_0 \\
    &\simeq \truncT{X \to_\star \Omega \,K(G,n+1)}_0 \\
    &\simeq \truncT{X \to_\star K(G,n)}_0 \\
    &= \widetilde{H}^n(X,G)
  \end{align*}
  It is easy to see that the above equivalence of types is a
  homomorphism and is natural in $X$. This concludes the proof of the suspension
  axiom.
  \smallskip
\item \textbf{Exactness:}
  Let $f : X \to Y$. We want to show that the sequence
  \begin{align*}
    \widetilde{H}^n(\cofib{f},G) \xrightarrow{\cfcod^*} \widetilde{H}^n(Y,G) \xrightarrow{f^*} \widetilde{H}^n(X,G)
  \end{align*}
  is exact.\footnote{Recall that the action on maps by $\widetilde{H}^n(-,G)$ is defined by precomposition, i.e.\ $\widetilde{H}^n(f,G) := f^*$.}
  In other words, we want to show that, given an element $\trunc{g} : \widetilde{H}^n(Y,G)$, we have $\trunc{g} \in \ker(f^*)$ iff $g \in \im{\cfcod^*}$. For the left-to-right implication, the assumption $\trunc{g} \in \ker(f^*)$ gives us a homotopy\footnote{This homotopy only exists up to propositional truncation, but since we are proving a proposition, this is not an obstacle.} $p : (x : X) \to 0_k =
  g(f(x))$. Using this, we construct a function $g_{\mathsf{cf}} :
  \cofib{f} \to K(G,n)$ by
  \begin{align*}
    g_{\mathsf{cf}}(\inl{\star}) &= 0_k \\
    g_{\mathsf{cf}}(\inr{y}) &= g(y) \\
    \ap{g_{\mathsf{cf}}}{\mathsf{push}\,x} &= p(x)
  \end{align*}
  The fact that $\cfcod^*(\trunc{g_{\mathsf{cf}}}) = \trunc{g}$
  holds by construction, and hence $\trunc{g} \in \im{\cfcod^*}$.

  \hspace{15pt} For the other direction, we simply want
  to show that the composition $f^* \circ \cfcod^*$ is
  null-homotopic. This holds by construction, since
  $\cfcod \circ f : X \to \cofib{f}$ is null-homotopic by definition
  of $\cofib{f}$.
  \smallskip
\item \textbf{Dimension:} when $n < 0$, the statement is trivial. When
  $n > 0$, we have
\begin{align*}
  \widetilde{H}^n(\mathbb{S}^0,G) = \truncT{\mathbb{S}^0 \to_\star K(G,n)}_0 \simeq \truncT{K(G,n)}_0
\end{align*}
Since $K(G,n)$ is $(n-1)$-connected, $\truncT{K(G,n)}_0$
vanishes for $n > 0$.
  \smallskip
\item \textbf{Additivity:} Again, the statement is trivial when
  $n < 0$. For the non-trivial case, the result follows from a simple
  rewriting of $\widetilde{H}^n(\bigvee_{i:I} X_i,G)$ using the elimination
  principle for $\bigvee_{i:I} X_i$:
\begin{align*}
  \widetilde{H}^n\left(\bigvee_{i:I} X_i,G\right) &:= \bigg\Vert \bigvee_{i:I} X_i \to_\star K(G,n) \bigg\Vert_0 \\
  &\simeq \Bigg\Vert \sum_{f : (\sum_{i:I} X_i) \to K(G,n)} \sum_{x_0 : K(G,n)} \left(((i : I) \to f(i,\star_{X_i}) = x_0) \times (x_0 = 0_k)\right) \Bigg\Vert_0\\
  &\simeq \Bigg\Vert \sum_{f : (\sum_{i:I} X_i) \to K(G,n)} \sum_{(x_0,p) : \sum_{x : K(G,n)} (x = 0_k)} \left((i : I) \to f(i,\star_{X_i}) = x_0\right) \Bigg\Vert_0\\
  &\simeq \Bigg\Vert\sum_{f : (\sum_{i:I} X_i) \to K(G,n)} ((i : I) \to f(i,\star_{X_i}) = 0_k) \Bigg\Vert_0 \\
  &\simeq \truncT{(i : I) \to X_i \to_\star K(G,n)}_0\\
  &\simeq (i : I) \to \truncT{X_i \to_\star K(G,n)}_0 & (\mathsf{AC}_0)\\
  &= \Pi_{i:I}\widetilde{H}^n(X_i,G)
\end{align*}
where the step from the third to fourth line is simply a deletion of the (contractible) singleton type $\sum_{x : K(G,n)} (x = 0_k)$.
The fact that this composition of equivalences indeed gives the usual homomorphism $\widetilde{H}^n\left(\bigvee_{i:I} X_i,G\right) \to \Pi_{i:I}\widetilde{H}^n(X_i,G)$ is immediate by construction. \qedhere
\end{itemize}
\end{proof}



Just like in traditional algebraic topology, many standard results and
constructions follow directly from the axioms; see e.g.\
\citep{CavalloMsc15,BuchholtzFavonia18} for concrete examples,
including the Mayer-Vietoris sequence \citep[Section~4.5]{CavalloMsc15}.




\subsection{The Mayer-Vietoris sequence}
\label{sec:mayervietoris}

The Mayer-Vietoris sequence is surprisingly easy to construct in
HoTT. It was first developed in HoTT by \citet{CavalloMsc15} for
axiomatic cohomology theories and then for the specific
$\bZ$-cohomology theory of \citet{Brunerie16} who constructed the
sequence directly. This direct construction in no way depends on the
use of $\bZ$-coefficients and the exact same proof translates to our
theory. In this short section, we will briefly recall this
construction. All notations are borrowed from \citet{Brunerie16}.

Consider three types $X,Y,Z$ with functions $f : X \to Y$ and
$g : X \to Z$. Let $D$ denote the pushout of the span
$Y \xleftarrow{f} X \xrightarrow{g} Z$. There is a natural map
$\widetilde{d} : (X \to K(G,n)) \to (D \to K(G,n+1))$ given by
\begin{align*}
  \widetilde{d}(\gamma)(\inl{y}) &:= 0_k \\
  \widetilde{d}(\gamma)(\inr{z}) &:= 0_k \\
  \ap{\widetilde{d}(\gamma)}{\push{x}} &:= \sigma_n(\gamma(x))
\end{align*}
This lifts to a map $d : H^n(X,G) \to H^{n+1}(D,G)$ by
$d(\trunc{\gamma}) = \trunc{\widetilde{d}(\gamma)}$. There is also a
natural map $\Delta : H^{n}(Y) \times H^{n}(Z) \to H^n(X,G)$ given by
\[
  \Delta(\alpha,\beta) := f^*(\alpha) -_h g^*(\beta)
\]
Finally, we have a map $i : H^n(D,G) \to H^n(Y,G) \times H^n(Z,G)$ given by
\[
  i(\delta) := (\mathsf{inl}^*(\delta), \mathsf{inr}^*(\delta))
\]

\begin{therm}[Mayer-Vietoris sequence]
  The families of functions $i$, $\Delta$ and $d$ give rise to a long
  exact sequence on the form:
\[\begin{tikzcd}[ampersand replacement=\&]
	{H^0(D,G)} \&\& {H^0(Y,G) \times H^0(Z,G)} \&\& {H^0(X,G)} \\
	{H^1(D,G)} \&\& {H^1(Y,G) \times H^1(Z,G)} \&\& {H^1(X,G)} \\
	{H^2(D,G)} \&\& \dots
	\arrow["i", from=1-1, to=1-3]
	\arrow["\Delta", from=1-3, to=1-5]
	\arrow["d"', from=1-5, to=2-1,bend right=2]
	\arrow["i",,from=2-1, to=2-3]
	\arrow["\Delta", from=2-3, to=2-5]
	\arrow["d"', from=2-5, to=3-1,bend right=2]
	\arrow["i",,from=3-1, to=3-3]
\end{tikzcd}\]
\end{therm}
\begin{proof}
  Straightforward generalisation of
  \citet[Proposition~5.2.2]{Brunerie16}.
\end{proof}

This is useful to compute cohomology of spaces defined as pushouts,
e.g.\ $\mathbb{C}P^2$ (see Section \ref{sec:cp2}). Let us now consider another useful long exact
sequence developed in HoTT by \citet{Brunerie16} known as the \emph{Gysin sequence}.

\subsection{The Thom isomorphism and the Gysin sequence}
\label{sec:gysin}

\citet{Brunerie16} also introduced, for the first time in HoTT, the
Thom isomorphism and the Gysin sequence. While Brunerie's original
constructions concerned integral cohomology, we may easily generalise
them to cohomology with coefficients in an arbitrary commutative ring
$R$. In this short section, we will give a quick review of Brunerie's
proof/construction---here generalised to arbitrary coefficients
(which has no major effect on the proofs).

Let $\alpha_n : \mathbb{S}^n \to_\star K(R,n)$ denote the image of
$1_r : R$ under the equivalence
\begin{align*}
  R \simeq \Omega^n K(R,n) \simeq (\mathbb{S}^n \to_\star K(R,n))
\end{align*}
This induces a family of equivalences
$g^i : K(R,i) \to (\mathbb{S}^n \to_\star K(R,i + n))$ defined by
\begin{align*}
  g^i(x) := y\mapsto x \smile_k \alpha_n(y)
\end{align*}
Now, let $B$ be a $0$-connected, pointed type and suppose we have
fibration $Q : B \to \mathsf{Type}_\star$ with
$\psi : Q(\star_B) \simeq_\star \mathbb{S}^n$ and a $B$-indexed family
of functions $c_b : Q(b) \to_\star K(R,n)$ with
$c_{\star_B} = \alpha_n \circ \psi$. We remark that such a family
automatically exists when $B$ is further assumed to be
$1$-connected. We may use this to define a $B$-relative version of
$g^i$:
\begin{align*}
  g_b^i &: K(R,i) \to (Q(b) \to_\star K(R,i+n))\\
  g_b^i(x) &:= y \mapsto x \smile_k c_b(y)
\end{align*}
Showing that this map is an equivalence is straightforward: for
connectedness reasons, we only need to do so when $b$ is $\star_B$, in
which case we have $g_{\star_B}^i(x) = g^i(x) \circ \psi$ which we
already know defines an equivalence.

The fact that $g_b^i$ is an equivalence means, in particular, that
the following map is one too.
\begin{align*}
  \varphi^i &: (B \to K(R,i)) \to ((b : B) \to Q(b) \to_\star K(R,i+n))\\
  \varphi^i(f) &:= (b,q \mapsto f(b) \smile c_b(q))
\end{align*}
We instantiate the above with some concrete spaces. Let
$P : B \to \mathsf{Type}_\star$ be a fibration with
$P(\star_B) \simeq \mathbb{S}^{n-1}$ and let
$Q(b)$ be the suspension of $P(b)$, i.e.\ $Q(b) := \Sigma\,(P(b))$.
Suppose also that we have a family
$c_b : (Q(b) \to_\star K(R,n))$ as above (again, this is automatic if
$B$ is $1$-connected). Let $E := \Sigma_{b:B}P(b)$ be the total
space of $P$ and $\widetilde{E}$ be the associated \emph{Thom space}, i.e.\ the cofibre of the projection map
$E \to B$. It is easy to see that
$((b : B) \to Q(b) \to_\star K(R,k)) \simeq (\widetilde{E} \to_\star
K(R,k))$. Combining this with $\varphi^i$, we get an isomorphism
\begin{align*}
  \phi^i : H^i(B,R) \cong \widetilde{H}^{i+ n}(\widetilde{E},R)
\end{align*}
Finally, we may consider the long exact sequence related to the
pushout defining $\widetilde{E}$ (the first row in the following
diagram). Substituting along $\phi^{i-n}$ yields the Gysin sequence
(the second row in the diagram).
\[\begin{tikzcd}[ampersand replacement=\&]
	\dots \& {H^{i-1}(E,R)} \& {\widetilde{H}^i(\widetilde{E},R)} \& {H^i(B,R)} \& {H^i(E,R)} \& \dots \\
	\dots \& {H^{i-1}(E,R)} \& {H^{i-n}(B,R)} \& {H^i(B,R)} \& {H^i(E,R)} \& \dots
	\arrow["\phi^{i-n}", from=2-3, to=1-3]
	\arrow[from=1-2, to=1-3]
	\arrow[dashed, from=2-2, to=2-3]
	\arrow[Rightarrow, no head, from=2-2, to=1-2]
	\arrow[from=1-1, to=1-2]
	\arrow[from=2-1, to=2-2]
	\arrow["{(-)\smile e}"', from=2-3, to=2-4]
	\arrow[from=1-3, to=1-4]
	\arrow[Rightarrow, no head, from=1-4, to=2-4]
	\arrow[from=1-4, to=1-5]
	\arrow[from=2-4, to=2-5]
	\arrow[from=1-5, to=1-6]
	\arrow[from=2-5, to=2-6]
	\arrow[Rightarrow, no head, from=1-5, to=2-5]
      \end{tikzcd}\]
    Here the dashed arrow is defined as the composition $H^{i-1}(E,R) \to \widetilde{H}^i(\widetilde{E},R) \xrightarrow{(\phi^{i-n})^{-1}} H^{i-n}(B,R)$.

Above, the class $e : H^n(B,R)$ is defined by
$e := \trunc{b \mapsto c_b(\mathsf{south})}$. The commutativity of the
centre square follows immediately by construction of $\phi^{i-n}$. For
the sake of completeness, let us state the existence of the Gysin
sequence as a theorem.
\begin{therm}[The Gysin Sequence]\label{thm:Gysin}
  Let $B$ be $0$-connected and $P : B \to \mathsf{Type}$ be a
  fibration equipped with an equivalence
  $\psi : P(\star_B) \simeq \mathbb{S}^{n-1}$. Assume that the
  following condition is satisfied:
  \begin{enumerate}[($\ast$)]
  \item There is a $B$-indexed family
    $c_b : \Sigma(P(b)) \to_\star K(R,{n})$ s.t.\ $c_{\star_{B}}$ is
    the image of $1_R : R$ under the composition of isomorphisms
    \[
      R \cong H^n(\mathbb{S}^n,R)\, \overset{(\Sigma(\psi))^*}{\cong}\,
      H^n(\Sigma(P(\star_B)),R)
      \]\label{gysin-cond}
  \end{enumerate}
  Let $E$ denote the total space of $P$ and define $e : H^n(B,R)$ by $e :=
\trunc{b \mapsto c_b(\mathsf{south})}$. We get a long exact sequence
\[\begin{tikzcd}[ampersand replacement=\&]
	\dots \& {H^{i-1}(E,R)} \& {H^{i-n}(B,R)} \& {H^i(B,R)} \& {H^i(E,R)} \& \dots
	\arrow[from=1-2, to=1-3]
	\arrow["{(-)\smile e}", from=1-3, to=1-4]
        \arrow[from=1-1,to=1-2]
	\arrow[from=1-4, to=1-5]
	\arrow[from=1-5, to=1-6]
\end{tikzcd}\]
Moreover, condition \hyperlink{gysin-cond}{($\ast$)} is always satisfied when $B$ is $1$-connected.
\end{therm}

The Gysin sequence was for instance crucially used
by~\citet{Brunerie16} to compute the integral cohomology \emph{ring} of
$\mathbb{C}P^2$ (see also Section \ref{sec:cp2}). We will also use this sequence to compute the cohomology ring of the infinite real projective space in Section \ref{sec:rpinf}.

\section{Computations of cohomology groups and rings}
\label{sec:direct}

In this section we will compute some cohomology groups and rings of
some common spaces. All of these groups/rings will be very familiar
to the traditional mathematician. Our intent is to showcase how these
computations are carried out in HoTT. What is particularly interesting
here is that the computations can often be done very directly, without
relying on more general machinery (e.g.\ the Mayer-Vietoris sequence,
spectral sequences, etc.). We have two reasons for preferring direct proofs whenever possible. Firstly,
we think that the existence of such proofs in many cases showcases the strengths of the synthetic approach to cohomology theory in HoTT. Many proofs become incredibly
short and rely solely on spelling out the computation rule for the
space in question. In particular, we can carry out these proofs
without any homological algebra. Secondly, by carefully choosing the
direct proofs, we often produce proof terms which are far simpler than
those produced by more advanced machinery. This means that we can use
the isomorphisms we construct here to make concrete computations
(i.e.\ normalisation of closed terms) in a constructive proof
assistant such as \CubicalAgda. We will discuss this more in
Section~\ref{sec:computations}.

Let us start by computing the cohomology groups of some special
types.
\begin{proposition}
  \label{prop:Hn-Unit}
  $H^n(\mathbbm{1},G) \cong \begin{cases} G & n = 0 \\ \mathbbm{1} & n \geq 1 \end{cases}$
\end{proposition}
\begin{proof}
  We have $H^n(\mathbbm{1}) := \truncT{\mathbbm{1} \to K(G,n)}_0 \simeq \truncT{K(G,n)}_0$.
This is $\truncT{G}_0 \simeq G$ when $n = 0$ and contractible when $n
> 0$ since $K(G,n)$ is $(n-1)$-connected.
\end{proof}
We also include note the following easy lemma
\begin{lemma}
  \label{lem:Hn-Trunc}
  For any type $X$, we have $H^n(X,G) \cong H^n(\truncT{X}_n,G)$.
\end{lemma}
\begin{proof}
  Using that $K(G,n)$ is $n$-truncated, we get, by the universal
  property of $n$-truncation, that
  \begin{align*}
    H^n(X,G) = \truncT{X \to K(G,n)}_0 \simeq \truncT{\truncT{X}_n \to K(G,n)}_0 = H^n(\truncT{X}_n,G)
  \end{align*}
  The fact that this equivalence is a homomorphism is trivial.
\end{proof}

Using this we can easily compute the zeroth cohomology group of a
$0$-connected type (and hence, any $n$-connected type with
$n \geq 0$).
\begin{proposition}\label{prop:cohom-connected}
  For any $0$-connected type $X$, we have $H^0(X,G) \cong G$.
\end{proposition}
\begin{proof}
  By Lemma \ref{lem:Hn-Trunc}, we have $H^0(X,G) =
  H^0(\truncT{X}_0,G)$. But $\truncT{X}_0 \simeq \mathbbm{1}$ since
  $X$ is $0$-connected. Thus, by Proposition \ref{prop:Hn-Unit}, the zeroth
  cohomology group is isomorphic to $G$.
\end{proof}


\subsection{Spheres}\label{subsec:sphere}

We now have what we need to compute the cohomology of the
$n$-sphere. Let us start with $n = 1$.

\begin{proposition}
    \label{prop:H1-S1}
    $H^1(\mathbb{S}^1,G) \cong G$.
  \end{proposition}
  \begin{proof}
    A function $\mathbb{S}^1 \to K(G,1)$ is uniquely determined by its
    action on the canonical base point and the canonical loop. Hence, it corresponds to choices
    of a point $x : K(G,1)$ and a loop $x = x$. In other words,
    $(\mathbb{S}^1 \to K(G,1)) \simeq \Sigma_{x : K(G,1)} (x =
    x)$. But $K(G,1)$ is homogeneous, so
    $(x = x) \simeq \Omega K(G,1)$. Thus, we have
    \begin{align*}
      H^1(\mathbb{S}^1,G) &= \truncT{\mathbb{S}^1 \to K(G,1)}_0 \\
      &\simeq \truncT{\Sigma_{x : K(G,1)} (x = x)}_0\\
      &\simeq \truncT{K(G,1) \times \Omega K(G,1)}_0 \\
      &\simeq \truncT{K(G,1)}_0 \times G
    \end{align*}
    Now, $K(G,1)$ is $0$-connected, so $\truncT{K(G,1)}_0$
    vanishes. Thus, we have $H^1(\mathbb{S}^1,G) \simeq G$.

    We need to verify that this isomorphism is a homomorphism. It is
    easier to show that its inverse is. Let us call it $\phi$. By
    definition, $\phi(g) = \trunc{f_g}$ where
    $f_g : \mathbb{S}^1 \to K(G,1)$ is the map sending $\cbase$ to
    $\star$ and $\Loop$ to $\LoopG{g}$. Thus, we only need to verify
    that for $g_1,g_2 : G$ and $x : \mathbb{S}^1$, we have
    $f_{g_1}(x) +_k f_{g_2}(x) = f_{g_1 + g_2}(x)$. We do this by
    induction on $x$. When $x$ is $\cbase$, the goal holds by
    $\refl$. For the action on $\Loop$, we need to show that
    \begin{align*}
      \app{+_k}{\LoopG{g_1}}{\LoopG{g_2}} = \LoopG{(g_1 + g_2)}
    \end{align*}
    which holds by Proposition \ref{prop:ap-+} and functoriality of $\LoopG{}$.
  \end{proof}
  It is easy to generalise this to get all non-trivial cohomology
  groups of $\mathbb{S}^n$ for $n \geq 1$.
  \begin{proposition}
    For $n \geq 1$, we have $H^n(\mathbb{S}^n,G) \cong G$.
  \end{proposition}
  \begin{proof}
    The case $n = 1$ is Proposition~\ref{prop:H1-S1}. For $n > 1$, we simply apply the suspension
    isomorphism inductively.
  \end{proof}
The vanishing groups can also be handled directly.
\begin{proposition}
  For $n \neq m$ and $n \geq 1$ have $H^n(\mathbb{S}^m,G) \cong \mathbbm{1}$.
\end{proposition}
\begin{proof}
  Let us first consider the case $n > m$. We induct
  on $m$. When $m = 0$, we have
  \begin{align*}
    H^n(\mathbb{S}^0) = \truncT{\mathbb{S}^0 \to K(G,n)}_0 \simeq \truncT{K(G,n) \times K(G,n)}_0
  \end{align*}
  which vanishes since $K(G,n)$ is $(n-1)$-connected and $n > m = 0$. For
  larger $m$, we know by suspension that $H^n(\mathbb{S}^m) \cong
  H^{n-1}(\mathbb{S}^{m-1})$ which vanishes by the induction
  hypothesis.

  When $n<m$ we have
  \begin{align*}
    H^n(\mathbb{S}^m,G) \simeq H^n(\truncT{\mathbb{S}^m}_n,G)
  \end{align*}
  which vanishes because $\mathbb{S}^m$ is $n \leq (m-1)$-connected and $n \geq 1$.
\end{proof}

Knowing these groups, the cohomology ring of the spheres is
easily computed.

\begin{proposition}\label{prop:H*-Sn}
    Given any ring $R$ and $n \geq 1$, we have
    $H^*(\mathbb{S}^n,R) \cong R[x]/(x^2)$.
  \end{proposition}
\begin{proof}
  The only non-trivial cohomology groups are
  $H^0(\mathbb{S}^n,R) \cong H^n(\mathbb{S}^n,R) \cong R$. Thus, the
  cup product can only be non-trivial in dimensions $0 \times n$ and
  $n \times 0$. In these dimension, the cup product is simply left
  respectively right $R$-multiplication. This gives the desired
  isomorphism by letting constants in $H^*(\mathbb{S}^n,R)$ represent
  elements coming from $H^0(\mathbb{S}^n,R)$ and elements of degree
  one represent elements coming from $H^n(\mathbb{S}^n,R)$.
\end{proof}

\subsection{The torus}

Now that we know the cohomology of spheres, let us investigate something slightly more advanced: the torus.
\begin{definition}[The torus]\label{def:torus}
  We define $\mathbb{T}^2$ as the HIT generated by the following constructors:\\
\begin{minipage}{.46\textwidth}
\begin{itemize}
\item $\star : \mathbb{T}^2$
\item $\ell_1,\ell_2 : \star = \star$
\item $\mathsf{sq} :  \Square{l_2}{l_2}{l_1}{l_1}$, i.e.\ a filler of the square:
\end{itemize}
\end{minipage}
  \begin{minipage}{.39\textwidth}
    \begin{tikzcd}[ampersand replacement=\&]
	\star \& \star \\
	\star \& \star
	\arrow["{\ell_2}", from=1-1, to=1-2]
	\arrow["{\ell_2}"', from=2-1, to=2-2]
	\arrow["{\ell_1}"', from=2-2, to=1-2]
	\arrow["{\ell_1}", from=2-1, to=1-1]
    \end{tikzcd}
  \end{minipage}
\end{definition}
It is well-known that $\mathbb{T}^2 \simeq \mathbb{S}^{1} \times \mathbb{S}^1$.\footnote{In plain HoTT, this was first fully proved by \citet{Sojakova16}. In cubical type theory, this fact is a triviality \citep[Section 3]{cubicalsynthetic}.}. We will see that this fact makes the computation of the cohomology groups of $\mathbb{T}^2$ rather direct. We follow the proof strategy of~\citet[Propositions 20--21]{BLM}, which directly translates from integral cohomology to our setting.

\begin{proposition}
  \[
  H^n(\mathbb{T}^2,G) \cong \begin{cases}
    G & \text{ if $n=0$ or $n=2$} \\
    G \times G & \text{ if $n=1$}\\
    \mathbbm{1} & \text{ otherwise}
    \end{cases}
  \]
\end{proposition}
\begin{proof}
  The case when $n=0$ follows by
  Proposition~\ref{prop:cohom-connected} since $\mathbb{T}^2$ is
  $0$-connected.
  For the case $n \geq 1$, let us inspect the underlying function space of $H^n(\mathbb{T}^2 , G)$. We have
  \begin{align*}
    (\mathbb{T}^2 \to K(G,n)) &\simeq (\mathbb{S}^1 \times \mathbb{S}^1 \to K(G,n))  \\
    &\simeq \mathbb{S}^1 \to (\mathbb{S}^1 \to K(G,n)) \\
    &\simeq \sum_{f :\,\mathbb{S}^1\to K(G,n)} (f = f)\\
    &\simeq \sum_{f :\,\mathbb{S}^1\to K(G,n)} ((x : \mathbb{S}^1) \to \Omega(K(G,n) , f(x)))\\
    &\simeq (\mathbb{S}^1\to K(G,n)) \times (\mathbb{S}^1 \to \Omega(K(G,n)))\\
    &\simeq (\mathbb{S}^1\to K(G,n)) \times (\mathbb{S}^1 \to K(G,n-1))\\
  \end{align*}
  where the step from line 3 to 4 is function extensionality and the
  step from line 4 to 5 follows by homogeneity of $K(G,n)$.  Under set
  truncation, the last type in this chain of equivalences is simply
  $H^n(S^1,G) \times H^{n-1}(S^1,G)$, which we know from
  Section~\ref{subsec:sphere} is simply $G \times G$ when $n = 1$, $G$
  when $n = 2$, and trivial for higher $n$. The fact that this map is a homomorphism follows by a similar argument to that of the corresponding statement in the proof of Proposition~\ref{prop:H1-S1} .
\end{proof}

\subsection{The real projective plane and the Klein bottle}

Let us consider two somewhat more complex (but closely related) examples: the
real projective plane and the Klein bottle. In HoTT,
we define the real projective plane, $\mathbb{R}P^2$, and the Klein
bottle, $K^2$, by HITs corresponding to their usual folding diagrams:
\begin{definition}[The real project plane]\label{def:RP2}
  We define $\mathbb{R}P^2$ as the HIT generated by the following constructors:\\
\begin{minipage}{.46\textwidth}
\begin{itemize}
\item $\star : \mathbb{R}P^2$
\item $\ell : \star = \star$
\item $\mathsf{sq} : \ell^{-1} = \ell$, i.e.\ a filler of the square:
\end{itemize}
\end{minipage}
\begin{minipage}{.54\textwidth}
\[\begin{tikzcd}[ampersand replacement=\&]
	\star \& \star \\
	\star \& \star
	\arrow["\ell",from=2-2,to=2-1]  
        \arrow["\ell",from=1-1,to=1-2]  
	\arrow[Rightarrow, no head, from=2-2, to=1-2]
	\arrow[Rightarrow, no head, from=1-1, to=2-1]
\end{tikzcd}\]
\end{minipage}
\end{definition}
\begin{definition}[The Klein bottle]\label{def:K2}
  We define $K^2$ as the HIT generated by the following
  constructors:\linebreak
  \begin{minipage}{.65\textwidth}
  \begin{itemize}
  \item $\star : K^2$
  \item $\ell_1,\ell_2 : \star = \star$
  \item $\mathsf{sq} : \Square{l_2}{l_2}{l_1^{-1}}{l_1}$, i.e.\ a filler of the
    square:
  \end{itemize}
  \end{minipage}
  \begin{minipage}{.39\textwidth}
    \begin{tikzcd}[ampersand replacement=\&]
	\star \& \star \\
	\star \& \star
	\arrow["{\ell_2}", from=1-1, to=1-2]
	\arrow["{\ell_2}"', from=2-1, to=2-2]
	\arrow["{\ell_1}"', from=2-2, to=1-2]
	\arrow["{\ell_1}"', from=1-1, to=2-1]
    \end{tikzcd}
  \end{minipage}
\end{definition}

Let us first tackle the cohomology of $\mathbb{R}P^2$. It is an easy
lemma that $\RP$ is $0$-connected, and hence $H^0(\RP,G) \cong G$. In
order to compute its higher cohomology groups, we will need to
introduce some new constructions: for an abelian group $G$ and an
integer $n \geq 0$, let $G[n]$ denote $n$-torsion subgroup of $G$,
i.e.\ the kernel of the map $n \cdot (-) : G \to G$. Let $G/n$ denote
the cokernel of the same homomorphism; that is, the group $G/\sim$
where $\sim$ is the relation $x^n = 0$.  We denote the quotient map
$G \to G/n$ by $[-]$.
\begin{proposition}\label{prop:H1-RP2}
  $H^1(\mathbb{R}P^2,G) \cong G[2]$.
\end{proposition}
\begin{proof}
  The elimination principle for $\RP$ tells us that
  \[(\RP \to K(G,1)) \simeq \sum_{x : K(G,1)}\sum_{p : x = x} (p^{-1} = p) \simeq \sum_{x : K(G,1)}\sum_{p : x = x} (p \cdot p = \refl) \]
  Since $K(G,n)$ is homogeneous, this is equivalent to the type
  \[{K(G,1)} \times \sum_{p : \Omega K(G,1)} (p \cdot p = \refl) \]
  which under the equivalence $\Omega K(G,1) \simeq G$ is equivalent
  to the type
  \[{K(G,1)} \times \sum_{g : G} (g + g = 0_G) \]
  i.e.\  $K(G,1) \times G[2]$. Thus, we get
\begin{align*}
  H^1(\RP,G) = \truncT{\RP \to K(G,1)}_0 &\simeq \truncT{K(G,1) \times G[2] }_0 \\
  &\simeq \truncT{K(G,1)}_0 \times G[2] \\
  &\simeq G[2]
\end{align*}
We verify that the inverse of this equivalence is a homomorphism. The
inverse $\psi : G[2] \to H^1(\RP,G)$ is given by $\psi(g) = f_g$ where
$f_g : \RP \to K(G,1)$ sends $\star$ to $0_k$ and $\ell$ to
$\LoopG{g}$. We do not need to worry about its action on $\mathsf{sq}$:
for truncation reasons this is not data but a property. We need to
verify that $f_{g_1+g_2}(x) = f_{g_1}(x) +_k f_{g_2}(x)$ for $x : \RP$
and $g_1,g_2 : G$. When $x$ is $\star$, this holds by $\refl$. For the
action on $\ell$, we need to verify that
\begin{align*}
  \LoopG{(g_1+g_2)} = \app{+_k}{\LoopG{g_1}}{\LoopG{g_2}}
\end{align*}
which we already verified in the proof of Proposition
\ref{prop:H1-S1}. This concludes the proof.
\end{proof}

\begin{proposition}\label{prop:H2-RP2}
  $H^2(\RP,G) \cong G/2$.
\end{proposition}
\begin{proof} Like
  in the proof of Proposition \ref{prop:H1-RP2}, we have
  \begin{align*}
    (\RP \to K(G,2)) &\simeq \sum_{x : K(G,2)}\sum_{p : x = x} {p \cdot p = \refl} \\
    &\simeq K(G,2) \times \sum_{p : \Omega K(G,2)} {p \cdot p = \refl} \\
    &\simeq K(G,2) \times \sum_{x : K(G,1)}{x +_k x = 0_k}
  \end{align*}
  Under set truncation, the $K(G,2)$-component vanishes which leaves
  us with the type $\truncT{\sum_{x : K(G,1)}{x +_k x = 0_k}}_0$. Let
  us construct a map $\phi : \truncT{\sum_{x : K(G,1)}{x +_k x =
      0_k}}_0 \to G/2$ by means of a curried map $\phi_x : x +_k x =
  0_k \to G/2$ depending on $x : K(G,1)$. We may define it by the
  set-elimination rule for $K(G,1)$. When $x$ is $0_k$, what is
  desired is a map $\phi_{\star} : \Omega K(G,1) \to G/2$. This map
  may simply be defined by $\phi_{\star} := [-] \circ
  \sigma_0^{-1}$. We now need to describe the action on
  $\LoopG{g}$. This amounts to providing a path $[g] = [g' + g + g']$,
  which of course trivially exists in $G/2$. This defines
  $\phi(\trunc{x,p}) := \phi_x(p)$.

  The inverse
  $\psi : G/2 \to \truncT{\Sigma_{x : K(G,1)} (x+_k x = 0_k)}_0$ is defined on
  canonical elements by $\psi[g] = (0_k,\LoopG{g})$. In order to show
  that this is well-defined, we need to show that
\begin{align*}
  \trunc{(0_k, \refl)} = \trunc{(0_k,\LoopG{(g+_G g)})}
\end{align*}
This is done component-wise. For first component, we need to provide a path $0_k = 0_k$. In fact, the right choice here it \emph{not}  $\refl$ but instead
$\LoopG{g}$. Showing that the second components agree
now amounts to providing a filler of the square
\[\begin{tikzcd}[ampersand replacement=\&]
	{0_k} \&\& {0_k} \\
	\\
	{0_k} \&\& {0_k}
	\arrow["{\LoopG{(g+_G g)}}", from=1-1, to=1-3]
	\arrow["({\app{+_k}{\LoopG{g}}{\LoopG{g}})^{-1}}", from=3-1, to=1-1]
	\arrow["{\mathsf{refl}}"', Rightarrow, no head, from=3-1, to=3-3]
	\arrow["{\mathsf{refl}}"', Rightarrow, no head, from=3-3, to=1-3]
\end{tikzcd}\]
which, in turn, is equivalent to proving that $\LoopG{(g +_G g)} =
\app{+_k}{\LoopG{g}}{\LoopG{g}}$, which we have already done in~the proofs of~Propositions~\ref{prop:H1-S1}~and~\ref{prop:H1-RP2}.

Finally, let us prove that the maps cancel. Verifying that $\phi(\psi(x)) =
x$ for $x:G/2$ is very direct: since the statement is a proposition,
it suffices to to this when $x$ is on the form $[g]$. In this case, we
have $\phi(\psi[g]) := \phi_{0_k}(\LoopG{g}) :=
\sigma_0^{-1}(\sigma_0(g)) = g$.

For the other direction, the fact that we are proving a proposition
means that it suffices to show the cancellation for elements of
$\truncT{\Sigma_{x : K(G,1)} (x +_k x = 0_k)}_0$ on the
form $\trunc{(0_k,p)}$ with $p : \Omega K(G,1)$. We get
$\psi(\phi\,\trunc{0_k,p}) := \psi(\sigma_0^{-1}(p)) = \trunc{0_k,\sigma_0(\sigma_0^{-1}(p))} = \trunc{0_k,p}$.
Thus, we have shown:
\begin{align*}
H^2(\RP,G) &\simeq \bigg\Vert K(G,2) \times \sum_{x : K(G,1)}{x +_k x = 0_k} \bigg\Vert_0 \\
&\simeq \bigg\Vert \sum_{x : K(G,1)}{x +_k x = 0_k} \bigg\Vert_0\\
&\simeq G/2
\end{align*}
The argument for why this equivalence is a homomorphism is technical but similar to
the argument given in previous proofs.
\end{proof}

\begin{proposition}\label{prop:H-high-RP2}
  $H^n(\RP,G)$ is trivial for $n > 2$.
\end{proposition}
\begin{proof}
  As we have seen in the two previous proofs:
  \begin{align*}
    H^n(\RP,G) &\simeq \bigg\Vert K(G,n) \times \sum_{x : K(G,n-1)} x +_k x = 0_k \bigg\Vert_0 \\
    &\simeq \bigg\Vert \sum_{x : K(G,n-1)} x +_k x = 0_k \bigg\Vert_0
  \end{align*}
  We may move in the truncation one step further to get
  \begin{align*}
    \bigg\Vert \sum_{x : K(G,n-1)} \truncT{x +_k x = 0_k}_0\bigg\Vert_0
  \end{align*}
  Since $n > 2$, any path space over $K(G,n-1)$ is (at least)
  $0$-connected. Thus, the truncation kills the space $(x +_k x =
  0_k)$. Hence, the above type is simply $\truncT{K(G,n-1)}_0$, which
  also vanishes due to connectedness.
\end{proof}

This gives us all the cohomology groups of $\RP$. We will now see how
this, in a rather direct way, gives us all the cohomology groups of
$K^2$ as well. Trivially, $H^0(K^2,G) \cong G$ for connectedness
reasons.

\begin{proposition}\label{prop:Hn-K2}
  $H^n(K^2,G) \cong \begin{cases} H^n(\RP,G) & n \neq 1 \\ G \times H^1(\RP,G) & n = 1 \end{cases}$
\end{proposition}
\begin{proof}
  Consider the function space $(K^2 \to K(G,n))$:
  \begin{align*}
    (K^2 \to K(G,n)) \simeq \sum_{x : K(G,n)}\sum_{p,q: x = x}(p
    \cdot q \cdot p = q)
  \end{align*}
  Again, we may use homogeneity of $K(G,n)$ to show that this is
  equivalent to the type $$K(G,n) \times \sum_{p,q: \Omega K(G,n)}(p
  \cdot q \cdot p = q)$$ We now consider the path space $p \cdot q
  \cdot p = q$ for $p,q : \Omega K(G,n)$. By commutativity of
  $\Omega\,K(G,n)$, composing both sides by $q^{-1}$ will cancel out
  $q$ everywhere. Thus, the type is simply $p\cdot p = \refl$. Hence, we have
  \begin{align*}
    H^n(K^2,G) &:= \truncT{K^2 \to K(G,n)}_0 \\
               &\simeq \bigg\Vert K(G,n) \times \Omega K(G,n) \times \sum_{p : \Omega K(G,n)} (p \cdot p \equiv \refl)\bigg\Vert_0
  \end{align*}
  Excluding the $\Omega K(G,n)$ component, this is precisely the
  characterisation of $H^n(\RP,G)$. Thus, the above type is equivalent
  to
  \begin{align*}
    \truncT{\Omega K(G,n)}_0 \times H^n(\RP,G)
  \end{align*}
  The factor $\truncT{\Omega K(G,n)}_0$ vanishes for $n > 1$. When
  $n = 1$, we have $\truncT{\Omega K(G,1)}_0 \simeq \truncT{G}_0
  \simeq G$. Verifying that this is a homomorphism is straightforward
  but somewhat technical.
\end{proof}

Computing the integral cohomology ring of $\RP$ is straightforward.
Naturally, the following proposition easily generalises to $H^*(\RP,R)$
for any $2$-torsion free ring $R$.

\begin{proposition}
  $H^*(\RP,\bZ) \cong \bZ[x]/(2x,x^2)$.
\end{proposition}
\begin{proof}
  By Proposition \ref{prop:H1-RP2}, we have
  $H^1(\RP,\bZ) \cong \bZ[2]$. Since $\bZ$ has no non-trivial torsion
  elements, this tells us that $H^1(\RP,\bZ)$ is trivial. Hence, we
  only have as non-trivial cohomology groups the groups
  $H^0(\RP,\bZ) \cong \bZ$ and $H^2(\RP,\bZ) \cong \bZTwo$. The proof
  now proceeds just like that of Proposition \ref{prop:H*-Sn},
  representing elements from $H^0(\RP,\bZ)$ and elements from
  $H^2(\RP,\bZ)$ respectively by degree-$0$ and degree-$1$ elements in
  $H^*(\RP,\bZ)$. The $2x$-quotient comes from the $2$-torsion in
  $H^2(\RP,\bZ)$.
\end{proof}

Let us now turn to the Klein bottle. For the coming results, the
following construction will be crucial.
\begin{definition}
  There is a homotopy
  $$\killDiag : (p : \Omega K(G,n))  \to  \app{\smile_k}{p}{p} = 0_k$$
  defined by the composite path\normalfont{
\[\begin{tikzcd}[ampersand replacement=\&]
	{\mathsf{ap}^2_{\smile_k}(p,p)} \&\&\& {\mathsf{ap}_{(-)\smile_k 0_k}(p)\cdot \mathsf{ap}_{0_k\smile_k (-)}(p)} \&\& \refl
	\arrow["{h(p)}", from=1-4, to=1-6]
	\arrow["{\textsf{ap}_{\smile_k}^2\textsf{-funct}(p,p)}", from=1-1, to=1-4]
      \end{tikzcd}\]}
   \!Here $h(p)$ is given by pointwise application of the right and
    left annihilation laws for $\smile_k$ using that they agree on
    $0_k$.
\end{definition}

We will rely on this construction to kill off certain cup products on
\RP and $K^2$. For instance, it makes the following computation easy.

\begin{proposition}
  $H^*(K^2,\bZ) \cong \bZ[x,y]/(2y,x^2,y^2,xy)$.
\end{proposition}
\begin{proof}
  Using Proposition \ref{prop:Hn-K2} and the preceding results, we know that the
  only non-trivial cohomology groups of $K^2$ are
  \begin{align*}
    H^0(K^2,\bZ) &\cong \bZ \\
    H^1(K^2,\bZ) &\cong \bZ \times \bZ[2] \cong \bZ \\
    H^2(K^2,\bZ) & \cong \bZTwo
  \end{align*}
  Consider first the ring $\bZ[x,y]$. Using the above isomorphisms, we
  may naively let degree-$0$ elements in $\bZ[x,y]$ represent elements
  from $H^0(K^2,\bZ)$, degree-$1$ elements with $x$-term represent
  elements from $H^1(K^2,\bZ)$ and degree-$1$ elements with $y$ term represent
  elements from $H^2(K^2 ,\bZ)$.
  Under this identification, we see
  that the relations described by the ideal $(2y,xy,y^2)$ are
  satisfied. Hence, we have a map $\bZ[x,y]/(2y,xy,y^2) \to
  H^1(K^2,\bZ)$.  In order to lift this to an isomorphism
  $\bZ[x,y]/(2y,x^2,y^2,xy) \cong H^*(K^2,\bZ)$, we need to show that
  the cup product is trivial in dimension $1 \times 1$. It suffices to
  show that it vanishes on generators. We can easily define a map
  $\alpha : K^2 \to K(1,\mathbb{Z})$ such that $\trunc{\alpha}$
  generates $H^1(K^2,\bZ)$. We define it by
\begin{align*}
  \alpha(\star) &:= 0_k \\
  \ap{\alpha}{\ell_1} &:= \refl \\
  \ap{\alpha}{\ell_2} &:= \LoopG{1}
\end{align*}
where, $1 : \bZ$ denotes the generator. The final step of the
definition, i.e.\ the action of $\alpha$ on $\mathsf{sq}$, corresponds
to providing a proof that $\LoopG{1} = \LoopG{1}$, which we do by
$\refl$. Let $x : K^2$. We show that $\alpha(x) \smile_k \alpha(x) =
0_k$ by induction on $x$. We prove $\alpha(\star)\smile_k \alpha(\star) = 0_k $ by $\refl$
and $\app{\smile_k}{\ap{\alpha}{\ell_2}}{\ap{\alpha}{\ell_2}} = \refl$ by $\refl$. For
$\app{\smile_k}{\ap{\alpha}{\ell_1}}{\ap{\alpha}{\ell_1}}$, we need to provide a path
$\app{\smile_k}{\LoopG{1}}{\LoopG{1}} = \refl$, which we give by
$\killDiag(\LoopG{1})$. Finally, for the action of $\mathsf{sq}$, we
need to provide a filler of the square
\[\begin{tikzcd}[ampersand replacement=\&]
	{\app{\smile_k}{\LoopG{1}}{\LoopG{1}}} \&\& {\refl} \\
	\\
	{\app{\smile_k}{\LoopG{1}}{\LoopG{1}}} \&\& {\refl}
	\arrow["{\killDiag(\LoopG{1})}", from=1-1, to=1-3]
	\arrow["{\mathsf{refl}}"', Rightarrow, no head, from=1-1, to=3-1]
	\arrow["{\killDiag(\LoopG{1})}"', from=3-1, to=3-3]
	\arrow["{\mathsf{refl}}"', Rightarrow, no head, from=3-3, to=1-3]
\end{tikzcd}\]
which is trivial. Thus, we have shown, in particular, that
$\trunc{\alpha} \smile \trunc{\alpha}$ is trivial, which means that
$\smile$ vanishes everywhere in dimension $1 \times 1$. This concludes
the proof.
\end{proof}
Let us now turn to the $\bZTwo$-cohomology of $K^2$ and \RP. When
working in $\bZTwo$-coefficients, it turns out that we will need to
investigate $\killDiag$ further. In particular, it turns out that
applying $\killDiag$ may give rise to non-trivial homotopies. In
particular, the twist in \RP and $K^2$ will force us to investigate
how much $\killDiag$ commutes with path inversion in the sense that
there exists a loop $q : \Omega^2 K(\bZTwo,2n)$ solving the following
square filling problem
\[\begin{tikzcd}[ampersand replacement=\&]
	{\mathsf{ap}^2_{\smile_k}(p,p)} \&\&\& \refl \\
	{(\mathsf{ap}^2_{\smile_k}(p^{-1},p^{-1}))^{-1}} \&\&\& \refl
	\arrow[from=1-1, to=2-1]
	\arrow["{\ap{(-)^{-1}}{\killDiag(p^{-1})}}"', from=2-1, to=2-4]
	\arrow["{\killDiag(p)}", from=1-1, to=1-4]
	\arrow["q", dashed, from=1-4, to=2-4]
\end{tikzcd}\]
We remark that the only case of interest to us is the case $n = 1$, due to the
vanishing of higher homotopy groups. We will need
the following construction.

\begin{definition}\label{def:res}
  Given two points $x ,y : A$ and a binary function function $F : A
  \to A \to B$, we construct the homotopy $$ \mathsf{Res}_F : (p : x = y) \to (\ap{F(-,y)}{p^{-1}} \cdot \ap{F(x,-)}{p^{-1}})^{-1} = \ap{F(-,x)}{p} \cdot \ap{F(y,-)}{p}$$ by the outer square of the following composite square
\[\begin{tikzcd}[ampersand replacement=\&]
	{F(x,x)} \&\&\&\&\&\& {F(y,y)} \\
	\\
	\& {F(x,x)} \&\&\&\& {F(y,x)} \\
	\\
	\& {F(x,y)} \&\&\&\& {F(y,y)} \\
	\\
	{F(x,x)} \&\&\&\&\&\& {F(y,y)}
	\arrow["{\mathsf{ap}_{F(y,-)}(p^{-1})}"{description}, from=5-6, to=3-6]
	\arrow["\refl"{description}, from=7-1, to=1-1]
	\arrow["\refl"{description}, from=3-2, to=1-1]
	\arrow["{\mathsf{ap}_{F(x,-)}(p^{-1})}"{description}, from=5-2, to=3-2]
	\arrow["{\mathsf{ap}_{F(x,-)}(p^{-1})}"{description}, from=5-2, to=7-1]
	\arrow[""{name=0, anchor=center, inner sep=0}, "{\mathsf{ap}_{F(-,y)}(p)}"', from=5-2, to=5-6]
	\arrow[""{name=1, anchor=center, inner sep=0}, "{\mathsf{ap}_{F(-,x)}(p)}", from=3-2, to=3-6]
	\arrow["{\mathsf{ap}_{F(y,-)}(p)}"{description}, from=3-6, to=1-7]
	\arrow["{\mathsf{ap}_{F(-,x)}(p) \cdot \mathsf{ap}_{F(y,-)}(p)}"{description}, from=1-1, to=1-7]
	\arrow["{(\mathsf{ap}_{F(-,y)}(p^{-1}) \cdot \mathsf{ap}_{F(x,-)}(p^{-1}))^{-1}}"{description}, from=7-1, to=7-7]
	\arrow["\refl"{description}, from=7-7, to=1-7]
	\arrow["\refl"{description}, from=5-6, to=7-7]
	\arrow["{\mathsf{ap}_{a \mapsto \mathsf{ap}_{F(-,a)}(p)}(p^{-1})}"{description}, shorten <=9pt, shorten >=9pt, Rightarrow, from=0, to=1]
\end{tikzcd}\]
where all the outermost squares are given by their obvious fillers.
\end{definition}
The following lemma is immediate by path induction.
\begin{lemma}
  For any point $x: A$ and $F : A \to A \to B$, we have $\mathsf{Res}_F(\refl_{x}) = \refl$.
\end{lemma}
The usefulness of $\mathsf{Res}$ is that it measures how much
$\mathsf{ap}^2_F\textsf{-funct}$ fails commute with path inversion, as
made clear by the following lemma.
\begin{lemma}\label{lem:res}
  Given a binary function $F : A \to A \to B$ and a path $p : x =_A y$, we have a filler
\[\begin{tikzcd}[ampersand replacement=\&]
	{\mathsf{ap}^2_F(p,p)} \&\&\&\& {\ap{F(-,\star_A)}{p}\cdot \ap{F(\star_A,-)}{p}} \\
	{(\mathsf{ap}^2_F(p^{-1},p^{-1}))^{-1}} \&\&\&\& {(\ap{F(-,\star_A)}{p^{-1}}\cdot \ap{F(\star_A,-)}{p^{-1}})^{-1}}
	\arrow[from=2-1, to=1-1]
	\arrow["{\mathsf{Res}_F(p)}"', from=2-5, to=1-5]
	\arrow["{\apfunct{p}{p}}", from=1-1, to=1-5]
	\arrow["{\ap{(-)^{-1}}{\apfunct{p^{-1}}{p^{-1}}}}"', from=2-1, to=2-5]
\end{tikzcd}\]
where the left-most path is the obvious coherence (defined by path induction in Book HoTT and simply by $\refl$ in cubical type theory).
\end{lemma}
\begin{proof}
  Path induction on $p$.
\end{proof}

What is rather surprising is that $\mathsf{Res}_F(p)$ is not just a
mere coherence path: it may be homotopically non-trivial when $p$ is a
loop. In particular, when $p$ is $\LoopG{1} : \Omega K(G,1)$ and $F$ is
the cup product $\smile_k : K(\bZTwo,1) \to K(\bZTwo,1) \to
K(\bZTwo,2)$, the inner square in Definition \ref{def:res} is, modulo
cancellation laws, precisely the two-cell $\sigma^2(1) :
\Omega^2 K(\bZTwo,2)$ (here $\sigma^2$ is shorthand for $\mathsf{ap}_{\sigma_1}\circ \sigma_0 : K(\bZTwo,0) \to \Omega^2 K(\bZTwo,2)$). This follows directly by construction of the cup
product. Hence we arrive at the following lemma, which is key to analysing  the behaviour of the cup product on $\RP$ and $K^2$.
\begin{lemma}\label{lem:killDiag-comm}
  We have a filler of the following square
\[\begin{tikzcd}[ampersand replacement=\&]
	{\mathsf{ap}^2_{\smile_k}(\LoopG{1},\LoopG{1})} \&\&\& \refl \\
	{(\mathsf{ap}^2_{\smile_k}((\LoopG{1})^{-1},(\LoopG{1})^{-1}))^{-1}} \&\&\& \refl
	\arrow[from=1-1, to=2-1]
	\arrow["{\ap{(-)^{-1}}{\killDiag({\LoopG{1}}^{-1})}}"', from=2-1, to=2-4]
	\arrow["{\killDiag(\LoopG{1})}", from=1-1, to=1-4]
	\arrow["{\sigma^2(1)}", from=1-4, to=2-4]
\end{tikzcd}\]
\end{lemma}
We may now characterise the cup product on the cohomology groups of
$\RP$ and $K^2$ with $\mathbb{Z}/2\mathbb{Z}$ coefficients. In fact,
Lemma \ref{lem:killDiag-comm} is the only technical lemma we will
need, and we do not need to rely on any other advanced
theorems/constructions. Let us start with $\RP$.
\begin{proposition}\label{prop:RP-cup}
  The following diagram commutes
\[\begin{tikzcd}[ampersand replacement=\&]
	{\bZTwo \times \bZTwo} \& {\bZTwo} \\
	{H^1(\RP,\bZTwo) \times H^1(\RP,\bZTwo)} \& {H^2(\RP,\bZTwo)}
	\arrow["\smile", from=2-1, to=2-2]
	\arrow["\sim", sloped, from=1-1, to=2-1]
	\arrow["{\mathsf{mult}}", from=1-1, to=1-2]
	\arrow["\sim", sloped, from=1-2, to=2-2]
\end{tikzcd}\]
\end{proposition}
\begin{proof}
  It is enough to show that $\trunc{\alpha} \smile \trunc{\alpha} =
  \trunc{\beta}$ for generators $\alpha : \RP \to K(\bZTwo,1)$ and
  $\beta : \RP \to K(\bZTwo,2)$. We define $\alpha$ by
  \begin{align*}
    \alpha(\star) &:= 0_k \\
    \ap{\alpha}{\ell} &:= \LoopG{1} \\
    \ap{\mathsf{ap}_{\alpha}}{\mathsf{sq}} &:= Q(\LoopG{1})
  \end{align*}
  where $Q : (p : \Omega\,K(\bZTwo,1) \to p = p^{-1}$ is given
  by the $2$-torsion in $K(\bZTwo,n)$ (its explicit construction is
  irrelevant for our purposes). We define $\beta$ by
  \begin{align*}
    \alpha(\star) &:= 0_k \\
    \ap{\beta}{\ell} &:= \refl \\
    \ap{\mathsf{ap}_{\beta}}{\mathsf{sq}} &:= \sigma^2(1)
  \end{align*}
  where $\sigma^2 : \bZTwo \xrightarrow{\sim} \Omega^2 K(\bZTwo,2)$. The fact
  that $\alpha$ and $\beta$ define generators on cohomology is
  straightforward by the constructions in the proofs of Propositions
  \ref{prop:H1-RP2} and \ref{prop:H2-RP2}. We now provide a homotopy
  $h : (x : \RP) \to \alpha(x) \smile_k\alpha(x) = \beta(x)$. We
  define
  \begin{align*}
    h(\star) &:= \refl\\
    \ap{h}{\ell} &:= \killDiag(\LoopG{1})
  \end{align*}
  For the action of $h$ on $\mathsf{sq}$, we need to provide a filler
  of the following square of paths.
\[\begin{tikzcd}[ampersand replacement=\&]
	{\mathsf{refl}} \&\&\&\& {\mathsf{refl}} \\
	\\
	\\
	{\mathsf{ap}_{\smile_k}^2(\LoopG{1},\LoopG{1})} \&\&\&\& {\mathsf{ap}^2_{\smile_k}(\LoopG{1},\LoopG{1})^{-1}}
	\arrow["{\sigma^{2}(1)}", from=1-1, to=1-5]
	\arrow["{\killDiag(\LoopG{1})}", from=4-1, to=1-1]
	\arrow["{\mathsf{ap}_{\mathsf{ap}_{\smile_k \circ \Delta}}(Q(\LoopG{1}))}"', from=4-1, to=4-5]
	\arrow["{\mathsf{ap}_{(-)^{-1}}(\killDiag(\LoopG{1}))}"', from=4-5, to=1-5]
\end{tikzcd}\]
The bottom path acts on the right-hand path by moving in a path inversion around $\LoopG{1}$. Hence, our new square filling problem is
\[\begin{tikzcd}[ampersand replacement=\&]
	{\mathsf{refl}} \&\&\&\& {\mathsf{refl}} \\
	\\
	\\
	{\mathsf{ap}_{\smile_k}^2(\LoopG{1},\LoopG{1})} \&\&\&\& {\mathsf{ap}^2_{\smile_k}((\LoopG{1})^{-1},(\LoopG{1})^{-1})^{-1}}
	\arrow["{\sigma^{2}(1)}", from=1-1, to=1-5]
	\arrow["{\killDiag(\LoopG{1})}", from=4-1, to=1-1]
	\arrow[from=4-1, to=4-5]
	\arrow["{\mathsf{ap}_{(-)^{-1}}(\killDiag((\LoopG{1})^{-1}))}"', from=4-5, to=1-5]
\end{tikzcd}\]
But this is precisely Lemma \ref{lem:killDiag-comm}, and we are done.
\end{proof}

As an immediate corollary, we get the cohomology ring of $\RP$ with
$\bZTwo$-coefficients.

\begin{proposition}
  $H^*(\RP,\bZTwo) \cong \bZTwo[x]/(x^3)$.
\end{proposition}
\begin{proof}
  We simply let elements in $H^n(\RP,\bZTwo)$ correspond to degree-$n$
  elements in $H^*(\RP,\bZTwo)$. By Proposition \ref{prop:RP-cup}, we
  know that the multiplicative structure on $H^*(\RP,\bZTwo)$ is
  regular $\bZTwo$-multiplication, which proves statement.
\end{proof}

Let us now turn to the Klein bottle. Although the cohomology ring
structure is different, we may still make use of the above. We start
by considering its multiplicative structure. First, we construct the
generators $\alpha ,\beta : K^2 \to K(\bZTwo,1)$ by\newline
\begin{minipage}{.5\textwidth}
\begin{align*}
  \alpha(\star) &:= 0_k \\
  \ap{\alpha}{\ell_1} &:= \LoopG{1} \\
  \ap{\alpha}{\ell_2} &:= \refl \\
  \ap{\mathsf{ap}_\alpha}{\mathsf{sq}} &:= \mathsf{flip}(\refl_{\LoopG{1}})
\end{align*}
\end{minipage}
\begin{minipage}{.5\textwidth}
\begin{align*}
  \beta(\star) &:= 0_k \\
  \ap{\beta}{\ell_1} &:= \refl \\
  \ap{\beta}{\ell_2} &:= \LoopG{1} \\
  \ap{\mathsf{ap}_\beta}{\mathsf{sq}} &:= \refl_{\LoopG{1}}
\end{align*}
\end{minipage}
It is immediate by construction that the isomorphism
$H^1(K^2,\bZTwo) \cong \bZTwo \times \bZTwo$ takes $\trunc{\alpha}$ to
$(1,0)$ and $\trunc{\beta}$ to $(0,1)$. The single generator of
$H^2(K^2,\bZTwo)$ is given by $\gamma : K^2 \to K(\bZTwo,2)$:
\begin{align*}
  \gamma(\star) &:= 0_k \\
  \ap{\gamma}{\ell_1} &:= \refl \\
  \ap{\gamma}{\ell_2} &:= \refl \\
  \ap{\mathsf{ap}_\gamma}{\mathsf{sq}} &:= \sigma^2(1)
\end{align*}
It again follows by construction that the isomorphism
$H^2(K^2,\bZTwo) \cong \bZTwo$ takes $\trunc{\gamma}$ to $1$.
\begin{proposition}\label{prop:H*-K2-rels}
  We have the following relations in $H^*(K^2,\bZTwo)$
  \begin{align*}
    \trunc{\alpha}^2 = \gamma \hspace{3cm} \trunc{\alpha} \smile \trunc{\beta} = \gamma \hspace{3cm} \trunc{\beta}^2 = 0_h
  \end{align*}
\end{proposition}
\begin{proof}
  Let us deal with the identities from right to left, in increasing
  difficulty. Showing that $\beta^2 = 0_h$ is easy. First, we construct a
  homotopy $f : (x : K^2) \to \beta(x) \smile_k \beta(x) = 0_k$ by
  induction on $x$.
  \begin{align*}
    f(\star) &:= \refl \\
    \ap{f}{\ell_1} &:= \refl \\
    \ap{f}{\ell_2} &:= \killDiag(\LoopG{1})
  \end{align*}
  For the action on $\mathsf{sq}$, we are reduced to simply providing
  a path $\killDiag(\LoopG{1}) = \killDiag(\LoopG{1})$, which we can by
  $\refl$. This gives us, in particular, that $\trunc{\beta}^2 = 0_h$.

  Let us now provide a homotopy $g : (x : K^2) \to \alpha(x) \smile_k
  \beta(x) = \gamma(x)$. We again proceed by $K^2$-induction.
  \begin{align*}
    g(\star) &:= \refl \\
    \ap{g}{\ell_1} &:= \refl \\
    \ap{g}{\ell_2} &:= \refl
  \end{align*}
The reason that $\refl$ works for the action on $\ell_1$ and $\ell_2$
is because $\ap{0_k \smile_k (-)}{\LoopG{1}}$ and $\ap{(-) \smile_k
  0_k}{\LoopG{1}}$ by definition reduce to $\refl_{0_k}$. The final step amount to showing that
\begin{align*}
  \ap{x \mapsto \ap{y \mapsto \alpha(x) \smile_k
      \beta(y)}{\LoopG{1}}}{\LoopG{1}} = \sigma^2(1)
\end{align*}
which is easy to see by unfolding the definition of $\smile_k$. Thus,
we have $\trunc{\alpha} \smile \trunc{\beta} = \trunc{\gamma}$.

Finally, let us verify that $\trunc{\alpha}^2 = \trunc{\gamma}$. We
construct a homotopy $h : (x : K^2) \to \alpha(x) \smile_k \alpha(x) =
\gamma(x)$ by
  \begin{align*}
    h(\star) &:= \refl \\
    \ap{h}{\ell_1} &:= \killDiag(\LoopG{1}) \\
    \ap{h}{\ell_2} &:= \refl
  \end{align*}
  For the action on $\mathsf{sq}$, we end up having to fill another
  square (or rather: another degenerate cube); fortunately for us,
  this filling problem can easily be reduced to that of filling the
  first square appearing in the proof of
  Proposition~\ref{prop:RP-cup}, which we have already filled. Thus
  $\trunc{\alpha}^2 = \trunc{\gamma}$.
\end{proof}
\begin{corollary} 
  $H^*(K^2,\bZTwo) \cong \bZTwo[x,y]/(x^3,y^2,xy+x^2)$
\end{corollary}
\begin{proof}
  By Proposition \ref{prop:H*-K2-rels}, we obtain the isomorphism by
  mapping the generators $1 : H^0(K^2,\bZTwo)$,
  $\alpha,\beta : H^1(K^2,\bZTwo)$ and $\gamma : H^2(K^2,\bZTwo)$ to
  $1$, $x$, $y$ and $x^2$ respectively.
\end{proof}

\subsection{The complex projective plane}
\label{sec:cp2}

Sometimes, providing direct cohomology computations is rather
difficult. One example of a space whose cohomology is difficult to
compute directly, but is straightforward to compute using the Gysin
sequence is $\mathbb{C}P^2$, i.e.\ the cofibre of the Hopf
map $h : \mathbb{S}^{3} \to \mathbb{S}^{2}$, the details of which will
not be needed here (see \citep[Definition 5]{LICS23} for a
self-contained definition or \citepalias[Section 8.5]{HoTT13} for the
original definition in HoTT). This was first done in HoTT by
\citet{Brunerie16} who applied the Gysin sequence to the fibre
sequence
\begin{align}
  \label{eq:CP2seq}
  \mathbb{S}^1 \to \mathbb{S}^{5}\to \mathbb{C}P^2
\end{align}
It follows from the Mayer-Vietoris sequence that, for any group $G$,
the cohomology of $\mathbb{C}P^2$ is concentrated in degrees $0$, $2$
and $4$:
\begin{align*}
  H^n(\mathbb{C}P^2,G) \cong
  \begin{cases}
    G & \text{if $n \in \{0, 2, 4 \}$} \\
    \mathbbm{1} & \text{otherwise}
    \end{cases}
\end{align*}
The space $\mathbb{C}P^2$ is $1$-connected and so we may apply the
Gysin sequence to \eqref{eq:CP2seq} with $n = 2$. Unfolding the
sequence for $i = 2$ and $i = 4$ tells us that
$e : H^2(\mathbb{C}P^2,\mathbb{Z})$ and
$e^2 : H^4(\mathbb{C}P^2,\mathbb{Z})$ are generators, thus giving us
$H^*(\mathbb{C}P^2,\mathbb{Z}) \cong \mathbb{Z}[x]/(x^3)$. This was
formalised in full detail by~\citet{CohomRings23}.


\subsection{Infinite real projective space}
\label{sec:rpinf}

To further showcase the usefulness of the Gysin sequence, let us
compute $H^*(\RPinf,\bZ/2\bZ)$. This computation is particularly
interesting since it is an example of an application of the Gysin
sequence to a space which is \emph{not} $1$-connected. For it to work,
we hence need to construct the family $c_b$ in
\hyperref[gysin-cond]{($\ast$)}
explicitly. As $\RPinf$ is well-known
to be an Eilenberg-MacLane space we may conveniently define it by
$\RPinf := \KTwoOne$. We remark that this differs from the definition given in~\citep[Definition IV.1]{RPn}, but the argument we provide here should be easy to translate to fit their definition.

Hence, we are to compute $H^*(\KTwoOne,\bZ/2\bZ)$. Using the
equivalence $\psi : \Omega \KTwoOne \simeq_\star \mathbb{S}^0$, we
construct the fibre sequence
\[\mathbb{S}^0 \to \Sigma_{x : \KTwoOne} (0_k = x) \to \KTwoOne\]
This fits the Gysin sequence with $n = 1$ and the fibration $P :
\KTwoOne \to \mathsf{Type}$ being defined by $P(x) := (0_k = x)$. In
order to apply the sequence, we need to construct, for each $x :
\KTwoOne$, a map
\[c_x : \Sigma (0_k = x) \to_\star K(\bZ/2\bZ,1)\] Let us initially
construct $c_x$ is the case when $x := 0_k$. In this case, the domain of $c_x$ is simply $\Sigma (\Omega(K(\bZTwo,1)))$ and hence the counit
of the adjunction $\Sigma \dashv \Omega$ gives us our map:
\[\epsilon : \Sigma (\Omega(K(\bZTwo,1))) \to_\star K(\bZ/2\bZ,1) \]

We would like to define $c_{0_k} := \epsilon$, but it is not entirely
obvious why this definition would define $c_x$ for \emph{all}
$x : \KTwoOne$---the type
$\Sigma (0_k = 0_k) \to_\star K(\bZ/2\bZ,1)$ is not a proposition but
a set, and hence this is not automatic. The trick is to add some
structure to the type of $c_x$ by instead considering the following type:
\begin{align}
  \sum_{e : \Sigma (0_k = x) \to_\star K(\bZ/2\bZ,1)} (e \neq \mathsf{const}) \label{eq:bigtype}
\end{align}
\begin{lemma}\label{lem:contr-fun-type}
  The type in \eqref{eq:bigtype} is contractible. Furthermore,
  $\epsilon$ is non-constant and thus is the centre of contraction
  when $x$ is $0_k$.
\end{lemma}
\begin{proof}
  Since we are proving a proposition, it suffices to do so when $x$ is
  $0_k$. We get
  \begin{align*}
  (\Sigma (0_k = 0_k) \to_\star \KTwoOne) &\simeq (\Omega (\KTwoOne) \to_\star \Omega (\KTwoOne))\\
  &\simeq (\mathbb{S}^0 \to_\star \mathbb{S}^0) \\
  &\simeq \mathsf{Hom}(\bZ/2\bZ,\bZ/2\bZ)
  \end{align*}
  By construction, this equivalence sends $\epsilon$ to the identity
  on $\bZ/2\bZ$. The predicate $(-) \neq \mathsf{const}$ is easily
  seen to be preserved by this chain of equivalences, and hence we get
  \[\sum_{e : \Sigma (0_k = 0_k) \to_\star K(\bZ/2\bZ,1)} (e \neq \mathsf{const}) \simeq \sum_{\phi : \mathsf{Hom}(\bZ/2\bZ,\bZ/2\bZ)} (\phi \neq \mathsf{const}) \]
  Since the identity is the unique non-constant map in
  $\mathsf{Hom}(\bZ/2\bZ,\bZ/2\bZ)$, we are done.
\end{proof}
Hence, by Lemma~\ref{lem:contr-fun-type}, we construct our family
$c_x : \Sigma(0_k = x) \to_\star \KTwoOne$ by choosing the unique
element of \eqref{eq:bigtype}.  This gives us our class
$e :H^1(\KTwoOne,\bZ/2\bZ)$. In fact, with this rather explicit
construction, it is easy to show that $e = \trunc{\mathsf{id}}$, but
this additional knowledge will not be needed here. We may now
instantiate the Gysin sequence. We get, for $i \geq 1$, exact
sequences
\[\begin{tikzcd}[ampersand replacement=\&]
	{H^{i-1}(E,\bZ/2\bZ)} \& {H^{i-1}(\RPinf,\bZ/2\bZ)} \& {H^i(\RPinf,\bZ/2\bZ)} \& {H^i(E,\bZ/2\bZ)}
	\arrow["{(-)\smile e}", from=1-2, to=1-3]
        \arrow[from=1-1,to=1-2]
	\arrow[from=1-3, to=1-4]
\end{tikzcd}\]
where, recall, $\RPinf := \KTwoOne$ and $E := \sum_{x: \KTwoOne}(0_k = x)$. Since $E$ is contractible, we get that the $(-)\smile e : {H^{i-1}(\RPinf,\bZ/2\bZ)} \to {H^i(\RPinf,\bZ/2\bZ)} $ is an isomorphism. Since $\RPinf$ is connected, we have $H^0(\RPinf,\bZ/2\bZ) \cong \bZ/2\bZ$, and hence we get a chain of isomorphisms
\[\bZ/2\bZ \cong H^0(\RPinf,\bZ/2\bZ) \cong H^1(\RPinf,\bZ/2\bZ) \cong \dots \cong H^n(\RPinf,\bZ/2\bZ)\]
whose composition is simply the map $(-)\smile e^n : \bZ/2\bZ \to H^n(\RPinf,\bZ/2\bZ)$. Hence, we have computed $H^*(\RPinf,\bZ/2\bZ)$.

\begin{proposition}
  \label{prop:RPinf-cohom}
  $H^*(\RPinf,\bZ/2\bZ) \cong \bZ/2\bZ[x]$.
\end{proposition}

\section{Computer computations and synthetic cohomology theory}
\label{sec:computations}

In this section we will discuss various computations which can be
performed using the constructions in the paper when formalised in a
system like \CubicalAgda, where univalence and HITs have computational
content. This section is the only in the paper which cannot be
directly formalised in Book HoTT where univalence and HITs are added
axiomatically and lack computational content.

\subsection{Proofs by computation in Cubical Agda}
\label{subsec:proofsbycomp}

One big advantage of formalising synthetic mathematics in a system
like \CubicalAgda compared to Book HoTT is that many simple routine
results hold definitionally. This sometimes vastly simplifies proofs. Even
more ambitiously, one can hope that entire results can be proved
directly by automatic computation. A classic conjecture of this form
is the computability of the `Brunerie number' \citep{Brunerie16},
i.e.\ the number $n : \bZ$ such that
$\pi_4(\mathbb{S}^3) \cong \bZ/n\bZ$. The number $n$ can, in theory,
be computed by running the function
\[
\mathbbm{1}\xrightarrow{(1,1)}\bZ \times \bZ \xrightarrow{\cong}
\pi_2(\mathbb{S}^2) \times \pi_2(\mathbb{S}^2) \xrightarrow{[-,-]}
\pi_3(\mathbb{S}^2) \xrightarrow{\cong} \bZ
\]
where $[-,-]$ denotes the Whitehead product. Brunerie dedicates the
second half of the thesis to manually proving that the absolute value of
the number returned by this function is $2$. To this end, he has to
introduce cohomology, the Hopf invariant and several other complex
constructions. If this number could instead be computed on a computer,
these manual computations would not be required.  Unfortunately, we
still have not been able to carry out this computation in
\CubicalAgda, or in any other proof assistant where univalence and
HITs have computational content.

However, while the original Brunerie number is still out of reach, a
much simplified version of the number was successfully computed by the
authors in \citep{LICS23}. We refer the interested reader to that
paper for details. Computations of related numbers have also been
reported by~\citet{TomJack}. However, an even simpler example of the
same type of problem can be found when trying to show that the wedge sum of two
circles and a sphere is not homotopy equivalent to the torus. This is
typically done by distinguishing their cohomology rings and the
problem boils down to showing that the cup product vanishes in degrees
$(1,1)$ for $\mathbb{S}^2 \vee \mathbb{S}^1 \vee \mathbb{S}^1$ but not
for $\mathbb{T}^2$. This was done computationally by
\citet[Section~6]{BLM} by running the following function in
\CubicalAgda
\[
\mathbbm{1}\xrightarrow{(1,0),(0,1)}(\bZ \times \bZ) \times (\bZ \times \bZ) \xrightarrow{\cong} H^1(X,\bZ) \times H^1(X,\bZ) \xrightarrow{\smile} H^2(X,\bZ) \xrightarrow{\cong} \bZ
\]
for $X$ the two spaces in question. Simply by noting that the output
was $0$ for one space and $1$ for the other, one can deduce that
$\mathbb{S}^2 \vee \mathbb{S}^1 \vee \mathbb{S}^1 \not\simeq \mathbb{T}^2$.

Another class of results which we have successfully proved this way
is that of statements concerning the fact that some cohomology group is generated by some
particular element $g : H^n(X,G)$. By first showing that $H^n(X,G)$ is
isomorphic to a simple well-known group like $\bZ$ by some isomorphism $\phi$, it
suffices to check that $\phi(g) = \pm 1$ for $g$ to generate
$H^n(X,G)$. However, the applicability of this method is also
sometimes limited by computations being infeasible. Already in
\citep{Brunerie16}, another such number appears:
namely, the output of the following function
\begin{align}\label{mini-brunerie}
\mathbbm{1}\xrightarrow{(1,1)}\bZ \times \bZ \xrightarrow{\cong}
H^2(\CP,\bZ) \times H^2(\CP,\bZ) \xrightarrow{\smile} H^4(\CP,\bZ)
\xrightarrow{\cong} \bZ
\end{align}
Successfully running it should give an output of absolute value $1$
and would provide a proof that the Hopf invariant
$\pi_3(\mathbb{S}^2) \to \bZ$ is an isomorphism. This computation was
attempted in \CubicalAgda by \citet{BLM} without luck. The definition
of this number is arguably less complex than that of the Brunerie
number, so the fact that \CubicalAgda is stuck on it is an indication
that a careful benchmarking of similar examples should be done.
Further similar numbers approximating the original Brunerie number can
be found in \citep{LICS23} and we have, in fact, also in this paper
implicitly defined other similar numbers. Consider, for instance the
function
\[
\mathbbm{1} \xrightarrow{\trunc{\alpha},\trunc{\alpha}} H^1(\RP,\bZTwo) \times H^1(\RP,\bZTwo) \xrightarrow{\smile} H^2(\RP,\bZTwo) \xrightarrow{\cong} \bZTwo
\]
where $\alpha : \RP \to K(\bZTwo,1)$ is defined as in Proposition
\ref{prop:RP-cup}. Computing this, we should get $1$ as output. This
would completely remove the need for Lemma \ref{lem:killDiag-comm} in
the proof of Proposition
\ref{prop:RP-cup}. Similar computations could also be used to verify
Proposition \ref{prop:H*-K2-rels}. We judge these computations to be
more feasible than that of~\eqref{mini-brunerie} since the degrees of
the cohomology groups involved are lower, although they have thus far not been successfully run in \CubicalAgda. 
\subsection{Benchmarks}

In order to find potential bottlenecks and get a better idea of which computations succeed and which fail, we have run several benchmarks in
\CubicalAgda. These all consider the underlying maps
$\phi : H^n(X,G) \to H$ maps of isomorphisms defined for concrete
groups $G$ and $H$, and spaces/types $X$. Using these maps, we have
tried computing $\phi(h)$ for concrete elements $h : H^n(X,G)$. As $H$
is often a simpler group like $\bZ$, $\bZ \times \bZ$ or $\bZ/2\bZ$,
these values are all easily understood. We typically pick $h$ as a
generator of $H^n(X,G)$, or as some more complicated element constructed from generators. Whenever
$H^n(X,G)$ is a cyclic group, we write $g^{(n)}$ for its generator and
when there are multiple generators, we simply write $g^{(n)}_i$ for the
$i$th generator. These generators are all easily constructed and we
refer the interested reader to the \Agda code for the precise
definitions.

The various benchmarks and choices are all given in
Table~\ref{benchmarks}. The first four columns gives $n$, $X$, $G$,
and $H$, s.t.\ $H^n(X,G) \cong H$. Both the $G$ and $H$ columns often
have two rows to indicate the two choices of $G$ which we have
considered. We have focused on $\bZ$ and $\bZTwo$ coefficients as
these are the most interesting from a computational point of view. The
fifth column contains the values of $h$ which we have considered,
separated by spaces. So, for e.g. $H^1(\mathbb{S}^1,G)$, we considered
$3$ values of $h$ for each of the two coefficient groups under consideration. The results can
be found in the final column. Its entries are aligned with the corresponding values of
$h$ which we have tried.

As we expect similar goals like these to appear in future
formalisations, the tests were run on a regular laptop with $1.60$GHz
Intel processor and $16$GB of RAM. The successful computation were
marked with \cmarkgreen and the failed computations, marked with \xmarkred,
were manually terminated after a few minutes. Details and exact timings can
be found at
\url{https://github.com/aljungstrom/SyntheticCohomologyTests/blob/master/CohomologyComputations.agda}.

As expected, more tests work for lower dimensions, but for
$\mathbb{R}P^2$, and the more complex $K^2$ and $\mathbb{R}P^\infty$,
all tests fail even in dimension $1$. This is not as surprising as it
may seem. For $\mathbb{R}P^2$ and $K^2$, $\phi$ attempts to compute
the winding number of a loop in $\Omega K(\bZ,1)$ which is constructed
in terms of the complex proof that $\sigma^{-1} : \Sigma K(\bZ,2) \to
K(\bZ,1)$ is a morphism of H-spaces. Another observation is that the
choice of coefficients does not seem to make much of a
difference. Somewhat more surprising, however, is the fact that
several computations which terminated successfully in~\citep{BLM}
fail to terminate here. It is not uncommon that generality comes at
the cost of efficiency, but it is puzzling that a minor change
of the definition of $K(\bZ,n)$ (from the definition using
$n$-truncated $n$-spheres to the one used here) would make such a difference.

\begin{table}[h!]
  \centering
  \begin{tabular}{ |c|c|c|c| c ||c| }
   \hline
     $n$ & $X$ & $G$ & $H$ & $h$ & Results \\
   \hline \hline

    $1$
  & $\mathbb{S}^1$
  & \begin{tabular}{@{}cc@{}}
      $\bZ$ \\
      $\bZTwo$
    \end{tabular}
  & \begin{tabular}{@{}cc@{}}
      $\bZ$ \\
      $\bZTwo$
    \end{tabular}
  & \begin{tabular}{@{}cc@{}}
      $g^{(1)}$ \hspace{0.26cm} $g^{(1)} +_h g^{(1)}$ \hspace{0.26cm} $-_h\,g^{(1)}$ \\
      $g^{(1)}$ \hspace{0.26cm} $g^{(1)} +_h g^{(1)}$ \hspace{0.26cm} $-_h\,g^{(1)}$
    \end{tabular}
  & \begin{tabular}{@{}cc@{}}
      \cmarkgreen \hspace{0.26cm} \cmarkgreen \hspace{0.26cm} \cmarkgreen \\
      \cmarkgreen \hspace{0.26cm} \cmarkgreen \hspace{0.26cm} \cmarkgreen
    \end{tabular}
    \\ \hline

    $2$
  & $\mathbb{S}^2$
  & \begin{tabular}{@{}cc@{}}
      $\bZ$ \\
      $\bZTwo$
    \end{tabular}
  & \begin{tabular}{@{}cc@{}}
      $\bZ$ \\
      $\bZTwo$
    \end{tabular}
  & \begin{tabular}{@{}cc@{}}
      $g^{(2)}$ \hspace{0.26cm} $g^{(2)} +_h g^{(2)}$ \hspace{0.26cm} $-_h\,g^{(2)}$ \\
      $g^{(2)}$ \hspace{0.26cm} $g^{(2)} +_h g^{(2)}$ \hspace{0.26cm} $-_h\,g^{(2)}$
    \end{tabular}
  & \begin{tabular}{@{}cc@{}}
      \cmarkgreen \hspace{0.26cm} \xmarkred \hspace{0.26cm} \xmarkred \\
      \cmarkgreen \hspace{0.26cm} \xmarkred \hspace{0.26cm} \xmarkred
    \end{tabular}
  \\ \hline

    $3$
  & $\mathbb{S}^3$
  & \begin{tabular}{@{}cc@{}}
      $\bZ$ \\
      $\bZTwo$
    \end{tabular}
  & \begin{tabular}{@{}cc@{}}
      $\bZ$ \\
      $\bZTwo$
    \end{tabular}
  & \begin{tabular}{@{}cc@{}}
      $g^{(3)}$ \hspace{0.26cm} $g^{(3)} +_h g^{(3)}$ \hspace{0.26cm} $-_h\,g^{(3)}$ \\
      $g^{(3)}$ \hspace{0.26cm} $g^{(3)} +_h g^{(3)}$ \hspace{0.26cm} $-_h\,g^{(3)}$
    \end{tabular}
  & \begin{tabular}{@{}cc@{}}
      \xmarkred \hspace{0.26cm} \xmarkred \hspace{0.26cm} \xmarkred \\
      \xmarkred \hspace{0.26cm} \xmarkred \hspace{0.26cm} \xmarkred
    \end{tabular}
  \\ \hline

    $1$
  & $\mathbb{T}^2$
  & \begin{tabular}{@{}cc@{}}
      $\bZ$ \\
      $\bZTwo$
    \end{tabular}
  & \begin{tabular}{@{}cc@{}}
      $\bZ^2$ \\
      $\bZTwo^2$
    \end{tabular}
  & \begin{tabular}{@{}cc@{}}
      $g_1^{(1)} +_h g_2^{(1)}$ \hspace{0.26cm} $g^{(1)}_1 -_h g^{(1)}_2$ \\
      $g_1^{(1)} +_h g_2^{(1)}$ \hspace{0.26cm} $g^{(1)}_1 -_h g^{(1)}_2$
    \end{tabular}
  & \begin{tabular}{@{}cc@{}}
      \cmarkgreen \hspace{0.26cm} \cmarkgreen \\
      \cmarkgreen \hspace{0.26cm} \cmarkgreen
    \end{tabular}
  \\ \hline

    $2$
  & $\mathbb{T}^2$
  & \begin{tabular}{@{}cc@{}}
      $\bZ$ \\
      $\bZTwo$
    \end{tabular}
  & \begin{tabular}{@{}cc@{}}
      $\bZ$ \\
      $\bZTwo$
    \end{tabular}
  & \begin{tabular}{@{}cc@{}}
      $g^{(2)}$ \hspace{0.26cm} $g_1^{(1)} \smile g_2^{(1)}$ \hspace{0.26cm} $\left(g^{(1)}_1 +_h g^{(1)}_1\right) \smile g^{(1)}_2$ \\
      $g^{(2)}$ \hspace{0.26cm} $g_1^{(1)} \smile g_2^{(1)}$ \hspace{0.26cm} $\left(g^{(1)}_1 +_h g^{(1)}_1\right) \smile g^{(1)}_2$
    \end{tabular}
  & \begin{tabular}{@{}cc@{}}
      \xmarkred \hspace{0.26cm} \xmarkred \hspace{0.26cm} \xmarkred \\
      \xmarkred \hspace{0.26cm} \xmarkred \hspace{0.26cm} \xmarkred
    \end{tabular}
  \\ \hline

    $1$
  & $\underset{{i =2,1,1}}{\bigvee} \mathbb{S}^i$
  & \begin{tabular}{@{}cc@{}}
      $\bZ$ \\
      $\bZTwo$
    \end{tabular}
  & \begin{tabular}{@{}cc@{}}
      $\bZ^2$ \\
      $\bZTwo^2$
    \end{tabular}
  & \begin{tabular}{@{}cc@{}}
      $g_1^{(1)} +_h g_2^{(1)}$ \hspace{0.26cm} $g^{(1)}_1 -_h g^{(1)}_2$ \\
      $g_1^{(1)} +_h g_2^{(1)}$ \hspace{0.26cm} $g^{(1)}_1 -_h g^{(1)}_2$
    \end{tabular}
  & \begin{tabular}{@{}cc@{}}
      \cmarkgreen \hspace{0.26cm} \cmarkgreen \\
      \cmarkgreen \hspace{0.26cm} \cmarkgreen
    \end{tabular}
  \\ \hline

    $2$
  & $\underset{{i =2,1,1}}{\bigvee} \mathbb{S}^i$
  & \begin{tabular}{@{}cc@{}}
      $\bZ$ \\
      $\bZTwo$
    \end{tabular}
  & \begin{tabular}{@{}cc@{}}
      $\bZ$ \\
      $\bZTwo$
    \end{tabular}
  & \begin{tabular}{@{}cc@{}}
      $g^{(2)}$ \hspace{0.26cm} $g_1^{(1)} \smile g_2^{(1)}$ \hspace{0.26cm} $\left(g^{(1)}_1 +_h g^{(1)}_1\right) \smile g^{(1)}_2$ \\
      $g^{(2)}$ \hspace{0.26cm} $g_1^{(1)} \smile g_2^{(1)}$ \hspace{0.26cm} $\left(g^{(1)}_1 +_h g^{(1)}_1\right) \smile g^{(1)}_2$
    \end{tabular}
  & \begin{tabular}{@{}cc@{}}
      \cmarkgreen \hspace{0.26cm} \cmarkgreen \hspace{0.26cm} \cmarkgreen \\
     \cmarkgreen \hspace{0.26cm} \cmarkgreen \hspace{0.26cm} \cmarkgreen
    \end{tabular}
    \\ \hline
    $1$
  & $\mathbb{R}P^2$
  & $\bZTwo$
  & $\bZTwo$
  & \begin{tabular}{@{}cc@{}}
      $g^{(1)}$ \hspace{0.26cm} $g^{(1)} +_h g^{(1)}$ \hspace{0.26cm} $-_h\,g^{(1)}$
    \end{tabular}
  & \begin{tabular}{@{}cc@{}}
      \xmarkred \hspace{0.26cm} \xmarkred \hspace{0.26cm} \xmarkred
    \end{tabular}
  \\ \hline

    $2$
  & $\mathbb{R}P^2$
  & \begin{tabular}{@{}cc@{}}
      $\bZ$ \\
      $\bZTwo$
    \end{tabular}
  & \begin{tabular}{@{}cc@{}}
      $\bZTwo$ \\
      $\bZTwo$
    \end{tabular}
  & \begin{tabular}{@{}cc@{}}
      $g^{(2)}$ \hspace{0.26cm} $g^{(2)} +_h g^{(2)}$ \hspace{0.26cm} $-_h\,g^{(2)}$ \\
      $g^{(2)}$ \hspace{0.26cm} $g_1^{(1)} \smile g_1^{(1)}$
    \end{tabular}
  & \begin{tabular}{@{}cc@{}}
      \xmarkred \hspace{0.26cm} \xmarkred \hspace{0.26cm} \xmarkred  \\
      \xmarkred \hspace{0.26cm} \xmarkred
    \end{tabular}
  \\ \hline

    $1$
  & $K^2$
  & \begin{tabular}{@{}cc@{}}
      $\bZ$ \\
      $\bZTwo$
    \end{tabular}
  & \begin{tabular}{@{}cc@{}}
      $\bZ$ \\
      $\bZTwo^2$
    \end{tabular}
  & \begin{tabular}{@{}cc@{}}
      $g^{(1)}$ \hspace{0.26cm} $g^{(1)} +_h g^{(1)}$ \hspace{0.26cm} $-_h\,g^{(1)}$ \\
      $g_1^{(1)} +_h g_2^{(1)}$ \hspace{0.26cm} $g^{(1)}_1 -_h g^{(1)}_2$
    \end{tabular}
  & \begin{tabular}{@{}cc@{}}
      \xmarkred \hspace{0.26cm} \xmarkred \hspace{0.26cm} \xmarkred \\
      \xmarkred \hspace{0.26cm} \xmarkred
    \end{tabular}
  \\ \hline

    $2$
  & $K^2$
  & \begin{tabular}{@{}cc@{}}
      $\bZ$ \\
      $\bZTwo$
    \end{tabular}
  & \begin{tabular}{@{}cc@{}}
      $\bZTwo$ \\
      $\bZTwo$
    \end{tabular}
  & \begin{tabular}{@{}cc@{}}
      $g^{(2)}$ \hspace{0.26cm} $g_1^{(1)} \smile g_1^{(1)}$ \\
      $g^{(2)}$ \hspace{0.26cm} $g_1^{(1)} \smile g_2^{(1)}$ \hspace{0.26cm} $\left(g^{(1)}_1 +_h g^{(1)}_1\right) \smile g^{(1)}_2$
    \end{tabular}
  & \begin{tabular}{@{}cc@{}}
      \xmarkred \hspace{0.26cm} \xmarkred \\
      \xmarkred \hspace{0.26cm} \xmarkred \hspace{0.26cm} \xmarkred
    \end{tabular}
  \\ \hline

    $1$
  & $\RPinf$
  & $\bZTwo$
  & $\bZTwo$
  & \begin{tabular}{@{}cc@{}}
      $g^{(1)}$ \hspace{0.26cm} $g^{(1)} +_h g^{(1)}$ \hspace{0.26cm} $-_h\,g^{(1)}$
    \end{tabular}
  & \begin{tabular}{@{}cc@{}}
      \xmarkred \hspace{0.26cm} \xmarkred \hspace{0.26cm} \xmarkred
    \end{tabular}
  \\ \hline

\hline
\end{tabular}
  \caption{Benchmarks}
  \label{benchmarks}
\end{table}

\section{More related and future work}

In this paper, we have described a synthetic approach to cohomology
theory in HoTT, with the possibility of doing direct computations in
\CubicalAgda. This works builds on lots of prior work which we have
discussed throughout the paper, but there is also additional
related work on which this work is not directly based. In this final
section we will discuss this work, as well as possible future
directions.

First of all, there is some related prior work already formalised in
\CubicalAgda. \citet{Zesen} formalised $K(G,1)$ as a HIT, following
\citet{LicataFinster14}, and proved that it satisfies
$\pi_1(K(G,1)) \cong G$.  \citet{Alfieri19} and \citet{Harington20}
formalised $K(G,1)$ as the classifying space $BG$ using $G$-torsors.
Using this, $H^1(\mathbb{S}^1,\bZ) \cong \bZ$ was proved---however,
computing using the maps in this definition proved to be infeasible.

\defcitealias{spectralsequences}{\systemname{Spectral Sequence Project}}

Certified computations of homology groups using proof assistants have
been considered prior to HoTT/UF. For instance, the \Coq
system~\citep{Coq} has been used to compute homology \citep{ctic12}
and persistent homology \citep{pershomology} with coefficients in a
field. This was later extended to homology with $\bZ$-coefficients by
\citet{edr}. The approach in these papers was entirely algebraic and
spaces were represented as simplicial complexes.  Other formalisations
of various classical (co)homology theories can be found in \Lean's
\mathlib \citep{mathlib2020}, but this work is also not synthetic.
However, an earlier HoTT formalisation of synthetic (co)homology
theory in \Lean~2 can be found in \citep{spectralsequences} which was
developed as part of \citep{FlorisPhd}.

An interesting approach to increase the efficiency of synthetic cohomology
computations would be to develop classical computational approaches to
cohomology synthetically. This was done for cellular cohomology by
\citet{BuchholtzFavonia18} who showed that a cohomology theory akin to the one
considered in this paper can also be computed using synthetic cellular
cohomology for spaces which are CW complexes. In ongoing work with Loïc Pujet, we are formalising
cellular (co)homology in \CubicalAgda with the aim of reducing
(co)homology computations to linear algebra, similar to the certified
computations of homology groups in \Coq mentioned above (which relied
on certified Gaussian elimination and Smith normal form computations, as is customary in traditional computational topology).
One envisioned application of this is in the formalisation of the
recent synthetic proof of Serre's classification theorem for homotopy
groups of spheres by Barton and Campion \citep{SerreFiniteness}.

A synthetic approach to homology in HoTT/UF was developed informally
by \citet{graham18} using stable homotopy groups. This was later
extended with a proof of the Hurewicz theorem by
\citet{christensen2020hurewicz}. There has also been some recent work
on synthetic definitions of other classical tools in homological
algebra, in particular Ext groups \citep{christensen2023ext}. Various
results related to synthetic homology theory, including a
formalisation of the Eilenberg-Steenrod axioms for homology were
formalised as part of \citep{spectralsequences}. It would be
interesting to also formalise this in \CubicalAgda and try to compute
also homology groups synthetically.

The definition of $\smile$ in the setting of $\bZ$-cohomlogy in HoTT/UF is due to
\citet[Chapter~5.1]{Brunerie16}. This definition, however, relies on the
smash product which has proved very complex to reason about
formally~\citep{brunerie18}.
Despite this, \citet{Baumann18}
generalised this to $H^n(X,G)$ and managed to formalise graded
commutativity in \HoTTAgda. Baumann's formal proof of this property
is $\sim 5000$ LOC while our formalisation is just $\sim 1200$ LOC. This
indicates that it would be infeasible to formalise other algebraic
properties of $H^*(X,R)$ with this definition. Associativity seems
particularly infeasible, but with our definition the formal proof is
only $\sim 600$ LOC.
In fact, recent work by \citet{AxelSmash} on the symmetric
monoidal structure of the smash product fills in the gaps in \citep{Brunerie16}, thus making
Brunerie's arguments, in theory, formalisable. Nevertheless, the
definition we have presented here still appears to us to be more
convenient to work with.



\end{document}